%% file: paper.tex
\documentclass{article}
\usepackage{fancyhdr}
\usepackage[american]{babel}

\usepackage{amssymb, amsmath, amsthm, amsfonts, enumerate}
\usepackage{amsmath,setspace,scalefnt}
\usepackage[dvipsnames,svgnames,table]{xcolor}
\usepackage{graphicx,tikz,caption,subcaption}
\usetikzlibrary{patterns}
\usetikzlibrary{shapes}
\usepackage{fullpage}
\usepackage{hyperref}
\usepackage{enumitem,verbatim}%,itemize
\usepackage{array,booktabs,longtable}
\usepackage{multicol}
\usepackage{float}
\usepackage{aliascnt}
\usepackage{cleveref}
\usepackage[normalem]{ulem}
\usepackage{tikz-cd}
\usepackage{xparse}
\usepackage{authblk}

\usepackage[margin=2.5cm]{geometry}

\AtBeginDocument{%
  \addtolength{\abovedisplayskip}{-1pt}%
  \addtolength{\belowdisplayskip}{-1pt}%
  \addtolength{\belowdisplayshortskip}{-1pt}%
}

\definecolor{gray}{rgb}{0.25, 0.25, 0.25}

\newtheorem{theorem}{Theorem}[section]
\newaliascnt{lemma}{theorem}
\newtheorem{lemma}[lemma]{Lemma}
\aliascntresetthe{lemma}
\newaliascnt{cor}{theorem}

\aliascntresetthe{cor}
\newaliascnt{conj}{theorem}
\newtheorem{conj}[conj]{Conjecture}
\aliascntresetthe{conj}
\newaliascnt{question}{theorem}

\aliascntresetthe{question}
\newtheorem*{theorem*}{Theorem}

\theoremstyle{definition}

\DeclareTextCompositeCommand{\v}{OT1}{l}{l\nobreak\hspace{-.1em}'} %for the name of Dan Kral

\newcommand{\eps}{\varepsilon}
\newcommand{\pc}{p_{\mathrm c}}
\newcommand{\W}{\mathcal W}

\newcommand{\dd}{\,d}
\newcommand{\EE}{\mathbb E}
\newcommand{\RR}{\mathbb R}
\newcommand{\1}{\mathbf 1}
\newcommand{\cut}{\square}
\newcommand{\dist}{\operatorname{dist}}

\newcommand{\sm}{s_{\mathrm M}}

\newcounter{propcounter}

\theoremstyle{plain}
\newaliascnt{claim}{theorem}

\aliascntresetthe{claim}

\newaliascnt{op}{theorem}

\aliascntresetthe{op}
\newaliascnt{prob}{theorem}

\aliascntresetthe{prob}
\newaliascnt{prop}{theorem}

\aliascntresetthe{prop}
\newaliascnt{proposition}{theorem}

\aliascntresetthe{proposition}
\newaliascnt{corollary}{theorem}
\newtheorem{corollary}[corollary]{Corollary}
\aliascntresetthe{corollary}
\theoremstyle{definition}
\newaliascnt{rmk}{theorem}

\aliascntresetthe{rmk}
\newaliascnt{remark}{theorem}
\newtheorem{remark}[remark]{Remark}
\aliascntresetthe{remark}
\theoremstyle{definition}
\newaliascnt{intp}{theorem}

\aliascntresetthe{intp}
\theoremstyle{definition}
\newaliascnt{kn}{theorem}

\aliascntresetthe{kn}
\theoremstyle{definition}
\newaliascnt{defn}{theorem}

\aliascntresetthe{defn}
\theoremstyle{definition}
\newaliascnt{expl}{theorem}

\aliascntresetthe{expl}
\theoremstyle{definition}
\newaliascnt{obs}{theorem}

\aliascntresetthe{obs}
\crefname{theorem}{Theorem}{Theorems}
\Crefname{theorem}{Theorem}{Theorems}
\crefname{lemma}{Lemma}{Lemmas}
\Crefname{lemma}{Lemma}{Lemmas}
\crefname{cor}{Corollary}{Corollaries}
\Crefname{cor}{Corollary}{Corollaries}
\crefname{corollary}{Corollary}{Corollaries}
\Crefname{corollary}{Corollary}{Corollaries}
\crefname{conj}{Conjecture}{Conjectures}
\Crefname{conj}{Conjecture}{Conjectures}
\crefname{question}{Question}{Questions}
\Crefname{question}{Question}{Questions}
\crefname{claim}{Claim}{Claims}
\Crefname{claim}{Claim}{Claims}
\crefname{op}{Open Question}{Open Questions}
\Crefname{op}{Open Question}{Open Questions}
\crefname{prob}{Problem}{Problems}
\Crefname{prob}{Problem}{Problems}
\crefname{prop}{Proposition}{Propositions}
\Crefname{prop}{Proposition}{Propositions}
\crefname{proposition}{Proposition}{Propositions}
\Crefname{proposition}{Proposition}{Propositions}
\crefname{rmk}{Remark}{Remarks}
\Crefname{rmk}{Remark}{Remarks}
\crefname{remark}{Remark}{Remarks}
\Crefname{remark}{Remark}{Remarks}
\crefname{intp}{Interpretation}{Interpretations}
\Crefname{intp}{Interpretation}{Interpretations}
\crefname{kn}{Known}{Knowns}
\Crefname{kn}{Known}{Knowns}
\crefname{defn}{Definition}{Definitions}
\Crefname{defn}{Definition}{Definitions}
\crefname{expl}{Example}{Examples}
\Crefname{expl}{Example}{Examples}
\crefname{obs}{Observation}{Observations}
\Crefname{obs}{Observation}{Observations}
\crefname{subclaim}{Subclaim}{Subclaims}
\Crefname{subclaim}{Subclaim}{Subclaims}

\NewDocumentEnvironment{shl}{}
  {%
    \begingroup
    \color{blue}%
    \bfseries
    \boldmath
    [SL:\space\ignorespaces
  }
  {%
    \unskip]\endgroup\ignorespacesafterend
  }

\title{Bipodal optimizers in the upper-tail variational problem for regular subgraph densities}

\date{}

\author[1,2]{Sangho Lim}
\author[3]{Seonghyuk Im}
\author[1,2]{Taeyoung Kim}
\author[4]{Kyeongsik Nam}
\author[1,2]{Hongseok Yang}
\affil[1]{School of Computing, KAIST, Daejeon, Korea}
\affil[2]{School of Computational Sciences, Korea Institute for Advanced Study (KIAS), Seoul, Korea}
\affil[3]{Center for AI and Natural Sciences, Korea Institute for Advanced Study (KIAS), Seoul, Korea}
\affil[4]{Department of Mathematical Sciences, KAIST, Daejeon, Korea}
\affil[ ]{E-mail addresses: \texttt{lim.sang@kaist.ac.kr},
\texttt{seonghyuk@kias.re.kr},
\texttt{taeyoungkim21@kaist.ac.kr},
\texttt{ksnam@kaist.ac.kr},
\texttt{hongseokyang@kias.re.kr}}

\begin{document}

\maketitle
\begin{abstract}
Let $H$ be a fixed $d$-regular graph with $d\ge2$, and let $t(H,\cdot)$
denote its homomorphism density.  We study the upper-tail event
$t(H,G(n,p))\ge r^{|E(H)|}$ for fixed $0<p<r<1$ in a dense
Erd\H{o}s--R\'enyi random graph $G(n,p)$.
Near the Lubetzky--Zhao replica-symmetric phase boundary and away from the
exceptional target density $(d-1)/d$, we prove that the optimizer of the
Chatterjee--Varadhan variational problem on the symmetry-breaking side is
bipodal (two-block) and unique up to relabeling.
Its block parameters depend analytically on $(p,r)$.
To treat the exceptional boundary point, where the nonexceptional theory
degenerates, we construct an analytic curve approaching that point from the
symmetry-breaking side along which the unique optimizers are nonconstant
rank-one bipodal graphons.
In both settings, we derive asymptotic expansions of the edge-density deficit
and the rate function that governs the exponential decay of the upper-tail probability.
Moreover, the conditioned random graph converges in cut distance
to the corresponding bipodal optimizer as $n\to\infty$.
As $(p,r)$ approaches the phase boundary from the symmetry-breaking side,
the optimizers converge to their constant limits through two distinct mechanisms.
For each fixed nonexceptional target density, one block shrinks to zero measure,
giving convergence in $L^1$ but not in $L^\infty$.
Along the exceptional curve, both blocks remain macroscopic: their sizes
tend to $1/2$ and all three block densities tend to $(d-1)/d$, yielding
convergence in $L^\infty$.
\end{abstract}

\input{sections/intro}
\input{sections/preliminaries}
\input{sections/lz_boundary}
\input{sections/nonexceptional}

\input{sections/singular}

\input{sections/lean}
\input{sections/conclusion}

\paragraph{AI disclosure.}
AI assistance in this work involved GPT Pro 5.5 and 5.6,
GPT Sol 5.6, GPT Astra 6.0, Claude Opus 4.8 and 5.0, and Claude Fable 5.0 and 5.1.
The authors independently reviewed the text and verified every proof,
and retained responsibility for the final writing and organization of the paper.

We used the neural graphon representation of Kim, Baek, Lee, and
Yang~\cite{KimBaekLeeYangGraphons}, together with numerical optimization
algorithms, to explore optimizer structure, including bipodality near the
Lubetzky--Zhao boundary and the rank-one family near the exceptional endpoint.
The authors proposed these experiments, and AI assisted with code implementation and debugging.

For the nonexceptional results in
\Cref{sec:lz-boundary,sec:nonexceptional-endpoint},
the main idea of applying the theorems of Kenyon, Radin, Ren, and
Sadun~\cite{kenyon2014entropy} to establish bipodality in the variational
problem originated with the authors.
Near the exceptional endpoint in \Cref{sec:singular-endpoint},
the authors also proposed constructing the rank-one bipodal family,
motivated by the numerical experiments.
These insights were developed into the proof arguments through human--AI collaboration.

The Lean formalization was produced entirely by Claude Opus 4.8 and 5.0
and Claude Fable 5.0 and 5.1 under the authors' direction.
The authors checked the Lean definitions and theorem statements against
those in the paper and reviewed the external inputs assumed as axioms.
They also verified that the project built without errors.

\paragraph{Acknowledgments.}
We would like to thank Jineon Baek,  Bartlomiej Kielak, Daniel Kr\'a\v{l}, and Joonkyung Lee for helpful discussions.
S. Im was supported by a KIAS individual Grant (AP109501) at Korea Institute for Advanced Study. 
S. Lim, T. Kim, and H. Yang were supported by the National Research Foundation of Korea (NRF)
grant funded by the Korean Government (MSIT) (No.~RS-2023-00279680).

\bibliographystyle{plain}
\bibliography{cites}

\clearpage
\appendix

\input{sections/appendix_preliminaries}
\input{sections/appendix_lz_boundary}
\input{sections/appendix_nonexceptional_endpoint}
\input{sections/appendix_singular_endpoint}

\end{document}

%% file: sections/intro.tex
\section{Introduction}\label{sec:introduction}
The Erd\H{o}s--R\'enyi graph $G(n,p)$ is one of the standard models of random graphs, defined as a graph on $n$ vertices in which each edge is present independently with probability $p$.
For a fixed graph $H$, let $X_H$ be the number of injective homomorphisms from
$H$ to $G(n,p)$.  Thus, writing
$(n)_k=n(n-1)\cdots(n-k+1)$,
\[
  \mathbb{E}[X_H]=(n)_{|V(H)|}p^{|E(H)|}.
\]
Our primary interest is the following two questions raised by Chatterjee and Varadhan~\cite{ChatterjeeVaradhanLDP} in the context of the upper-tail large deviation problem: (1) what is the probability that $X_H$ is at least $(1+\delta)\mathbb{E}[X_H]$ for some fixed $\delta>0$?; (2) what is the typical structure of $G(n,p)$ conditioned on this event?

This upper-tail problem and its variants have been studied extensively; see, for example,
\cite{BhattacharyaDembo2021,ChatterjeeMissingLog,DeMarcoKahnTriangles,
ganguly2024upper,Janson2004uppertail,JansonRucinskiInfamous,Raz2020Cycles,Sileikis2020Stars,Vu2001}.
However, even when $H$ is a triangle, the problem is difficult and not
completely understood.  Existing results take different forms in the sparse
and dense regimes.  For the sparse results below, $\delta>0$ is fixed while
$n\to\infty$ and $p=p(n)\to0$.  Regarding the first question,
Kim and Vu~\cite{KimVu2004} showed that the probability that
$X_{K_3}\ge(1+\delta)\mathbb{E}[X_{K_3}]$ lies between
$\exp(-O_\delta(n^2p^2 \log(1/p)))$ and $\exp(-\Omega_\delta(n^2p^2))$ whenever $p\ge\log n/n$.
Chatterjee~\cite{ChatterjeeMissingLog} and, independently, DeMarco and
Kahn~\cite{DeMarcoKahnTriangles} determined the correct order when $p \ge C_\delta \log n/n$ by showing that 
\[
  \mathbb{P}\bigl(X_{K_3}\ge(1+\delta)\mathbb{E}[X_{K_3}]\bigr)
  =\exp\bigl(-\Theta_\delta(n^2p^2\log(1/p))\bigr).
\]
These results determine the order of the exponent,
but not its leading constant.  Chatterjee
and Dembo~\cite{ChatterjeeDemboNonlinear} subsequently proved a nonlinear
large-deviation principle which, when $p$ does not decay too quickly, reduces
the upper-tail problem to a weighted-graph relative-entropy variational
problem.
Lubetzky and Zhao~\cite{LubetzkyZhaoSparse} solved this variational problem for
triangles and, by combining their solution with the Chatterjee--Dembo
principle, obtained, for $p\ge n^{-1/42}\log n$ with $p\to0$,
\[
  \mathbb{P}\bigl(X_{K_3}\ge(1+\delta)\mathbb{E}[X_{K_3}]\bigr)
  =
  \exp\left[
    -\left(
      \min\left\{\frac{\delta^{2/3}}2,\frac{\delta}{3}\right\}
      +o_\delta(1)
    \right)n^2p^2\log(1/p)
  \right].
\]
This determines the exact leading constant.  It applies, however, over a
narrower range of $p$, since $(\log n)/n=o(n^{-1/42}\log n)$.
More generally, upper-tail bounds and asymptotics have been established
for broader classes of graphs \(H\) and across a range of sparse regimes;
see \cite{Augeri2020,BasakBasu2023,BhattacharyaGangulyLubetzkyZhao,
CookDembo2020,CookDemboPham2024,DeMarcoKahnCliques,Eldan2018,HarelMoussetSamotij2022}.
For the second question,  Cook and Dembo~\cite{CookDembo2024Structure}
showed that, in suitable sparse regimes, graphs conditioned on upper-tail
events typically resemble an Erd\H{o}s--R\'enyi graph with a planted clique
and/or a planted complete bipartite graph.

In this paper, we primarily consider the dense regime, in which $p$ is a fixed constant.
In this regime, it is convenient to frame the upper-tail event using
homomorphism densities.  For a graph $G$ on $n$ vertices, define
\[
  t(H,G):=\frac{\hom(H,G)}{n^{|V(H)|}},
\]
where $\hom(H,G)$ is the number of graph homomorphisms from $H$ to $G$.
Since $H$ is fixed, the number of non-injective homomorphisms is
$O_H(n^{|V(H)|-1})$, and hence, for $G=G(n,p)$,
\[
  t(H,G(n,p))=\frac{X_H}{n^{|V(H)|}}+O_H(n^{-1}).
\]
Thus, the homomorphism density $t(H,G(n,p))$ differs from the normalized
injective count $X_H/n^{|V(H)|}$ by only $O_H(n^{-1})$.  We therefore
use the following homomorphism-density version of the upper-tail event.  For
$p<r<1$, write
\[
  \mathcal E_{H,p,r}:=
  \{t(H,G(n,p))\ge r^{|E(H)|}\}.
\]
The count event in the first paragraph corresponds asymptotically to the
choice $r^{|E(H)|}=(1+\delta)p^{|E(H)|}$. In this paper,
a central role is played by the relative-entropy function
\[
  J_p(z) = z \log \frac{z}{p} + (1-z) \log \frac{1-z}{1-p}.
\]
Throughout the paper, $\log$ denotes the natural logarithm and we use the
convention $0\log0=0$.  The quantity $J_p(z)$ is the large-deviation cost per edge of
raising the edge density from $p$ to $z$.  In this notation, Lubetzky and
Zhao~\cite{lubetzky2015large} showed that if the point $(r^2, J_p(r))$ lies on
the convex minorant of $x \mapsto J_p(x^{1/2})$, then
\[
  \mathbb{P}(\mathcal E_{K_3,p,r})
  =\exp(-J_p(r)n^2/2+o_{p,r}(n^2)),
\]
and they extended this result to general regular graphs $H$.

Conditioned on $\mathcal E_{H,p,r}$, the typical structure of $G(n,p)$ can take one of two forms: either it looks like an Erd\H{o}s--R\'enyi graph with a higher edge density (called the \emph{replica-symmetric phase}), or its edges are clustered (called the \emph{symmetry-breaking phase}).
Chatterjee and Varadhan~\cite{ChatterjeeVaradhanLDP} showed that, depending on the parameters, both structures can occur.
Lubetzky and Zhao~\cite{lubetzky2015large} determined exactly when the replica-symmetric phase occurs and when the symmetry-breaking phase occurs in terms of the function $J_p(z)$ defined above for every $d$-regular graph $H$.
\begin{theorem}[Lubetzky--Zhao criterion, informal version~\cite{lubetzky2015large}]\label{thm:lz-criterion-informal}
  Let $H$ be a $d$-regular graph with $d \geq 2$, and let $0<p<r<1$. If the point $(r^d, J_p(r))$ lies on the convex minorant of the function $x\mapsto J_p(x^{1/d})$, then the typical structure of $G(n,p)$ conditioned on $\mathcal E_{H,p,r}$ is like an Erd\H{o}s--R\'enyi graph with edge density $r$, i.e., $(p, r)$ is in the replica-symmetric phase. Otherwise, the typical structure of $G(n,p)$ conditioned on $\mathcal E_{H,p,r}$ is clustered, i.e., $(p, r)$ is in the symmetry-breaking phase.
\end{theorem}
See \Cref{thm:lz-criterion} for the formal statement.
The result of Lubetzky and Zhao also yields that the probability of $\mathcal E_{H,p,r}$ is $\exp(-J_p(r)n^2/2 + o_{H,p,r}(n^2))$ when $(p, r)$ is in the replica-symmetric phase.
However, in the symmetry-breaking phase, neither the manner in which the
edges cluster nor, in general, the leading exponential decay rate of
$\mathbb P(\mathcal E_{H,p,r})$ is known.
Lubetzky and Zhao~\cite[Section~7]{lubetzky2015large} raised the question of
whether the conditioned random graph always has a bipodal structure in this
phase.  Informally, this means that the vertex set of $G(n,p)$ can be partitioned
into two parts such that the edges within each part and between the two parts
behave like Erd\H{o}s--R\'enyi graphs, with possibly different edge densities.  Numerical experiments using the
neural graphon representation of Kim, Baek, Lee, and
Yang~\cite{KimBaekLeeYangGraphons} provide further evidence for an
affirmative answer to this question; see \Cref{fig:neural-bipodal-patterns}.

Our main contribution is a step toward resolving this bipodality question
and has two complementary parts.  Away from the exceptional density
$r_*=(d-1)/d$, we prove bipodality and uniqueness throughout a nontrivial
neighborhood of the phase boundary on the symmetry-breaking side.  At $r_*$,
where this nondegenerate local theory breaks down, we instead find a family of
balanced bipodal optimizers along an analytic curve approaching the phase boundary.
The Lubetzky--Zhao phase boundary can be written as $p=\pc(r)$,
as we establish in \Cref{thm:scalar-lz-boundary}.
For the nonexceptional case, we express the expansions
in terms of the log-odds distance to this boundary,
\[
  \lambda(p, r) = \log \frac{1-p}{p} - \log \frac{1-\pc(r)}{\pc(r)},
\]
which vanishes exactly on the boundary and is positive on the symmetry-breaking side $p<\pc(r)$.
The exceptional case uses a separate parameter $h$, whose meaning is explained later.
\begin{theorem}[Informal main results]\label{thm:main-results-informal}
  Let $H$ be a $d$-regular graph with $d \geq 2$, and let $\pc(r)$ be the Lubetzky--Zhao boundary curve.
  Set $r_*=(d-1)/d$.
  \begin{enumerate}[label=(\alph*)]
    \item \label{thm:main-results-informal-a} Let $r_0\in(0,1)\setminus\{r_*\}$.  There exists an open
    neighborhood $U$ of $(\pc(r_0),r_0)$ such that, for every
    $(p,r)\in U$ with $p<\pc(r)$, the typical structure of $G(n,p)$
    conditioned on $\mathcal E_{H,p,r}$ is described by a unique nonconstant bipodal graphon,
    up to relabeling.  Its block sizes and edge densities are analytic in $(p,r)$.
    For fixed $r$, the smaller block size tends to zero and the edge density
    within the larger block tends to $r$ as $p\uparrow\pc(r)$.

    Furthermore, for $(p,r)\in U$ with $p<\pc(r)$,
    \[
      \mathbb P(\mathcal E_{H,p,r})
      =\exp\left[-\left(J_p(r)-\frac{\lambda(p,r)^2}{2A_H(r)}
        +O_{H,U}(\lambda(p,r)^3)\right)\frac{n^2}{2}+o_{H,p,r}(n^2)\right],
    \]
    where $A_H:(0,1)\setminus\{r_*\}\to(0,\infty)$ is analytic.
    The $O_{H,U}$-term is uniform for $(p,r)\in U$ as $\lambda(p,r)\downarrow0$,
    while the $o_{H,p,r}$-term is taken as $n\to\infty$ with $H,p,r$ fixed.

    \item \label{thm:main-results-informal-b} There exist $h_0>0$ and an analytic curve
    $h\mapsto(p_h,r_h)$, defined for $0<h<h_0$, of
    symmetry-breaking parameters with $p_h<\pc(r_h)$ approaching the exceptional boundary point:
    \[
      (p_h,r_h)\longrightarrow(\pc(r_*),r_*)
      \qquad(h\downarrow0).
    \]
    Along this curve, the typical structure of $G(n,p_h)$ conditioned on
    $\mathcal E_{H,p_h,r_h}$ is described by a unique nonconstant bipodal graphon
    $W_h$, up to relabeling.  Its block sizes and edge densities are analytic
    in $h$ and tend to $1/2$ and $r_*$, respectively, as $h\downarrow0$.

    Furthermore, for every fixed $h\in(0,h_0)$,
    \[
      \mathbb P(\mathcal E_{H,p_h,r_h})
      =\exp\left[-\left(J_{p_h}(r_h)-\frac{d^3}{3}h^4
        +O_d(h^6)\right)\frac{n^2}{2}+o_{H,h}(n^2)\right].
    \]
    The $O_d$-term is taken as $h\downarrow0$, while the $o_{H,h}$-term
    is taken as $n\to\infty$ with $H,h$ fixed.
  \end{enumerate}
\end{theorem}
The two parts describe qualitatively different mechanisms by which a
nonconstant bipodal optimizer approaches a constant graphon at the phase
boundary.  Away from $r_*$, one pode shrinks to measure zero while the other
occupies almost all of $[0,1]$ and its within-pode density tends to $r$.  The
densities involving the shrinking pode need not tend to $r$, but the regions
on which they occur have vanishing measure.  The optimizer therefore converges
to $W\equiv r$ in $L^1$, and hence in cut distance, but not in
$L^\infty$.  At $r_*$, neither pode disappears: their sizes tend to $1/2$, while
both within-pode densities and the cross density tend to $r_*$.  Thus, the
exceptional optimizer converges in $L^\infty$ by homogenization of two
macroscopic podes, rather than by the disappearance of a small pode.
The precise variational statements underlying parts~(a) and~(b) are
\Cref{thm:nonexceptional-optimizers,thm:endpoint-optimizers}, respectively.

\subsection{Chatterjee--Varadhan variational problem}
We now recast the preceding random-graph questions as a variational problem on
graphons.
A \emph{graphon} is a symmetric measurable function $W:[0,1]^2\to[0,1]$.
We write $\W_0$ for the space of all graphons.
Graphons are limit objects for dense graphs; see~\cite{LovaszGraphLimits}.
For a graphon $W$, define the \emph{edge density}, \emph{homomorphism density}, and \emph{relative entropy} by 
\begin{equation}\label{eq:graphon-quantities}
  \begin{aligned}
  e(W) & =\int_{[0,1]^2}W(x,y)\,dx\,dy,
  \\
  t(H,W) & =\int_{[0,1]^{|V(H)|}}\prod_{ij\in E(H)}W(x_i,x_j)\prod_{i\in V(H)}dx_i,
  \\
  I_p(W) & =\int_{[0,1]^2}J_p(W(x,y))\,dx\,dy.
  \end{aligned}
\end{equation}
For $0<p<r<1$ and a graph $H$, define the Chatterjee--Varadhan variational problem~\cite{ChatterjeeVaradhanLDP} by
\begin{equation}\label{eq:graphon-variational-problem}
    \inf_{W\in\W_0} I_p(W)\quad \text{subject to } t(H,W)\ge r^{|E(H)|}.
\end{equation}
We denote by $\Phi_H(p,r)$ the optimal value of the above variational problem
and by $\W_H(p,r)$ its set of minimizers.  A graphon is \emph{feasible} for
this problem if it satisfies the constraint
$t(H,W)\ge r^{|E(H)|}$; the collection of all such graphons is the feasible
set.  This set is nonempty because the constant graphon $W\equiv r$ is
feasible.  Standard compactness, continuity, and lower-semicontinuity results
for graphons imply that the infimum is attained, so $\W_H(p,r)$ is nonempty~\cite[Theorem~2.7]{lubetzky2015large}.
The cut distance $\delta_\square$ measures the difference between two graphons
after allowing measure-preserving relabelings; its formal definition, together
with the graphon-topological facts used above, is given in
\Cref{sec:graphons}.  A graph $G$ can be viewed as its associated step graphon
$W_G$, in which case $t(H,G)=t(H,W_G)$.
The following theorem makes precise the connection between the variational
problem, the upper-tail probability, and the conditioned random graph.
\begin{theorem}[Chatterjee--Varadhan large-deviation principle~\cite{ChatterjeeVaradhanLDP}]
\label{thm:graphon-large-deviations}
  For every $0<p<r<1$, we have 
  \[
    \lim_{n \to \infty}\frac{1}{\binom{n}{2}}
    \log \mathbb{P}\bigl(t(H,G(n,p))\ge r^{|E(H)|}\bigr)
    =-\Phi_H(p,r).
  \]
  Furthermore, for any $\eps>0$, there exists $C=C(\eps, p, r, H)>0$ such that, for all sufficiently large $n$,
  \[
    \mathbb{P}\left(
      \min_{W \in \W_H(p,r)}\delta_\square(W_{G(n,p)},W)>\eps
      \ \middle|\ t(H,G(n,p))\ge r^{|E(H)|}
    \right)
    \le e^{-Cn^2}.
  \]
\end{theorem}
The first conclusion gives the leading exponential decay rate of the
upper-tail probability, while the second says that, conditional on the
upper-tail event, the graphon of a typical $G(n,p)$ is close in cut distance
to some optimizer in $\W_H(p,r)$, with probability tending exponentially fast
to one.  In variational terms, the Lubetzky--Zhao criterion stated informally in
\Cref{thm:lz-criterion-informal} takes the following form.
\begin{theorem}[Lubetzky--Zhao criterion~\cite{lubetzky2015large}]\label{thm:lz-criterion}
    Let $H$ be a $d$-regular graph with $d\ge2$ and let $0<p<r<1$.
    If the point $(r^d, J_p(r))$ lies on the convex minorant of the function $x\mapsto J_p(x^{1/d})$, then the unique minimizer of~\eqref{eq:graphon-variational-problem} is the constant graphon $W\equiv r$. Otherwise, the constant graphon is not contained in the set of minimizers $\W_H(p, r)$.
\end{theorem}
A graphon $W$ is called \emph{bipodal} if there is a measurable set $A\subseteq[0,1]$ such that, writing $B=[0,1]\setminus A$, the graphon $W$ is almost everywhere equal to a constant on each of the three sets $A\times A$, $(A\times B)\cup(B\times A)$, and $B\times B$; see \Cref{sec:graphons}.
Thus, in the graphon formulation, the question of Lubetzky and Zhao asks
whether every minimizer of~\eqref{eq:graphon-variational-problem} in the
symmetry-breaking phase is bipodal.  The following theorem gives a local
affirmative answer away from the exceptional density, being the variational version of
Theorem~\hyperref[thm:main-results-informal-a]{\ref*{thm:main-results-informal}~\ref*{thm:main-results-informal-a}}.

\begin{theorem}\label{thm:nonexceptional-optimizers}
    Let $H$ be a $d$-regular graph with $d\ge2$.
    If $r_0 \in (0, 1)\setminus\{(d-1)/d\}$, then there exists
    an open neighborhood $U$ of $(\pc(r_0), r_0)$ such that for
    any $(p, r)\in U$ with $p<\pc(r)$, the minimizer of~\eqref{eq:graphon-variational-problem}
    is unique up to relabeling, nonconstant, and bipodal.
    For every fixed $r$ with $(\pc(r),r)\in U$, the size of its smaller
    block tends to zero as $p\uparrow\pc(r)$.
    Furthermore, the unique optimizer $W_{p, r}$ satisfies
      \[
    e(W_{p,r})=r-\frac{\lambda(p,r)}{A_H(r)}+O_{H,U}(\lambda(p,r)^2),
    \qquad
    \Phi_H(p,r)=J_p(r)-\frac{\lambda(p,r)^2}{2A_H(r)}
      +O_{H,U}(\lambda(p,r)^3),
  \]
    where $A_H:(0,1) \setminus\{(d-1)/d\} \to (0, \infty)$ is an analytic function.
\end{theorem}

To explain why the exceptional density requires a separate argument,
consider $H=K_3$, for which $d=2$ and the exceptional density is $1/2$.
Fix $r\ne1/2$, and let $c(p,r)$ denote the size of the smaller
block.  As $p\uparrow\pc(r)$ from the symmetry-breaking side,
\Cref{rmk:bipodal-parameter-expansions} gives the expansion
\[
  c(p,r)
  =\frac{r}{(2r-1)^2A_{K_3}(r)}\lambda(p,r)
   +O_r\bigl(\lambda(p,r)^2\bigr).
\]
Here the $O_r(\lambda(p,r)^2)$ term is a remainder whose absolute value is
at most $C_r\lambda(p,r)^2$ for $p$ sufficiently close to $\pc(r)$.
For this fixed $r$, the coefficient of $\lambda(p,r)$ is finite and
$\lambda(p,r)\to0$, so the expansion implies $c(p,r)\to0$: the smaller
block disappears.
The difficulty arises when $r$ also tends to $1/2$ as the phase boundary is
approached.  The factor $(2r-1)^2$ in the denominator then tends to zero,
and our estimates for $A_{K_3}(r)$ do not determine whether the full
coefficient of $\lambda(p,r)$ remains bounded.
The error constant $C_r$ and the required closeness of $p$ to $\pc(r)$ may
also depend on $r$.  Thus, knowing that $\lambda(p,r)\to0$ is no longer
enough to conclude from this expansion that $c(p,r)\to0$.
The preceding analysis therefore does not settle whether a block still
disappears along approaches to the exceptional boundary point.

The following result gives the variational version of
Theorem~\hyperref[thm:main-results-informal-b]{\ref*{thm:main-results-informal}~\ref*{thm:main-results-informal-b}},
in terms of asymptotically balanced rank-one graphons.  Here, a graphon $W$ is called
\emph{rank-one} if $W(x,y)=f(x)f(y)$ almost everywhere for some measurable
$f:[0,1]\to[0,1]$.  The result still supports an affirmative answer to the
bipodality question by producing symmetry-breaking points arbitrarily close to the boundary
whose optimizers are bipodal, but with different limiting behavior.
Here we use $h$, half the difference between the two values of the rank-one
factor $f$, rather than $\lambda$ as the expansion parameter.
Unlike $A_H(r)$ above, the leading coefficients in the expansions below
depend only on $d$, although the admissible range $0<h<h_0$ may depend on $H$.
\begin{theorem}\label{thm:endpoint-optimizers}
Let $H$ be a $d$-regular graph with $d\ge2$.  There exist $h_0>0$, an
analytic curve $h\mapsto(p_h,r_h)$ for $0<h<h_0$, and nonconstant rank-one
bipodal graphons $W_h$ such that $p_h<\pc(r_h)$ and
\[
  (p_h,r_h)\longrightarrow
  \left(\pc\left(\frac{d-1}{d}\right),\frac{d-1}{d}\right)
  \qquad
  (h\downarrow0).
\]
The graphon $W_h$ satisfies $t(H,W_h)=r_h^{|E(H)|}$
and is the unique minimizer of \eqref{eq:graphon-variational-problem}, up to
relabeling, at $(p_h,r_h)$.  As $h\downarrow0$,
the sizes of its two blocks both converge to $1/2$, while its two within-block densities
and its cross density all tend to $(d-1)/d$.  Hence, $W_h$
converges in $L^\infty$ to the constant graphon $W\equiv(d-1)/d$.
Moreover,
\[
  e(W_h)=r_h-(d-1)h^2+O_d(h^4),
  \qquad
  \Phi_H(p_h,r_h)=J_{p_h}(r_h)-\frac{d^3}{3}h^4+O_d(h^6).
\]
\end{theorem}

\subsection{Proof strategy}\label{sec:proof-strategy}
The proofs of \Cref{thm:nonexceptional-optimizers,thm:endpoint-optimizers}
use the supporting-line geometry underlying the Lubetzky--Zhao boundary,
but identify the optimizers using different strategies.
Let $r_*=(d-1)/d$ be the exceptional density.

\begin{enumerate}[label=(\roman*),leftmargin=2.2em]
\item \emph{Nonexceptional case (\Cref{thm:nonexceptional-optimizers}).}
For $r\ne r_*$, we first prove a quadratic lower bound for the gap between
the scalar function $x\mapsto J_{\pc(r)}(x^{1/d})$ and its supporting line
at $x=r^d$.  This line also touches the curve at $x=\sm(r)^d$, with
$\sm(r)\ne r$.  Then, we show that the gap is at least a
positive constant times the squared distance from $x$ to the nearer of
these two contact points.  The same positive constant works for every $r$
in a fixed closed interval contained in $(0,1)$ that does not contain
$r_*$.  We exclude $r_*$ because there the two contacts merge and the gap
vanishes to fourth order, so a quadratic lower bound no longer holds.
See \Cref{sec:lz-boundary}.

We next reduce the choice of an optimizing graphon to the choice of its
edge density.  We show that every optimizer
has exactly the required $H$-density $r^{|E(H)|}$, and that its edge density
approaches $r$ from below as $p\uparrow\pc(r)$, so that we can constrain its
edge density at $r-\delta$, where $\delta>0$, the edge-density deficit, is sufficiently small.
By using the relative entropy identity \eqref{eq:relative-entropy-identity},
for each fixed $p$, $r$, and $\delta$, minimizing $I_p$ among graphons with the
two constraints $e(W)=r-\delta$ and $t(H,W)=r^{|E(H)|}$ is equivalent to maximizing
graphon entropy in \eqref{eq:graphon-entropy}.
The theorem of Kenyon, Radin, Ren, and Sadun \mbox{(KRR--S)}~\cite{kenyon2014entropy}
shows that maximizing graphon entropy under these two constraints
yields a unique bipodal graphon, up to relabeling, for sufficiently small $\delta>0$.
Thus, each $\delta$ selects one bipodal candidate for the original relative
entropy minimization problem.  Write $B_\delta$ for this candidate.
See \Cref{sec:local-reduction}.

It remains to prove that the minimizing deficit $\delta$ is unique and to
derive the expansions of the corresponding edge density and relative entropy.
For fixed $r$, we first compute the extra cost of $B_\delta$ at $p=\pc(r)$,
compared to the constant graphon $W\equiv r$,
since changing $p$ adds only an explicit linear correction in $\delta$, as shown below.
The constant graphon is optimal there, so the comparison quantifies
the cost of lowering the edge density while keeping the same $H$-density.
For sufficiently small $\delta>0$, the scalar gap bound and generalized
H\"older's inequality give a quadratic lower bound on this extra cost.
We also obtain a quadratic upper bound by constructing a graphon with
the same densities and an extra cost of order $\delta^2$; $B_\delta$ has no
greater cost because it minimizes relative entropy under these constraints.
The analytic bipodal parametrization then gives the leading Taylor term
$\frac12A_H(r)\delta^2$, where $A_H(r)>0$ is the second derivative of the
extra cost at $\delta=0$ and is analytic in $r$.
See \Cref{sec:nonexceptional-quadratic-growth}.

We now shift this boundary calculation to $p<\pc(r)$, where the constant graphon is no longer optimal.
The relative entropy identity \eqref{eq:relative-entropy-identity} gives, for $p<\pc(r)$,
\[
  I_p(B_\delta)-J_p(r)
  =\bigl[I_{\pc(r)}(B_\delta)-J_{\pc(r)}(r)\bigr]
   -\lambda(p,r)\delta.
\]
The bracketed term is the excess cost at the boundary, analyzed in the
previous paragraph; the term $-\lambda(p,r)\delta$ is the correction
caused by changing $p$.  Inserting the previous boundary calculation gives
\[
  I_p(B_\delta)-J_p(r)
  =-\lambda(p,r)\delta+\frac12A_H(r)\delta^2+O_{H,r}(\delta^3)
  \qquad(\delta\downarrow0,\ r\text{ fixed}).
\]
The linear term favors a reduction in edge density, while the quadratic
term penalizes it.  Balancing the two gives a minimizing value of $\delta$
with leading term $\lambda(p,r)/A_H(r)$.  Strict convexity makes this
minimizing value unique, and hence also makes the graphon optimizer unique
up to relabeling.  The analytic implicit function theorem shows that the
minimizing $\delta$ depends analytically on $(p,r)$.  Evaluating the
parameters of $B_\delta$ at the minimizing value of $\delta$ determines
the optimizing graphon; the parameters also depend analytically on $(p,r)$
due to the KRR--S theorem.  Evaluating $I_p(B_\delta)$ at the same value
gives the expansion of the minimum relative entropy $\Phi_H(p,r)$.
See \Cref{sec:nonexceptional-proof}.

\item \emph{Exceptional case (\Cref{thm:endpoint-optimizers}).}
At $r=r_*$, the curve $x\mapsto J_{\pc(r_*)}(x^{1/d})$ has a supporting
line at $x=r_*^d$, and the gap between the curve and this line vanishes
to fourth order there.  Since the KRR--S theorem does not cover this density,
we develop a separate argument in \Cref{sec:singular-endpoint}.

We first seek a family of candidate graphons of the product form $W(x,y)=f(x)f(y)$
for the original relative entropy minimization problem.
Here $f$ assigns a weight to each vertex, and the edge probability is the
product of the two vertex weights.  This form simplifies the $H$-density:
since $H$ is $d$-regular,
\[
  t(H,W)=\left(\int_0^1 f(x)^d\,dx\right)^{|V(H)|}.
\]
Thus, the $H$-density constraint reduces to a condition on a single
integral of $f$.  We use the product form to construct candidate graphons;
we then prove their optimality by comparison with arbitrary graphons
satisfying the required $H$-density constraint.

To determine the candidate graphons, we impose necessary conditions for a
constrained minimum: the derivative of the relative entropy must vanish
under differentiable perturbations that keep the $H$-density fixed.
These are the \emph{stationarity conditions}.  For nonconstant $W$ of
this product form with $f$ bounded away from $0$ and $1$, the resulting equations
force $f$ to take exactly two values, thus $W$ is bipodal.  We also impose
stationarity under changes of block size.  We use $h>0$ to denote half the difference
between the two values of $f$.  The stationarity conditions for the
graphon values and for changes in block size together determine the two
values of $f$, the block sizes, and $p$ as analytic functions of $h$
near $0$.  We define $r$ by $t(H,W)=r^{|E(H)|}$ and parametrize the resulting
family by $h \mapsto (f_h,W_h,p_h,r_h)$; see \Cref{sec:rank-one-stationary-family}.

By construction, the block sizes tend to $1/2$ and the parameter pair
$(p_h,r_h)$ approaches $(\pc(r_*),r_*)$ as $h\downarrow0$.
We then show that $W_h$ has lower relative entropy, by a quantity of order
$h^4$, than the constant graphon $W\equiv r_h$ with the same $H$-density.
See \Cref{sec:constant-graphon-comparison}.

It remains to show that $W_h$ is the unique optimizer among all graphons.
Take any graphon $W$ satisfying the density constraint
$t(H,W)\ge r_h^{|E(H)|}$ and having relative entropy
$I_{p_h}(W)\le I_{p_h}(W_h)$.  We must prove that $W$ agrees with $W_h$ up to
relabeling.  The fourth-order supporting gap described above provides
the first step: together with the density constraint and this relative entropy
comparison, it gives
\begin{equation}\label{eq:intro-quartic-localization}
  \int_{[0,1]^2}|W(x,y)-r_*|^4\,dx\,dy\le C_d h^4.
\end{equation}
Thus, every such competitor is within $O_d(h)$ of the constant graphon
$r_*$ in the $L^4$-norm.  To compare such $W$ with $W_h$ more precisely, we write
$W(x,y)=f(x)f(y)+E(x,y)$, separating a product part from a remainder $E$.
We choose $f$ so that the difference between the $H$-densities of $W$
and its product part is at most a cubic quantity in the $L^2$-norm of $E$:
\begin{equation}\label{eq:intro-density-error}
  \left|t(H,W)-\left(\int_0^1 f(x)^d\,dx\right)^{|V(H)|}\right|
  \le C_H\|E\|_2^3.
\end{equation}
For this choice of $f$, the localization bound
\eqref{eq:intro-quartic-localization} gives
$\|f-\sqrt{r_*}\|_4=O_d(h)$ and $\|E\|_2=O_d(h)$.
These bounds provide the smallness needed to control the higher-order
terms in the following entropy comparison.
See \Cref{sec:localization-rank-one}.

We then establish a lower bound for the relative entropy difference
$I_{p_h}(W)-I_{p_h}(W_h)$ that accounts both for the remainder $E$
and for deviations of $f$ from the two distinct values
$s_h$ and $t_h$ taken by $f_h$.  Combining this bound with the density
constraint and the cubic bound
\eqref{eq:intro-density-error} shows that, for sufficiently
small $h$, $I_{p_h}(W)\ge I_{p_h}(W_h)$, with equality if and only if
\[
  E(x,y)=0 \quad \text{a.e.},\qquad
  f(x)\in\{s_h,t_h\} \quad \text{a.e.},\qquad
  \int_0^1 f(x)^d\,dx=\int_0^1 f_h(x)^d\,dx.
\]
Our assumption $I_{p_h}(W)\le I_{p_h}(W_h)$ therefore forces all three
conditions.  The first two say that $W$ itself has the product form,
with the same two factor values as $W_h$.  The last condition fixes the block sizes.
Thus, $W$ agrees with $W_h$ up to relabeling.  This proves optimality
and uniqueness along the constructed curve.
Since $\Phi_H(p_h,r_h)=I_{p_h}(W_h)$, the earlier cost expansion for $W_h$
now gives the expansion of the minimum relative entropy.
See \Cref{sec:auxiliary-lagrangian,sec:graphon-comparison,sec:singular-proof}.
\end{enumerate}

\paragraph{Organization.}
\Cref{sec:preliminaries} records the graphon preliminaries and the
entropy-maximization results.  \Cref{sec:lz-boundary} analyzes the Lubetzky--Zhao
phase boundary, and \Cref{sec:nonexceptional-endpoint,sec:singular-endpoint}
prove the nonexceptional and exceptional results, respectively.
\Cref{sec:lean-formalization} describes our Lean 4 formalization,
including its scope and assumptions.
\Cref{sec:concluding-remarks} discusses further directions.
Technical arguments deferred from the main text
appear in Appendices~\ref{app:krrs-analytic-extension}, \ref{app:lz-scalar-geometry},
\ref{app:nonexceptional-proofs}, and
\ref{app:endpoint-calculations}, which contain, respectively, the derivation of the
two-sided KRR--S parametrization, the scalar convex-analysis proofs,
the reduction proofs and bipodal parameter expansions for nonexceptional endpoints,
and the construction and parameter calculations for the singular-endpoint family.

%% file: sections/preliminaries.tex
\section{Preliminaries}\label{sec:preliminaries}

This section reviews the graphon framework and the background results used
later.  We recall the entropy-maximization results of Kenyon, Radin, Ren,
and Sadun~\cite{kenyon2014entropy} and state the two-sided extension of their
parametrization, proved in Appendix~\ref{app:krrs-analytic-extension}.

Subscripts on $O$, $\Omega$, and $\Theta$ indicate the dependence of
implicit constants; a set subscript indicates uniformity over that set.
Subscripts on $o$ indicate fixed parameters, with no uniformity asserted
unless stated.

\subsection{Graphons}\label{sec:graphons}

A graphon is a symmetric measurable function
$W:[0,1]^2\to[0,1]$.  We view $W$ as specifying a random graph model with
vertex labels in $[0,1]$, where $W(x,y)$ gives the probability of an edge
between vertices labeled $x$ and $y$.  We denote the space of all graphons
by $\W_0$.

The \emph{cut norm} of a symmetric measurable
$F:[0,1]^2\to\RR$ is
\[
  \|F\|_\cut
  =\sup_{S,T}\left|\int_{S\times T}F(x,y)\dd x\dd y\right|,
\]
the supremum being over measurable $S,T\subseteq[0,1]$, and the \emph{cut
distance} between two graphons is
\[
  \delta_\cut(U,W)=\inf_\sigma\|U-W^\sigma\|_\cut,
  \qquad
  W^\sigma(x,y):=W(\sigma x,\sigma y),
\]
where $\sigma$ ranges over measure-preserving a.e. bijections of $[0,1]$.
Two graphons $U$ and $W$ are \emph{weakly isomorphic} if
$\delta_\cut(U,W)=0$.  We denote the quotient of $\W_0$ by this equivalence
relation by $\widetilde{\W}_0$.  The cut distance induces a metric on this
quotient, and the resulting metric space $(\widetilde{\W}_0,\delta_\cut)$
is compact.  When discussing uniqueness or compactness, we work in
$\widetilde{\W}_0$: saying that a graphon is \emph{unique up to relabeling} means it is unique in this
quotient, i.e., graphons differing by a measure-preserving
relabeling of $[0,1]$ are identified.  Every homomorphism density
$t(H,\cdot)$ is $\delta_\cut$-continuous, and $I_p$ is lower semicontinuous
in the cut metric.  See~\cite{LovaszGraphLimits} for these facts.

For a finite graph $G$, define its associated step graphon $W_G$ by
partitioning $[0,1]$ into $|V(G)|$ intervals $I_i$ of equal length,
indexed by the vertices of $G$, and setting $W_G=1$ on $I_i\times I_j$
when $ij\in E(G)$ and $W_G=0$ otherwise.
We then write $\delta_\cut(G,W)=\delta_\cut(W_G,W)$.
This is independent of the vertex ordering, which changes $W_G$ only
by a measure-preserving relabeling.

We say that a graphon $W$ is \emph{bipodal} if there is a measurable set
$A\subset[0,1]$ such that, writing $A^c=[0,1]\setminus A$, the graphon $W$
is almost everywhere equal to a constant on each of
\[
  A\times A,
  \qquad
  (A\times A^c)\cup(A^c\times A),
  \qquad
  A^c\times A^c.
\]
Equivalently, after relabeling, a bipodal graphon is described by parameters
$(q_{11},q_{12},q_{22},c)$, where $c=|A|$ denotes the Lebesgue measure, or block
size, of $A$.  The parameters $q_{11}$ and $q_{22}$ are the edge densities
within $A$ and $A^c$, respectively, and $q_{12}$ is the edge density between
the two blocks.  Following the terminology in the entropy-maximization
literature, we also call the two blocks the two \emph{podes} of the graphon.

Recall the edge density $e(W)$, homomorphism density $t(H,W)$, and relative
entropy $I_p(W)$ defined in \eqref{eq:graphon-quantities}.
We write $W\equiv z$ for the constant graphon with value $z$.
For a finite $d$-regular graph $H$, generalized H\"older's inequality gives
the following bound for every graphon $W$ (see \cite{lubetzky2015large} for a proof):
\begin{equation}\label{eq:generalized-holder}
  t(H,W)\le
  \left(\int_{[0,1]^2}W(x,y)^d\,dx\,dy\right)^{|E(H)|/d}.
\end{equation}

\subsection{Entropy maximization problem}
The entropy-maximization problem reviewed here is closely related to the
Chatterjee--Varadhan variational problem~\eqref{eq:graphon-variational-problem}.
The connection between such constrained entropy problems and large
deviations for the uniform random graph $G(n,M)$ was established by Dembo
and Lubetzky~\cite{DemboLubetzky2018}.  For related work on constrained
graphon entropy maximization, see
\cite{KenyonRadinRenSadun2017,KenyonRadinRenSadun2017b,NeemanRadinSadun2023}.
Define the entropy density and the graphon entropy by
\begin{equation}\label{eq:graphon-entropy}
  S_0(z):=-\frac12\bigl[z\log z+(1-z)\log(1-z)\bigr], \qquad
  s(W)=\int_{[0,1]^2}S_0(W(x,y))\,dx\,dy.
\end{equation}
For fixed $p$ and fixed edge density, minimizing $I_p$ is equivalent to
maximizing $s$, since
\begin{equation}\label{eq:relative-entropy-identity}
  I_p(W)=-2s(W)-\log(1-p)+e(W)\log\frac{1-p}{p}.
\end{equation}

Fix a $d$-regular graph $H$ with $d\ge2$ and write
$m:=|E(H)|$.  For $0<\eps<1$ and $\eps^m<\tau<1$, consider the
entropy-maximization problem with edge density $\eps$ and $H$-density $\tau$:
\[
  \sup\{s(W):e(W)=\eps,\ t(H,W)=\tau\}.
\]
The constant graphon $W\equiv\eps$ has $H$-density $\eps^m$.
Accordingly, define
\[
  \vartheta:=\tau-\eps^m
\]
which is the excess $H$-density above the constant-graphon value.

For small positive excess, Kenyon, Radin, Ren,
and Sadun~\cite[Theorem~1.1]{kenyon2014entropy} proved
that for every $\eps\ne(d-1)/d$ there is a threshold
$\tau_0(\eps)>\eps^m$ such that, whenever
$\eps^m<\tau<\tau_0(\eps)$, this problem has a unique maximizer up to
relabeling, and the maximizer is bipodal.  Its four bipodal parameters are
real-analytic in $(\eps,\tau)$ throughout this region, which
is therefore open.  Consequently, for every fixed
$\eps_0\ne(d-1)/d$, one can choose a neighborhood $U$ of $\eps_0$ and a
single $\Delta>0$ such that the KRR--S conclusions hold for every $\eps\in U$
whenever $0<\vartheta<\Delta$.

In addition to establishing this analytic parametrization for
$\vartheta>0$, KRR--S determined its one-sided behavior as
$\vartheta\downarrow0$.  To state their result for the limiting cross density
$q_{12}$, define
\[
  \psi_d(\eps,z):=
  \frac{2\bigl[S_0(z)-S_0(\eps)-S_0'(\eps)(z-\eps)\bigr]}
       {z^d-\eps^d-d\eps^{d-1}(z-\eps)},
\]
where the removable singularity at $z=\eps$ is filled in continuously.  For
each $\eps\in(0,1)$, the function $z\mapsto\psi_d(\eps,z)$ has a unique
maximizer on $(0,1)$, denoted by $\zeta_d(\eps)$; see
\cite[Theorem~3.3 and Section~4]{kenyon2014entropy}.  With this notation, for
each fixed $\eps\ne(d-1)/d$, their result gives
\[
  q_{22}\longrightarrow\eps,
  \qquad
  q_{12}\longrightarrow\zeta_d(\eps),
  \qquad
  c=O_{H,\eps}(\vartheta)
  \qquad(\vartheta\downarrow0).
\]
\Cref{thm:krrs-analytic-extension} below also implies that $\zeta_d$ is real-analytic on
$(0,1)\setminus\{(d-1)/d\}$.

Because the four parameters depend on $(\eps,\tau)$, we call the function
\[
  (\eps,\tau)\longmapsto(q_{11},q_{12},q_{22},c)
\]
the KRR--S parameter map.  Equivalently, using
$\tau=\eps^m+\vartheta$, we view it as a function of
$(\eps,\vartheta)$.  Our later analysis differentiates its components and
uses their Taylor expansions at the boundary $\vartheta=0$.  For this
purpose, the parameter map must be defined on an open neighborhood of the
boundary, whereas KRR--S define it only for $\vartheta>0$.

The following theorem serves two purposes.  It records the KRR--S
conclusions for $\vartheta>0$ in the form used throughout the rest of the
paper, so that later arguments can refer directly to this theorem.  It also
adds the two-sided real-analytic continuation of the parameter map through
$\vartheta=0$ that those arguments require.
\begin{theorem}[Two-sided KRR--S extension for regular graphs]\label{thm:krrs-analytic-extension}
Fix an integer $d\ge2$.  Let $H$ be a $d$-regular graph with $m=|E(H)|\ge2$.
Fix an edge density
\[
  \eps_0\in(0,1)\setminus\left\{\frac{d-1}{d}\right\}.
\]
Then, there are an open neighborhood
$U\subseteq(0,1)\setminus\{(d-1)/d\}$ of $\eps_0$ and a constant $\Delta>0$
with the following property.  For every $\eps\in U$ and every $\tau$ such that
$0<\vartheta:=\tau-\eps^m<\Delta$, the entropy maximizer among graphons satisfying
\[
  e(W)=\eps,
  \qquad
  t(H,W)=\tau
\]
is unique up to relabeling and is bipodal.  After relabeling so that the
smaller block is the first block, the corresponding bipodal parameters define a map
\[
  (\eps,\vartheta)\longmapsto
  \bigl(q_{11}(\eps,\vartheta),q_{12}(\eps,\vartheta),
  q_{22}(\eps,\vartheta),c(\eps,\vartheta)\bigr)
\]
on $U\times(0,\Delta)$ that extends real-analytically
to $U\times(-\Delta,\Delta)$.  On $U\times\{0\}$, the extension satisfies
\[
  q_{11}(\eps,0)\in(0,1),
  \qquad
  q_{12}(\eps,0)=\zeta_d(\eps)\ne\eps,
  \qquad
  q_{22}(\eps,0)=\eps,
  \qquad
  c(\eps,0)=0.
\]
Moreover,
\[
  c(\eps,\vartheta)=O_{H,U}(\vartheta)
  \qquad(\vartheta\downarrow0),
\]
uniformly for $\eps\in U$.
\end{theorem}
At $\vartheta=0$, the first pode has size zero, and hence the bipodal graphon
reduces to the constant graphon $W\equiv\eps$.  Consequently, the values of
$q_{11}$ and $q_{12}$ on the null first pode are not determined by the graphon
itself; their values at $\vartheta=0$ are the distinguished ones selected by
analytic continuation from $\vartheta>0$.  Likewise, for $\vartheta<0$, the
extended parameter functions are only analytic continuations and do not
describe entropy maximizers.
The proof of \Cref{thm:krrs-analytic-extension} uses two applications of the analytic implicit
function theorem, using $d$-regular restrictions of Theorems~1.1 and~3.3 of KRR--S~\cite{kenyon2014entropy};
see \Cref{app:krrs-analytic-extension}.

%% file: sections/lz_boundary.tex
\section{The Lubetzky--Zhao boundary}\label{sec:lz-boundary}
Fix $d\ge2$, and write
\[
  r_*:=\frac{d-1}{d},
  \qquad
  p_*:=\frac{d-1}{(d-1)+\exp(d/(d-1))}.
\]
For $p\in(0,1)$, define
\[
  \varphi_{p,d}(x)=J_p(x^{1/d})\qquad (0\le x\le1).
\]
Recall that the result of Lubetzky and Zhao~\cite{lubetzky2015large} states
that when $H$ is $d$-regular, the constant graphon $W\equiv r$ is the unique minimizer
of~\eqref{eq:graphon-variational-problem} if and only if the point
$(r^d, J_p(r))$ lies on the convex minorant\footnote{For a continuous
function $f$ on a compact interval $I$, the convex minorant is the largest
convex function $g:I\to\mathbb R$ such that $g(x)\le f(x)$ for every
$x\in I$.} of $\varphi_{p,d}$.
In this section, we analyze in detail the behavior of the function $\varphi_{p,d}$ and of its convex minorant; this is the key to the proof of \Cref{thm:nonexceptional-optimizers}.
The main objective of this section is to prove the following.
\begin{theorem}[Scalar Lubetzky--Zhao boundary]\label{thm:scalar-lz-boundary}
    There exist analytic functions
    \[
        \pc,\sm:(0,1)\setminus\{r_*\}\to(0,1)
    \]
    such that the following holds for every $r \in (0,1)\setminus\{r_*\}$, where $x_r = r^d$ and $x_s = \sm(r)^d$.
    \begin{enumerate}[label=(M\arabic*)]
        \item $0<\pc(r)<\min\{r,p_*\}$ and $\sm(r)\neq r$.
        \item For every $p\in(0,1)$, the point $(r^d, J_p(r))$ lies on the convex minorant of $\varphi_{p,d}$ if and only if $p\ge \pc(r)$. Thus, $\pc(r)$ defines the Lubetzky--Zhao boundary.
        \item The tangent line   \[
            \ell_r(x)=\varphi_{\pc(r),d}(x_r)+\varphi_{\pc(r),d}'(x_r)(x-x_r)
          \]
        to the graph of $\varphi_{\pc(r),d}$ at $x_r$ lies below that graph, i.e.\ $\ell_r(x)\le\varphi_{\pc(r),d}(x)$ for all $x\in[0,1]$.
        \item Equality in (M3) holds exactly when $x\in\{x_r,x_s\}$.
        \item For every compact $K \subseteq (0,1)\setminus\{r_*\}$, there is $\gamma_K>0$ such that, for all $r\in K$ and all $x\in[0,1]$,
        \begin{equation}\label{eq:contact-quadratic-separation}
          \varphi_{\pc(r),d}(x)-\ell_r(x) \ge \gamma_K\,\dist(x,\{r^d,\sm(r)^d\})^2.
        \end{equation}
    \end{enumerate} 
    Furthermore, $\pc$ and $\sm$ extend continuously to the exceptional
    point by setting
    \[
      \pc(r_*)=p_*,
      \qquad
      \sm(r_*)=r_*.
    \]
    With this continuous extension, condition (M2) holds for every $r\in(0,1)$.
\end{theorem}

The curvature threshold and the qualitative common-tangent description used
below already appear in the scalar analysis of Lubetzky and
Zhao~\cite[Appendix~A]{lubetzky2015large}.  We record the details in a form
adapted to the later local argument.  In particular, we establish analytic
dependence of the contact points, their endpoint limits, and the uniform
quadratic separation in (M5).  To keep the main line of the paper visible,
the proofs of the following three supporting lemmas are collected in
Appendix~\ref{app:lz-scalar-geometry}.

For $p\in(0,1)$ set
\[
  h_{p,d}(z):=zJ_p''(z)-(d-1)J_p'(z)
  \qquad (0<z<1).
\]

\begin{lemma}[Convexity defect]\label{lem:convexity-defect}
For every $p\in(0,1)$,
\[
  h_{p,d}'(z)=\frac{d(z-r_*)}{z(1-z)^2}
  \qquad\text{and}\qquad
  h_{p,d}(r_*)=d-(d-1)\log\frac{(d-1)(1-p)}{p}.
\]
Thus, $h_{p,d}$ is strictly decreasing on $(0,r_*)$, strictly increasing on
$(r_*,1)$, and has its unique minimum at $r_*$.  Moreover, for
$x\in(0,1)$ and $z=x^{1/d}$,
\[
  \varphi_{p,d}''(x)=\frac{h_{p,d}(z)}{d^2z^{2d-1}},
\]
so $h_{p,d}(z)$ and $\varphi_{p,d}''(x)$ have the same sign.  If $p\ge p_*$,
then $h_{p,d}\ge0$ and $\varphi_{p,d}$ is convex on $[0,1]$.  If $p<p_*$,
then $h_{p,d}$ has exactly two zeros
\[
  u_-(p)<r_*<u_+(p),
\]
is positive on $(0,u_-(p))\cup(u_+(p),1)$, and is negative on
$(u_-(p),u_+(p))$.  Accordingly, $\varphi_{p,d}$ is strictly convex on
$(0,u_-(p)^d)$ and $(u_+(p)^d,1)$ and strictly concave on
$(u_-(p)^d,u_+(p)^d)$.
\end{lemma}

When $p<p_*$, the graph of $\varphi_{p,d}$ has one concave region between two strictly convex regions.
The convex minorant of $\varphi_{p,d}$ consists of the two convex pieces and an affine face that covers the concave region.  
Moreover, at the two endpoints of this affine face, the line must be tangent to the graph, so that it joins smoothly onto the remaining convex pieces and stays below the graph everywhere.
The following lemma proves this formally. 

\begin{lemma}[Contact points]\label{lem:contact-points}
For every $p\in(0,p_*)$, there are unique points, called the \emph{contact points},
\[
  0<x_a(p)<u_-(p)^d<u_+(p)^d<x_b(p)<1,
\]
where $u_-(p)<r_*<u_+(p)$ are the two zeros of $h_{p,d}$, such that
\begin{align}
  \varphi_{p,d}'(x_a(p))&=\varphi_{p,d}'(x_b(p)),\label{eq:contact-equal-slopes}\\
  \varphi_{p,d}(x_b(p))-\varphi_{p,d}(x_a(p))
  &=\varphi_{p,d}'(x_a(p))(x_b(p)-x_a(p)).\label{eq:contact-chord-identity}
\end{align}
The convex minorant of $\varphi_{p,d}$ agrees with $\varphi_{p,d}$ outside
$[x_a(p),x_b(p)]$; on $[x_a(p),x_b(p)]$, it is the straight line segment joining
the two contact points, and it lies strictly below $\varphi_{p,d}$ on the open
interval $(x_a(p),x_b(p))$.  The maps $p\mapsto x_a(p),x_b(p)$ are analytic on
$(0,p_*)$, with $dx_a/dp>0$ and $dx_b/dp<0$.
\end{lemma}

\begin{lemma}[Contact-point limits]\label{lem:contact-point-limits}
As $p\uparrow p_*$, both
$u_a(p):=x_a(p)^{1/d}$ and $u_b(p):=x_b(p)^{1/d}$ converge to $r_*$.  As
$p\downarrow0$, $u_a(p)\downarrow0$ and $u_b(p)\uparrow1$.
\end{lemma}

We finally prove Theorem~\ref{thm:scalar-lz-boundary}.

\begin{proof}[Proof of Theorem~\ref{thm:scalar-lz-boundary}]
\emph{Construction and analyticity.}
For $p\in(0,p_*)$, set
\[
  u_a(p):=x_a(p)^{1/d},
  \qquad
  u_b(p):=x_b(p)^{1/d}.
\]
By Lemmas~\ref{lem:contact-points} and
\ref{lem:contact-point-limits}, $p\mapsto u_a(p)$ is an analytic
increasing bijection from $(0,p_*)$ to $(0,r_*)$, and $p\mapsto u_b(p)$ is
an analytic decreasing bijection from $(0,p_*)$ to $(r_*,1)$.  For
$r\in(0,r_*)$, define $\pc(r)$ as the inverse of $p\mapsto u_a(p)$ and put
\[
  \sm(r)=u_b(\pc(r)).
\]
For $r\in(r_*,1)$, define $\pc(r)$ as the inverse of
$p\mapsto u_b(p)$ and put
\[
  \sm(r)=u_a(\pc(r)).
\]
The derivatives of $u_a$ and $u_b$ are nonzero by
Lemma~\ref{lem:contact-points}, so the analytic inverse function theorem
shows that these inverse functions are locally analytic.  These local
analytic inverses can be globally extended: on any overlap, they agree with the unique
set-theoretic inverse of the same one-to-one map.  Hence, $\pc$ is analytic
on each component of $(0,1)\setminus\{r_*\}$.  The function $\sm$ is analytic
because $\sm(r)=u_b(\pc(r))$ on $(0,r_*)$ and
$\sm(r)=u_a(\pc(r))$ on $(r_*,1)$.

The detailed verifications of (M1), (M3), and (M4) are deferred to
Appendix~\ref{app:lz-boundary-details}.  The geometric reason for
(M3)--(M4) is that, at $p=\pc(r)$, Lemma~\ref{lem:contact-points} identifies
the affine face of the convex minorant with the common tangent at the two
contact points $r^d$ and $\sm(r)^d$, which lies below the graph and touches it
exactly at those two points.

\emph{Proof of (M5).}
It suffices to treat the case in which $K$ is contained in one component of
$(0,1)\setminus\{r_*\}$; for a general compact set, apply the argument on the
two components separately and take the smaller constant.  Fix such a compact
interval $K$.  For $r\in K$, write
\[
  g_r(x)=\varphi_{\pc(r),d}(x)-\ell_r(x).
\]

We first choose, uniformly in $r$, small intervals around the two contact
points.  Put $p_r=\pc(r)$, and let $u_-(p_r)<r_*<u_+(p_r)$ be the two zeros of
$h_{p_r,d}$.  By Lemma~\ref{lem:contact-points}, the two contact points, in
increasing order, satisfy
\[
  x_a(p_r)<u_-(p_r)^d<u_+(p_r)^d<x_b(p_r).
\]
The contact points vary continuously with $r$ by
Lemma~\ref{lem:contact-points}.  The zeros $u_-(p)$ and $u_+(p)$ also vary
continuously in $p$.  To see this, apply the implicit function theorem to
$h_{p,d}(z)=0$, noting that
\[
  \partial_z h_{p,d}(z)=\frac{d(z-r_*)}{z(1-z)^2}
\]
is nonzero at both zeros.  Since $p_r=\pc(r)$ is continuous, all the quantities
above vary continuously with $r$.  Hence, the positive gaps
\[
  x_a(p_r),\quad
  u_-(p_r)^d-x_a(p_r),\quad
  x_b(p_r)-u_+(p_r)^d,\quad
  1-x_b(p_r),\quad
  x_b(p_r)-x_a(p_r)
\]
have a common positive lower bound on $K$.  We may therefore choose $\eta>0$
such that, for every $r\in K$, the two closed intervals of radius $\eta$ around
$x_a(p_r)$ and $x_b(p_r)$ are disjoint and lie inside the two strictly convex
outer regions.

On the compact set
\[
  E_K=\{(r,x)\in K\times[0,1]:
  \dist(x,\{x_a(p_r),x_b(p_r)\})\le\eta\},
\]
the continuous function $(r,x)\mapsto\varphi_{\pc(r),d}''(x)$ is positive.
Let
\[
  m_K:=\min_{(r,x)\in E_K}\varphi_{\pc(r),d}''(x)>0.
\]

We now apply Taylor's theorem near each contact point.  Let
$c\in\{x_a(p_r),x_b(p_r)\}$, and suppose that $x$ lies in the interval of radius
$\eta$ around $c$.  The function $g_r$ satisfies $g_r(c)=0$ because $\ell_r$
agrees with the graph at both contact points.  It also satisfies $g_r'(c)=0$:
at the contact $r^d$, this is the definition of $\ell_r$ as a tangent line, and
at the other contact, it follows from the equal-slope condition in
Lemma~\ref{lem:contact-points}.  Since $\ell_r$ is affine,
$g_r''=\varphi_{\pc(r),d}''$.  Taylor's theorem gives, for some point $\xi$
between $c$ and $x$,
\[
  g_r(x)=\frac12 g_r''(\xi)(x-c)^2
  =\frac12\varphi_{\pc(r),d}''(\xi)(x-c)^2.
\]
The segment between $c$ and $x$ stays in the same $\eta$-interval, so
$\varphi_{\pc(r),d}''(\xi)\ge m_K$.  Since the two intervals are disjoint, $c$
is the closest contact point to $x$.  Using
\[
  \{x_a(p_r),x_b(p_r)\}=\{r^d,\sm(r)^d\},
\]
we get
\[
  g_r(x)\ge \frac{m_K}{2}\,
  \dist(x,\{r^d,\sm(r)^d\})^2
\]
inside the two contact neighborhoods.

It remains to handle the complement of these neighborhoods.  On the compact set
\[
  R_K:=\{(r,x): r\in K,\ x\in[0,1],\
  \dist(x,\{r^d,\sm(r)^d\})\ge\eta\},
\]
the continuous function $g_r(x)$ is strictly positive by (M4).  Hence
\[
  b_K:=\min_{(r,x)\in R_K}g_r(x)>0.
\]
Since the distance in $[0,1]$ is at most $1$, this also gives
\[
  g_r(x)\ge b_K\,\dist(x,\{r^d,\sm(r)^d\})^2
\]
on $R_K$.  Taking
\[
  \gamma_K=\min\{m_K/2,b_K\}
\]
gives (M5).

\emph{Proof of (M2) away from $r_*$.}
Fix $r\in(0,1)\setminus\{r_*\}$.
If $p\ge p_*$, then
$\varphi_{p,d}$ is convex by Lemma~\ref{lem:convexity-defect}, so every point
of its graph, in particular $(r^d,J_p(r))$, lies on its convex minorant.
Assume now that $p<p_*$.  If $r<r_*$, then $r=u_a(\pc(r))$.  If
$\pc(r)\le p<p_*$, then $u_a(p)\ge u_a(\pc(r))=r$, so
$r^d\le x_a(p)$.  Hence, $r^d$ lies outside the affine part of the convex
minorant, on the left side where the convex minorant agrees with
$\varphi_{p,d}$.  Thus, $(r^d,J_p(r))$ lies on the convex minorant.  If
$p<\pc(r)$, then $u_a(p)<r<r_*<u_b(p)$, so $r^d$ lies strictly between the two
contact points $x_a(p)$ and $x_b(p)$.  On this open interval the graph of
$\varphi_{p,d}$ lies strictly above the affine part of the convex minorant.
Thus, $(r^d,J_p(r))$ does not lie on the convex minorant.

If $r>r_*$, then $r=u_b(\pc(r))$.  If $\pc(r)\le p<p_*$, then
$u_b(p)\le u_b(\pc(r))=r$, so $r^d\ge x_b(p)$ and the convex minorant agrees with
$\varphi_{p,d}$ at $r^d$.  If $p<\pc(r)$, then
$u_a(p)<r_*<r<u_b(p)$, so $r^d$ lies strictly between the two contact points and
again the graph lies strictly above the affine part of the convex minorant.
This proves the stated dichotomy for all $p\in(0,1)$, and hence $p=\pc(r)$ is
the Lubetzky--Zhao boundary on $(0,1)\setminus\{r_*\}$.

\emph{Continuous extension and (M2) at $r_*$.}
By Lemma~\ref{lem:contact-point-limits},
\[
  u_a(p)\to r_*,
  \qquad
  u_b(p)\to r_*
  \qquad (p\uparrow p_*).
\]
Therefore the two inverse functions both satisfy $\pc(r)\to p_*$ as
$r\to r_*$ from their respective sides.  Moreover,
$\sm(r)=u_b(\pc(r))\to r_*$ as $r\uparrow r_*$, and
$\sm(r)=u_a(\pc(r))\to r_*$ as $r\downarrow r_*$.  Thus, defining
\[
  \pc(r_*)=p_*,
  \qquad
  \sm(r_*)=r_*
\]
gives continuous functions on $(0,1)$.

It remains only to check (M2) at the exceptional value $r=r_*$.  If
$p\ge p_*$, then $\varphi_{p,d}$ is convex by
Lemma~\ref{lem:convexity-defect}, so $(r_*^d,J_p(r_*))$ lies on its convex
minorant.  If $p<p_*$, then by Lemma~\ref{lem:contact-points},
\[
  x_a(p)<u_-^d<r_*^d<u_+^d<x_b(p).
\]
Hence, $r_*^d$ lies strictly inside the affine part of the convex minorant of
$\varphi_{p,d}$, where the graph is strictly above the minorant.  Thus,
$(r_*^d,J_p(r_*))$ does not lie on the convex minorant.  This proves (M2) at
$r=r_*$ and completes the proof of the continuous extension.
\end{proof}

%% file: sections/nonexceptional.tex
\section{Optimizers near nonexceptional endpoints}
\label{sec:nonexceptional-endpoint}

In this section, we prove \Cref{thm:nonexceptional-optimizers} in a more precise form.
We first use \Cref{sec:local-reduction} to reduce the graphon variational
problem on the symmetry-breaking side to a one-variable problem in the edge
density.  We then establish the quadratic growth of the reduced objective
at the phase boundary in \Cref{sec:nonexceptional-quadratic-growth}.
Finally, in \Cref{sec:nonexceptional-proof}, we minimize the reduced
objective on the symmetry-breaking side to complete the proof.
First-order refinements for the individual bipodal parameters are derived in
\Cref{app:bipodal-parameter-expansions}.

Recall the log-odds displacement
\[
  \lambda(p,r)=
  \log\frac{1-p}{p}-\log\frac{1-\pc(r)}{\pc(r)}.
\]
Thus, $\lambda(\pc(r),r)=0$, and $\lambda(p,r)>0$ is equivalent to
$p<\pc(r)$.  The following theorem is the main result of this section.
The remainder of this section is devoted to the proof of this theorem.

\begin{theorem}[Optimizers near nonexceptional endpoints]\label{thm:nonexceptional-endpoint}
Let $H$ be a fixed $d$-regular graph with $d\ge2$, and let
$r_0\in(0,1)\setminus\{r_*\}$.
There exist $\rho,\eta>0$ such that
\[
  U:=\{(p,r): |p-\pc(r)|<\eta,\ |r-r_0|<\rho\}
\]
is an open neighborhood of $(\pc(r_0),r_0)$ on which, with $p < \pc(r)$,
the minimizer of~\eqref{eq:graphon-variational-problem}
is unique up to relabeling, nonconstant, and bipodal.
Let $W_{p,r}$ denote the unique optimizer.
Write $c(p,r)\in(0,1/2)$ for the size of its smaller block,
$q_{11}(p,r)$ and $q_{22}(p,r)$ for the edge densities within the
smaller and larger blocks, respectively, and $q_{12}(p,r)$ for the
cross-block edge density.  These functions are real-analytic on $U\cap\{(p,r):p<\pc(r)\}$.
For fixed $r$ with $|r-r_0|<\rho$, the smaller block vanishes:
\[
  c(p,r)\to0 \qquad (p \uparrow\pc(r)).
\]
Moreover, the edge density and minimum relative entropy satisfy
\begin{equation}\label{eq:optimizer-expansion}
  e(W_{p,r})=r-\frac{\lambda(p,r)}{A_H(r)}+O_{H,U}(\lambda(p,r)^2),
  \qquad
  \Phi_H(p,r)=J_p(r)-\frac{\lambda(p,r)^2}{2A_H(r)}+O_{H,U}(\lambda(p,r)^3),
\end{equation}
where $A_H$ is positive and real-analytic, and the remainders are uniform
for $(p,r)\in U$ with $p<\pc(r)$.
\end{theorem}

\input{sections/nonexceptional_local_reduction}
\input{sections/nonexceptional_quadratic_growth}
\input{sections/nonexceptional_proof}

%% file: sections/nonexceptional_local_reduction.tex
\subsection{The local reduction}\label{sec:local-reduction}

Let $H$ be a $d$-regular graph with $d\ge2$ and write $m=|E(H)|$, so that
the homomorphism density constraint reads $t(H,W)\ge r^m$.
We obtain the reduction uniformly in a sufficiently small neighborhood of
each fixed $r_0\in(0,1)\setminus\{r_*\}$.
\Cref{lem:active-constraint,lem:edge-density-deficit} show that optimizers
saturate the $H$-density constraint and have edge density below $r$ on the
symmetry-breaking side.  The proof of boundary convergence in
\Cref{lem:boundary-convergence} is deferred to \Cref{app:local-reduction}.

\begin{lemma}[Active constraint]\label{lem:active-constraint}
Let $0<p<r<1$.  Every optimizer of \eqref{eq:graphon-variational-problem} satisfies
$t(H,W)=r^m$.
\end{lemma}

\begin{proof}
Let $W$ be an optimizer, and suppose for contradiction that $t(H,W)>r^m$.
For $0\le\theta\le1$, define
\[
  W_\theta=p+\theta(W-p).
\]
Then, $W_0\equiv p$ has $t(H,W_0)=p^m<r^m$, while $W_1=W$ has
$t(H,W)>r^m$.  By continuity of $t(H,\cdot)$, there exists
$0<\theta_0<1$ with $t(H,W_{\theta_0})=r^m$.  Since $J_p$ is strictly convex
and uniquely minimized at $p$ with $J_p(p) = 0$, we have, for every
$(x,y)\in[0,1]^2$,
\begin{align*}
  J_p(p+\theta_0(W(x,y)-p))
  & {}
  \le
  (1-\theta_0)J_p(p)+\theta_0J_p(W(x,y))
  =
  \theta_0J_p(W(x,y))
  {}
  \le
  J_p(W(x,y)),
\end{align*}
with strict inequality unless $W(x,y)=p$.  Integrating gives
\[
   I_p(W_{\theta_0}) \le I_p(W).
\]
On the other hand, $W_{\theta_0}$ is feasible, so the optimality of $W$ gives
$I_p(W_{\theta_0})\ge I_p(W)$.  Hence, $I_p(W_{\theta_0})=I_p(W)$.
Since the pointwise difference is nonnegative, the equality of the integrals forces
\[
  J_p(p+\theta_0(W(x,y)-p))=J_p(W(x,y))
\]
for almost every $(x,y)$, and therefore $W(x,y)=p$ almost everywhere.  This
would imply $t(H,W)=p^m<r^m$, contradicting $t(H,W)>r^m$.  Thus,  every
optimizer satisfies $t(H,W)=r^m$.
\end{proof}

\begin{lemma}[Strict edge-density deficit]\label{lem:edge-density-deficit}
Let $r\in(0,1)\setminus\{r_*\}$ and $0<p<\pc(r)$.
If $W$ is an optimizer of \eqref{eq:graphon-variational-problem}, then $e(W)<r$.
\end{lemma}

\begin{proof}
By condition (M2) of \Cref{thm:scalar-lz-boundary}, for $p<\pc(r)$, the scalar point $(r^d,J_p(r))$ is not on the convex minorant of
$x\mapsto J_p(x^{1/d})$. \Cref{thm:lz-criterion} therefore shows that $W\equiv r$ is not a
minimizer. Since $W\equiv r$ is feasible, every optimizer satisfies $I_p(W)<J_p(r)$.
By Jensen's inequality,
$
  I_p(W)\ge J_p(e(W)).
$
Since $p<\pc(r)<r$, the function $J_p$ is strictly increasing on $[r,1]$.
If $e(W)\ge r$, it is a contradiction since
\[
  I_p(W)\ge J_p(e(W))\ge J_p(r).
  \qedhere
\]
\end{proof}

\begin{lemma}[Boundary convergence of optimizers]\label{lem:boundary-convergence}
Let $r_n\in(0,1)\setminus\{r_*\}$ and $p_n\in(0,\pc(r_n))$ satisfy
\[
  r_n\to r_0\in(0,1)\setminus\{r_*\},
  \qquad
  \pc(r_n)-p_n\to0.
\]
For each $n$, let $W_n\in\W_H(p_n,r_n)$.  Then, writing $r_0$ also for the constant
graphon with value $r_0$,
\[
  \delta_\square(W_n,r_0)\to0,
  \qquad
  e(W_n)\to r_0.
\]
\end{lemma}

For $0\le\eps\le1$ and $0\le\tau\le1$, define the fixed-$(e,t(H,\cdot))$ entropy
envelope
\begin{equation}\label{eq:fixed-density-entropy}
  S_H(\eps,\tau)=
  \sup\{s(W):e(W)=\eps,\ t(H,W)=\tau\}.
\end{equation}
We use the convention that the supremum is $-\infty$ if the constraint set is
empty.  When the maximizer is bipodal and unique up to relabeling, denote it
by $B_{\eps,\tau}$; its subscripts specify the edge and $H$-densities.
For $p,r\in(0,1)$, define the reduced objective
\begin{equation}\label{eq:reduced-objective}
  I_{p,r}(\eps)=
  -2S_H(\eps,r^m)
  -\log(1-p)+\eps\log\frac{1-p}{p}.
\end{equation}
By \eqref{eq:relative-entropy-identity}, $I_{p,r}(\eps)$ is the
infimum of $I_p$ among graphons with edge density $\eps$ and $H$-density
$r^m$.  In particular, every graphon $W$ with these densities satisfies
\[
  I_p(W)\ge I_{p,r}(\eps),
\]
with equality exactly when $W$ maximizes entropy under these two constraints.

Every optimizer of the original problem satisfies $t(H,W)=r^m$ by
\Cref{lem:active-constraint}.  We can therefore first minimize the cost at
each fixed edge density $\eps$, then minimize over $\eps$.
Near the phase boundary on the symmetry-breaking side, the next lemma
places every optimizer's edge density in $(r-\rho,r)$, uniformly for $r$
near $r_0$.  For each $\eps$ in this interval,
\Cref{thm:krrs-analytic-extension} identifies the entropy maximizer at
densities $(\eps,r^m)$ as a unique bipodal graphon, up to relabeling.
Thus, only the choice of $\eps$ remains: \Cref{cor:scalar-reduction} shows
that minimizing $I_{p,r}$ on $\eps \in (r-\rho,r)$ is equivalent to the
original problem, with a one-to-one correspondence between scalar
minimizers and graphon optimizers modulo relabeling.
Its proof is deferred to \Cref{app:local-reduction}.

\begin{lemma}[Uniform localization and fixed-density bipodality]\label{lem:fixed-density-bipodality}
Fix $r_0\in(0,1)\setminus\{r_*\}$.
There exist $\rho_0>0$ and an open interval $I\subset(0,1)\setminus\{r_*\}$ containing
$r_0$ such that, for every $r \in I$ and $\eps\in(r-\rho_0,r)$,
the entropy maximizer with edge density
$\eps$ and $H$-density $r^m$ is bipodal and unique up to relabeling.
Thus $B_{\eps,r^m}$ is well defined up to relabeling, and for every $p\in(0,1)$,
\begin{equation}\label{eq:reduced-objective-attainment}
  s(B_{\eps,r^m})=S_H(\eps,r^m),
  \qquad
  I_{p,r}(\eps)=I_p(B_{\eps,r^m}).
\end{equation}
Furthermore, 
for every $0<\rho<\rho_0$, one can choose $0<\eta<\inf_{r\in I}\pc(r)$ so that,
if $r\in I$ and $\pc(r)-\eta<p<\pc(r)$, then every optimizer $W$ of
\eqref{eq:graphon-variational-problem} has edge density
\[
  e(W)\in(r-\rho,r).
\]
\end{lemma}

\begin{proof}
Apply \Cref{thm:krrs-analytic-extension} once at $\eps_0=r_0$, obtaining
an open neighborhood $V$ and a width $\Delta>0$ for the analytic parameter
map.  Choose an open interval $I$ containing $r_0$ whose compact closure
lies in $V\cap((0,1)\setminus\{r_*\})$.  Choose $\rho_0>0$ so small that
$[r-\rho_0,r]\subset V\cap(0,1)$ for every $r\in\overline I$ and
$m\rho_0<\Delta$.
For $r\in I$ and $r-\rho_0<\eps<r$, the mean value theorem gives
\[
  0<r^m-\eps^m\le m(r-\eps)<\Delta.
\]
Thus, for every $r\in I$ and $\eps\in(r-\rho_0,r)$,
\Cref{thm:krrs-analytic-extension} with $\tau=r^m$ supplies $B_{\eps,r^m}$
that satisfies $e(B_{\eps,r^m})=\eps$, $t(H,B_{\eps,r^m})=r^m$,
maximizes entropy among graphons with these densities,
and is unique up to relabeling.
Thus the definition \eqref{eq:fixed-density-entropy} gives $s(B_{\eps,r^m})=S_H(\eps,r^m)$,
and substituting $B_{\eps,r^m}$ into \eqref{eq:relative-entropy-identity} gives,
for every $p\in(0,1)$,
\[
  I_p(B_{\eps,r^m})=-2S_H(\eps,r^m)-\log(1-p)+\eps\log\frac{1-p}{p}=I_{p,r}(\eps).
\]

Fix $0<\rho<\rho_0$.  We claim that, for some
$0<\eta<\inf_{r\in I}\pc(r)$, every optimizer
$W$ of \eqref{eq:graphon-variational-problem} with
$r\in I$ and $\pc(r)-\eta<p<\pc(r)$ satisfies
\begin{equation}\label{eq:edge-density-localization}
  e(W)\in(r-\rho,r).
\end{equation}
Otherwise, taking $\eta_n\to0$, there
are $r_n\in I$, $0<p_n<\pc(r_n)$ with $\pc(r_n)-p_n\to0$, and optimizers
$W_n$ violating \eqref{eq:edge-density-localization}.
By \Cref{lem:edge-density-deficit}, this means $e(W_n)\le r_n-\rho$.
Passing to a subsequence, compactness of $\overline I$ gives
$r_n\to\bar r$ for some $\bar r\in\overline I$.
Then \Cref{lem:boundary-convergence} gives $e(W_n)\to\bar r$,
contradicting $e(W_n)\le r_n-\rho$.
\end{proof}

\begin{corollary}[Uniform one-dimensional reduction]\label{cor:scalar-reduction}
Let $I$ and $\rho_0$ be as in \Cref{lem:fixed-density-bipodality}.
For each $0<\rho<\rho_0$, choose $\eta>0$ as in that lemma.
For every $r\in I$ and $p\in(\pc(r)-\eta,\pc(r))$, the variational problem
defining $\Phi_H(p,r)$ in \eqref{eq:graphon-variational-problem} reduces to the
one-dimensional problem of minimizing $I_{p,r}$:
\begin{equation}\label{eq:edge-density-minimization}
  \Phi_H(p,r)=
  \min_{\eps\in(r-\rho,r)}I_{p,r}(\eps).
\end{equation}
Furthermore, for such $p,r$, any optimizer $W^*$ attaining $\Phi_H(p,r)$
is $B_{\eps^*,r^m}$ up to relabeling, where $\eps^*=e(W^*)$ is a minimizer of
$I_{p,r}$ over $(r-\rho,r)$.  Conversely, if
$\eps\in(r-\rho,r)$ minimizes $I_{p,r}$ over $(r-\rho,r)$, then
$B_{\eps,r^m}$ is an optimizer of \eqref{eq:graphon-variational-problem}.  In
particular, the correspondence $W^*\mapsto e(W^*)$ is a bijection between the
set of optimizers modulo relabeling and the set of minimizers of
$I_{p,r}$ on $(r-\rho,r)$.
\end{corollary}

%% file: sections/nonexceptional_quadratic_growth.tex
\subsection{The quadratic cost of an edge-density deficit}\label{sec:nonexceptional-quadratic-growth}

At the phase boundary, we show that the reduced objective
$I_{\pc(r),r}(\eps)$ increases quadratically from the constant-graphon
cost as the edge density $\eps$ decreases from $r$, while the $H$-density
remains fixed at $r^m$.  We express this in terms of the edge-density deficit
\[
  \delta:=r-\eps.
\]
Note that by \Cref{lem:edge-density-deficit}, $\delta>0$ for every optimizer on the
symmetry-breaking side.
Fix $r_0\in(0,1)\setminus\{r_*\}$ and let $I$ be the open interval containing $r_0$
supplied by \Cref{lem:fixed-density-bipodality}, with compact closure in
$(0,1)\setminus\{r_*\}$.  We may shrink $I$ around $r_0$ below; all estimates
are uniform for $r\in I$.
At the boundary value $p=\pc(r)$, define the boundary excess
\begin{equation}\label{eq:boundary-excess}
  G_r(\delta)=I_{\pc(r),r}(r-\delta)-J_{\pc(r)}(r)
  \qquad (0<\delta<\rho_0),
\end{equation}
with $\rho_0$ as in \Cref{lem:fixed-density-bipodality}; that lemma ensures
that $I_{\pc(r),r}(r-\delta)$ is finite throughout this range.  The value
at $\delta=0$ belongs instead to the analytic extension $G(r,\delta)$
constructed in \Cref{thm:positive-second-variation}; it is not part of the domain of $G_r$.
Thus, $G_r(\delta)$ measures this additional cost relative to the constant
graphon.  More precisely, we will prove that
\[
  G_r(\delta)=\frac12A_H(r)\delta^2+O_{H,I}(\delta^3)
  \qquad (\delta\downarrow0),
  \qquad A_H(r)>0.
\]

For $r\in I$, recall that $\sm(r)$ is the second contact density associated
with the boundary point $(\pc(r),r)$.  Let $\ell_r$ be the supporting line to
the graph of $x\mapsto J_{\pc(r)}(x^{1/d})$ at $x=r^d$:
\[
  \ell_r(x)=J_{\pc(r)}(r)
  +\frac{J_{\pc(r)}'(r)}{d r^{d-1}}(x-r^d)
  \qquad (0 \le x \le 1).
\]
Define the \emph{supporting cost gap} by
\begin{equation}\label{eq:supporting-cost-gap}
  g_r(z)=J_{\pc(r)}(z)-\ell_r(z^d)
  \qquad (0\le z\le1).
\end{equation}
Since $\ell_r$ is a supporting line, $g_r\ge0$.  We claim that there is
$\gamma_{d,I}>0$ such that, uniformly for $r\in I$,
\begin{equation}\label{eq:supporting-gap-quadratic-bound}
  g_r(z)\ge \gamma_{d,I}\,\dist(z,\{r,\sm(r)\})^2.
\end{equation}
Apply \eqref{eq:contact-quadratic-separation} on $\overline I$.
The contact densities are uniformly bounded below by some $\eta_{d,I}>0$,
and for $z\in[0,1]$ and $v\in\{r,\sm(r)\}$,
\[
  |z^d-v^d|
  =
  |z-v|\,(z^{d-1}+z^{d-2}v+\cdots+v^{d-1})
  \ge
  \eta_{d,I}^{d-1}|z-v|.
\]
Taking the minimum over the two contact densities proves
\eqref{eq:supporting-gap-quadratic-bound}.

\begin{lemma}[Scalar quadratic lower bound]\label{lem:scalar-quadratic-bound}
There exist constants $C_{d,I}>0$ and $\delta_0>0$ such that, for every
$r\in I$, the following holds.  If $0<\delta<\delta_0$ and $X$ is a random variable with
values in $[0,1]$ satisfying
$\EE X=r-\delta$ and
$\EE X^d\ge r^d$,
then
\begin{equation}\label{eq:scalar-quadratic-bound}
  \EE J_{\pc(r)}(X)\ge J_{\pc(r)}(r)+C_{d,I}\delta^2.
\end{equation}
\end{lemma}

\begin{proof}
The slope of $\ell_r$ is positive, because $\pc(r)<r$.  Hence
\[
  \EE\ell_r(X^d)
  =   \ell_r(\EE X^d)
  \ge \ell_r(r^d)
  =   J_{\pc(r)}(r).
\]
Using \eqref{eq:supporting-cost-gap},
\[
  \EE g_r(X)
  =   \bigl(\EE J_{\pc(r)}(X)-\EE \ell_r(X^d)\bigr)
  \le \bigl(\EE J_{\pc(r)}(X)-J_{\pc(r)}(r)\bigr).
\]
It remains to show $\EE g_r(X)\ge C_{d,I}\delta^2$.
Set $s:=\sm(r)$ and
\[
  k:=\frac{s^d-r^d}{s-r}
  =s^{d-1}+s^{d-2}r+\cdots+r^{d-1}
  \ge\eta_{d,I}^{d-1}>0.
\]
Let $Y$ be the measurable nearest-point projection of $X$ onto $\{r,s\}$,
with ties assigned to $r$.  Since $Y\in\{r,s\}$, we have
$Y^d-r^d=k(Y-r)$.  Using $\EE X=r-\delta$ and the fact that
$z\mapsto z^d$ is $d$-Lipschitz on $[0,1]$, we obtain
\begin{align*}
  0\le\EE X^d-r^d
  &\le\EE Y^d-r^d+d\,\EE|X-Y|\\
  &=k\bigl(\EE(Y-X)-\delta\bigr)+d\,\EE|X-Y|\\
  &\le(k+d)\EE|X-Y|-k\delta.
\end{align*}
Hence
\[
  \EE|X-Y|\ge\frac{k}{k+d}\delta
  \ge\frac{\eta_{d,I}^{d-1}}{d+\eta_{d,I}^{d-1}}\delta.
\]
Since $|X-Y|=\dist(X,\{r,s\})$, the bound
\eqref{eq:supporting-gap-quadratic-bound} and Cauchy--Schwarz give
\[
  \EE g_r(X)\ge\gamma_{d,I}\,\EE|X-Y|^2
  \ge\gamma_{d,I}\bigl(\EE|X-Y|\bigr)^2
  \ge C_{d,I}\delta^2,
\]
where $C_{d,I}:=\gamma_{d,I}\bigl(\eta_{d,I}^{d-1}/(d+\eta_{d,I}^{d-1})\bigr)^2>0$.
\end{proof}

\begin{corollary}[Quadratic lower bound for graphons]\label{prop:graphon-quadratic-bound}
For every $r\in I$ and $0<\delta<\delta_0$, if $e(W)=r-\delta$ and
$t(H,W)\ge r^m$, then
\begin{equation}\label{eq:graphon-quadratic-bound}
  I_{\pc(r)}(W)
  \ge
  J_{\pc(r)}(r)+C_{d,I}\delta^2,
\end{equation}
with $C_{d,I}$ and $\delta_0$ as in \Cref{lem:scalar-quadratic-bound}.
\end{corollary}

\begin{proof}
Let $X=W(U,V)$, where $(U,V)$ is uniformly distributed on $[0,1]^2$.  Then
$\EE X=e(W)=r-\delta$.  By the generalized H\"older inequality
\eqref{eq:generalized-holder} and feasibility,
\[
  \EE X^d=\int W^d\ge t(H,W)^{d/m} \ge r^d.
\]
By \Cref{lem:scalar-quadratic-bound}, we conclude the proof.
\end{proof}

We now prove that the preceding quadratic lower bound has
the correct order.  For small $\delta>0$, we construct a bipodal graphon with
edge density $r-\delta$, $H$-density $r^m$, and $I_{\pc(r)}$-value only
$O_{H,I}(\delta^2)$ above that of the constant graphon.  The construction starts
from the constant graphon $r$, makes a small block, uses the second contact density
$s=\sm(r)$ on the cross edges, and then adjusts the large-block density to
restore the two constraints.

\begin{lemma}[Quadratic upper bound]\label{lem:bipodal-quadratic-bound}
After shrinking $I$ around $r_0$, there exist constants $C_{H,I}>0$ and
$\delta_0>0$ such that, for every $r\in I$ and $0<\delta<\delta_0$, there is a bipodal graphon
$V_\delta$ satisfying $e(V_\delta)=r-\delta$, $t(H,V_\delta)=r^m$,
and
\begin{equation}\label{eq:bipodal-quadratic-bound}
  I_{\pc(r)}(V_\delta)
  \le
  J_{\pc(r)}(r)+C_{H,I}\delta^2.
\end{equation}
\end{lemma}

\begin{proof}
Fix $a\in(0,1)$; we take $a=1/2$ at the end.
For $r$ near $r_0$, put $s=\sm(r)$ and let $V_{r,c,q}$ have block size $c$
and edge densities $a,s,q$ on the small block, cross edges, and large block,
respectively.  To solve the two constraints, define
\[
  F(r,\delta,c,q)
  :=\bigl(e(V_{r,c,q})-(r-\delta),\,t(H,V_{r,c,q})-r^m\bigr).
\]
The density formulas are polynomial in $c,a,s,q$, so they define $F$ also
for real $c$ near $0$.  Since $\sm$ is analytic near $r_0$, this extension
is analytic and satisfies $F(r,0,0,r)=0$.

Writing $n=|V(H)|=2m/d$, the assignments with zero or one vertex in the
small block give
\begin{align*}
  e(V_{r,c,q})&=c^2a+2c(1-c)s+(1-c)^2q,\\
  t(H,V_{r,c,q})&=(1-c)^nq^m
  +nc(1-c)^{n-1}s^dq^{m-d}+O_H(c^2).
\end{align*}
Thus, with
\[
  D_d(r,s):=\frac{s^d-r^d}{d r^{d-1}}-(s-r)>0,
\]
where positivity follows from strict convexity and $s\ne r$, the Jacobian is
\begin{gather*}
  \partial_{(c,q)}F(r,0,0,r)=
  \begin{pmatrix}
    2(s-r) & 1 \\
    \frac{2m}{d}r^{m-d}(s^d-r^d) & m r^{m-1}
  \end{pmatrix},\\
  \det\partial_{(c,q)}F(r,0,0,r)=-2m r^{m-1}D_d(r,s)<0.
\end{gather*}
The analytic implicit function theorem at $(r_0,0,0,r_0)$ therefore gives
jointly analytic functions $c(r,\delta)$ and $q(r,\delta)$ near $(r_0,0)$.
Local uniqueness and $F(r,0,0,r)=0$ give $c(r,0)=0$ and $q(r,0)=r$, while
\begin{equation}\label{eq:comparison-graphon-constraints}
  e(V_{r,c(r,\delta),q(r,\delta)})=r-\delta,
  \qquad
  t(H,V_{r,c(r,\delta),q(r,\delta)})=r^m.
\end{equation}
Shrink $I$ and choose $\delta_0>0$ so that
$\overline I\times[-\delta_0,\delta_0]$ lies in this analytic domain.
Differentiating the constraints at $\delta=0$ and using the Jacobian gives,
with uniform Taylor remainders,
\begin{equation}\label{eq:comparison-block-size}
  c(r,\delta)=\frac{\delta}{2D_d(r,s)}+O_{H,I,a}(\delta^2),
\end{equation}
\begin{equation}\label{eq:comparison-block-density}
  q(r,\delta)=r-\frac{2(s^d-r^d)}{d r^{d-1}}c(r,\delta)
  +O_{H,I,a}(c(r,\delta)^2).
\end{equation}
Since $D_d(r,\sm(r))$ is bounded above and away from zero on $\overline I$,
we may decrease $\delta_0$ so that $0<c(r,\delta)<1/2$ and $0<q(r,\delta)<1$
for every $r\in I$ and $0<\delta<\delta_0$.
Hence $V_\delta:=V_{r,c(r,\delta),q(r,\delta)}$ is a graphon with the required
densities.

It remains to estimate its cost.  Abbreviate
$c=c(r,\delta)$, $q=q(r,\delta)$.  Since $q-r=O_{H,I,a}(c)$ and all edge
densities and $\pc(r)$ stay away from $0$ and $1$, Taylor's theorem gives
\begin{align*}
  I_{\pc(r)}(V_\delta)-J_{\pc(r)}(r)
  &=J_{\pc(r)}'(r)(q-r)+2c\{J_{\pc(r)}(s)-J_{\pc(r)}(r)\}
    +O_{H,I,a}(c^2)\\
  &=2c\left\{J_{\pc(r)}(s)-J_{\pc(r)}(r)
    -\frac{J_{\pc(r)}'(r)}{d r^{d-1}}(s^d-r^d)\right\}
    +O_{H,I,a}(c^2).
\end{align*}
The bracket vanishes because $\ell_r$ touches the scalar curve at both
$r^d$ and $s^d$; see (M3)--(M4) of \Cref{thm:scalar-lz-boundary}.
Thus the cost difference is $O_{H,I,a}(c^2)=O_{H,I,a}(\delta^2)$ by
\eqref{eq:comparison-block-size}.  Taking $a=1/2$ proves
\eqref{eq:bipodal-quadratic-bound}.
\end{proof}

Recall that the boundary excess \eqref{eq:boundary-excess} is defined only for
$\delta>0$, because the Kenyon--Radin--Ren--Sadun regime of \Cref{thm:krrs-analytic-extension}
requires a strictly positive $H$-density excess $\tau-\eps^m$.  The theorem
below therefore produces, in addition to the quadratic expansion, the
real-analytic \emph{extension} of $G_r$ to a two-sided neighborhood of
$\delta=0$; it is that extension whose $\delta$-derivatives appear in
\Cref{sec:nonexceptional-proof}.

\begin{theorem}[Universal positivity of the scalar second variation]\label{thm:positive-second-variation}
After shrinking $I$ around $r_0$, there exist $\overline\delta>0$, an open
set $\mathcal N$ containing $I\times(-\overline\delta,\overline\delta)$,
and a real-analytic
function $G:\mathcal N\to\RR$ such that
\begin{equation}\label{eq:boundary-excess-extension}
  G(r,\delta)=G_r(\delta)
  \qquad(r\in I,\ 0<\delta<\overline\delta).
\end{equation}
Setting $A_H(r):=\partial_\delta^2G(r,0)$, the following hold.
\begin{enumerate}[label=(\roman*)]
  \item $G(r,0)=0$ and $\partial_\delta G(r,0)=0$ for every $r\in I$.
  \item $A_H$ is real-analytic on $I$ and satisfies $A_H(r)\ge a_0>0$
  there for some constant $a_0$.
  \item There is $C_2\ge1$, depending only on $H,I$, such that
  \begin{equation}\label{eq:boundary-excess-curvature}
    \bigl|\partial_\delta^2G(r,\delta)-A_H(r)\bigr|\le C_2|\delta|
    \qquad(r\in I,\ |\delta|<\overline\delta).
  \end{equation}
  \item Consequently the boundary excess has the expansion
  \begin{equation}\label{eq:boundary-excess-expansion}
    G_r(\delta)=\frac12A_H(r)\delta^2+O_{H,I}(\delta^3)
    \qquad(r\in I,\ 0<\delta<\overline\delta).
  \end{equation}
\end{enumerate}
\end{theorem}

\begin{proof}
Use the single analytic parameter map from
\Cref{thm:krrs-analytic-extension} chosen in the proof of
\Cref{lem:fixed-density-bipodality}.  Evaluating it at
\[
  (\eps,\vartheta)=(r-\delta,\,r^m-(r-\delta)^m)
\]
gives jointly analytic parameters $(q_{11},q_{12},q_{22},c)$ near $(r_0,0)$,
including small negative $\delta$.  Define
\[
  G(r,\delta):=
  c^2J_{\pc(r)}(q_{11})+2c(1-c)J_{\pc(r)}(q_{12})
  +(1-c)^2J_{\pc(r)}(q_{22})-J_{\pc(r)}(r).
\]
The three edge-density parameters and $\pc(r)$ stay in $(0,1)$ near
$(r_0,0)$, so $G$ is analytic there.  For $\delta>0$, this is the cost of
the fixed-density optimizer minus the constant-graphon cost, by
\eqref{eq:reduced-objective-attainment}.  At $\delta=0$, the parameters have
$c=0$ and $q_{22}=r$, giving $G(r,0)=0$.
Shrink $I$ and choose $\overline\delta>0$ so that
$\overline I\times[-\overline\delta,\overline\delta]$ lies in the analytic
domain and the preceding upper and lower bounds apply throughout the
positive half of this rectangle.

For $r\in I$ and $0<\delta<\overline\delta$, the fixed-density optimizer
satisfies \Cref{prop:graphon-quadratic-bound} and has no greater cost than
the graphon in \Cref{lem:bipodal-quadratic-bound}.  Consequently,
\[
  C_{d,I}\delta^2\le G_r(\delta)\le C_{H,I}\delta^2.
\]
Together with analyticity and $G(r,0)=0$, this gives
$\partial_\delta G(r,0)=0$ and
\[
  A_H(r)=\partial_\delta^2G(r,0)
  =\lim_{\delta\downarrow0}\frac{2G_r(\delta)}{\delta^2}>0.
\]
Thus $A_H$ is positive and analytic near $r_0$; after shrinking $I$,
continuity gives $A_H(r)\ge a_0:=A_H(r_0)/2>0$ for all $r\in I$.

The continuous function $|\partial_\delta^3G|$ is bounded on the closed
rectangle above.  Taking $C_2\ge1$ larger than this bound, the mean value
theorem gives \eqref{eq:boundary-excess-curvature}.  Taylor's theorem then
gives \eqref{eq:boundary-excess-expansion}, with remainder bounded by
$C_2|\delta|^3/6$.  All bounds are uniform for $r\in I$.
\end{proof}

The function $A_H$ does not depend on the chosen neighborhood or analytic
extension: any two extensions agree for small positive $\delta$, where
both equal \eqref{eq:boundary-excess}, and hence their second derivatives
at $\delta=0$ agree.  Since $r_0$ was arbitrary, these local functions
define a positive real-analytic function
\begin{equation}\label{eq:second-variation-domain}
  A_H:(0,1)\setminus\{r_*\}\longrightarrow(0,\infty).
\end{equation}

%% file: sections/nonexceptional_proof.tex
\subsection{Proof of the nonexceptional-endpoint theorem}\label{sec:nonexceptional-proof}

\begin{proof}[Proof of \Cref{thm:nonexceptional-endpoint}]
Let $I$ and $\rho_0$ be as in
\Cref{cor:scalar-reduction}.  After shrinking $I$ around $r_0$, let
$G$, $\overline\delta$, $a_0$, and $C_2$ be supplied by
\Cref{thm:positive-second-variation}.  Choose
\[
  0<\delta_0<
  \min\left\{\rho_0,\overline\delta,\frac{a_0}{2C_2},\inf_{r\in I}r\right\}.
\]
Then, uniformly for $r\in I$ and $|\delta|<\delta_0$,
\[
  \partial_\delta^2G(r,\delta)
  \ge A_H(r)-C_2|\delta|\ge\frac{a_0}{2}>0.
\]
Apply \Cref{cor:scalar-reduction} with $\rho=\delta_0$ to choose $\eta>0$.
For $r\in I$ and $\pc(r)-\eta<p<\pc(r)$,
it reduces the graphon problem to choosing
$\delta=r-e(W)\in(0,\delta_0)$.
The relative entropy identity \eqref{eq:relative-entropy-identity},
applied to $B_{r-\delta,r^m}$ and the constant graphon $r$, gives
\[
  I_{p,r}(r-\delta)-J_p(r)
  =G(r,\delta)-\lambda(p,r)\delta.
\]
Indeed, changing $p$ from $\pc(r)$ adds an affine function of the edge
density, and the difference of the two edge densities is $-\delta$.
Consequently,
\begin{equation}\label{eq:edge-deficit-minimization}
  \Phi_H(p,r)
  =J_p(r)+\min_{0<\delta<\delta_0}
    \{G(r,\delta)-\lambda(p,r)\delta\}.
\end{equation}
The correspondence in \Cref{cor:scalar-reduction} identifies its minimizers
with graphon optimizers up to relabeling.

To find the minimizing deficit, consider the critical-point equation
\[
  \Psi(p,r,\delta):=\partial_\delta G(r,\delta)-\lambda(p,r)=0.
\]
This function is analytic near $(\pc(r_0),r_0,0)$, where
\[
  \Psi(\pc(r_0),r_0,0)=0,
  \qquad
  \partial_\delta\Psi(\pc(r_0),r_0,0)=A_H(r_0)>0.
\]
The analytic implicit function theorem gives an analytic
solution $\delta_*(p,r)\in(-\delta_0,\delta_0)$ near $(\pc(r_0),r_0)$.
Shrink $I$ to $(r_0-\rho,r_0+\rho)$ and decrease $\eta>0$ if necessary
so that
\[
  U=\{(p,r):|p-\pc(r)|<\eta,\ |r-r_0|<\rho\}
\]
has compact closure in this analytic domain and satisfies $0<p<r<1$.
Since $\partial_\delta G(r,0)=\lambda(\pc(r),r)=0$, uniqueness gives
\[
  \delta_*(\pc(r),r)=0,
  \qquad
  \partial_\delta G(r,\delta_*(p,r))=\lambda(p,r).
\]

For $(p,r)\in U$ with $p<\pc(r)$, the mean value theorem yields
\[
  \lambda(p,r)
  =\partial_\delta^2G(r,\theta\delta_*)\delta_*
  \qquad (\text{for some }\theta\in(0,1)).
\]
The positive lower bound on the second derivative shows that
\[
  0<\delta_*(p,r)<\delta_0,
  \qquad
  \delta_*(p,r)\le\frac{2}{a_0}\lambda(p,r).
\]
Strict convexity now makes $\delta_*$ the unique minimizer in
\eqref{eq:edge-deficit-minimization}.  Moreover,
\eqref{eq:boundary-excess-curvature} gives
$\lambda(p,r)=A_H(r)\delta_*+O_{H,U}(\delta_*^2)$, uniformly on this region.
Since $A_H(r)\ge a_0$, it follows that
\begin{equation}\label{eq:optimal-deficit-expansion}
  \delta_*(p,r)
  =\frac{\lambda(p,r)}{A_H(r)}+O_{H,U}(\lambda(p,r)^2).
\end{equation}

By \Cref{cor:scalar-reduction}, the unique graphon optimizer is
$W_{p,r}:=B_{r-\delta_*,r^m}$ up to relabeling.  It is bipodal and
satisfies
\[
  e(W_{p,r})=r-\delta_*,
  \qquad
  t(H,W_{p,r})=r^m.
\]
It is nonconstant, since a constant graphon with $H$-density $r^m$ has
edge density $r$, whereas $\delta_*>0$.
Write $\widetilde c$ and $\widetilde q_{ij}$ for the single analytic
KRR--S parameter map chosen in \Cref{lem:fixed-density-bipodality}, and set
\[
  \eps(p,r):=r-\delta_*(p,r),
  \qquad
  \vartheta(p,r):=r^m-\eps(p,r)^m.
\]
The bipodal parameters are
\[
  c(p,r):=\widetilde c(\eps(p,r),\vartheta(p,r)),
  \qquad
  q_{ij}(p,r):=\widetilde q_{ij}(\eps(p,r),\vartheta(p,r))
  \quad((i,j)\in\{(1,1),(1,2),(2,2)\}).
\]
They are analytic in $(p,r)$, as are the edge density $r-\delta_*$ and
the minimum cost $J_p(r)+G(r,\delta_*)-\lambda(p,r)\delta_*$.

After decreasing $\eta$ if necessary, the uniform block-size bound in
\Cref{thm:krrs-analytic-extension} gives
\[
  c(p,r)=O_{H,U}(\vartheta(p,r))
  =O_{H,U}(\delta_*(p,r))
  =O_{H,U}(\lambda(p,r)).
\]
Here $c(p,r)>0$: otherwise the graphon would be constant with value
$\eps(p,r)<r$, contradicting its $H$-density $r^m$.
The function $\lambda$ vanishes at $p=\pc(r)$ and has bounded $p$-derivative
on $\overline U$, so decreasing $\eta$ makes $c(p,r)<1/2$ uniformly.
Thus the first block $A_{p,r}:=[0,c(p,r)]$ is the smaller one, and $c(p,r)\to0$ as
$p\uparrow\pc(r)$ for every fixed $r\in I$.

Finally, \eqref{eq:optimal-deficit-expansion} gives the edge-density
expansion in \eqref{eq:optimizer-expansion}.  For the cost, substitute it
into \eqref{eq:edge-deficit-minimization} and use
\eqref{eq:boundary-excess-expansion}:
\[
  \Phi_H(p,r)
  =J_p(r)+\frac12A_H(r)\delta_*^2-\lambda(p,r)\delta_*
    +O_{H,U}(\delta_*^3)
  =J_p(r)-\frac{\lambda(p,r)^2}{2A_H(r)}
    +O_{H,U}(\lambda(p,r)^3).
\]
The remainders are uniform on the symmetry-breaking part of $U$.
\end{proof}

%% file: sections/singular.tex
\section{Optimizers near the singular endpoint}
\label{sec:singular-endpoint}

In this section, we prove \Cref{thm:endpoint-optimizers} in a more precise form.
Along the Lubetzky--Zhao boundary, the supporting line to the graph of
$x\mapsto J_{\pc(r)}(x^{1/d})$ touches it at $x=r^d$ and
$x=\sm(r)^d$.  At the exceptional density $r_*=(d-1)/d$, these two
points coincide at $x=r_*^d$, so the nonexceptional reduction no longer
applies.  Following the strategy in
\Cref{sec:proof-strategy}, we construct an analytic rank-one bipodal family
and establish its global optimality and uniqueness along the associated
parameter curve approaching $(\pc(r_*),r_*)$.

Fix a finite simple $d$-regular graph $H$ with $d\ge2$,
and write $m:=|E(H)|$ and $v:=|V(H)|=2m/d$.  Set
\begin{equation*}
  r_*=\frac{d-1}{d},
  \qquad
  p_*=\pc(r_*)
  =\frac{d-1}{(d-1)+\exp(d/(d-1))},
  \qquad
  u_*:=\sqrt{r_*},
  \qquad
  P_*:=(p_*,r_*).
\end{equation*}
We write $\lambda$ for Lebesgue measure on $[0,1]$, and also for its
product measure on $[0,1]^2$.
In this section, a kernel means a measurable real-valued function of two
variables; unlike a graphon, it need not be symmetric or take values in
$[0,1]$.
For a measurable function $g:[0,1]\to\mathbb R$, write $g\otimes g$ for
the rank-one kernel $(g\otimes g)(x,y):=g(x)g(y)$.
The following theorem is the main result of this section.
The remainder of this section is devoted to proving this theorem.

\begin{theorem}[Optimizers near the singular endpoint]
\label{thm:singular-endpoint}
There exist $h_0>0$ and real-analytic functions
$p_h,r_h,u_h,\alpha_h$, defined for $|h|<h_0$, such that
\[
  (p_0,r_0)=P_*,
  \qquad
  u_0=u_*,
  \qquad
  \alpha_0=\frac12.
\]
For every $h\in(0,h_0)$, we have $p_h<\pc(r_h)$,
and for any measurable set $A_h\subset[0,1]$
with $\lambda(A_h)=\alpha_h$, set
\[
  s_h:=u_h-h,
  \qquad
  t_h:=u_h+h,
  \qquad
  f_h:=s_h\1_{A_h}+t_h\1_{A_h^c},
  \qquad
  W_h:=f_h\otimes f_h.
\]
Then $W_h$ is a nonconstant rank-one bipodal graphon satisfying $t(H,W_h)=r_h^m$
and is the unique minimizer of \eqref{eq:graphon-variational-problem}
at $(p_h,r_h)$, up to relabeling.  Moreover,
\begin{equation*}
  e(W_h)=r_h-(d-1)h^2+O_d(h^4),
  \qquad
  \Phi_H(p_h,r_h)
  =J_{p_h}(r_h)-\frac{d^3}{3}h^4+O_d(h^6).
\end{equation*}
\end{theorem}

\input{sections/singular_stationary_family}
\input{sections/singular_constant_comparison}

\input{sections/singular_localization}
\input{sections/singular_auxiliary_lagrangian}
\input{sections/singular_graphon_comparison}
\input{sections/singular_proof}

%% file: sections/singular_stationary_family.tex
\subsection{The rank-one stationary family near the exceptional endpoint}
\label{sec:rank-one-stationary-family}

We now construct an analytic family of rank-one bipodal
candidates.  The numerical exploration in \Cref{fig:bipodal-parameter-maps}
suggests that rank-one graphons describe the optimizers along a curve of
parameter pairs $(p,r)$ approaching $P_*$.  Motivated by this observation,
we seek rank-one solutions of the full graphon stationarity equations,
together with the block-proportion stationarity condition.
At this stage, the resulting family consists only of stationary
rank-one bipodal candidates; the subsequent arguments establish their optimality and
uniqueness along the corresponding parameter curve.

Throughout this subsection, graphons are assumed to take values in a compact
subinterval of $(0,1)$.  For a rank-one graphon $W=f\otimes f$ with
measurable $f:[0,1]\to[0,1]$, this requires that, for some $\eta>0$,
\[
  \eta\le f\le1-\eta
  \qquad\text{a.e.}
\]
This interiority condition is used only to construct the candidate family in
this subsection.  It ensures that $W+\varepsilon U$ remains a graphon for every
bounded symmetric function $U$ and all sufficiently small $|\varepsilon|$, so
that stationarity can be tested in every bounded direction.  The subsequent
optimality and uniqueness arguments impose no such restriction on competing
graphons.  We say that $W$ is \emph{stationary} for
$I_p-\mu t(H,\,\cdot\,)$ if
\begin{equation}\label{eq:graphon-stationarity}
  \left.\frac{d}{d\varepsilon}\right|_{\varepsilon=0}
  \left[
    I_p(W+\varepsilon U)
    -\mu t(H,W+\varepsilon U)
  \right]=0
\end{equation}
for every such $U$.
For a rank-one graphon $W=f\otimes f$, we say that $W$ \emph{satisfies the
KKT conditions} at $(p,r)$ with multiplier $\mu$, if \eqref{eq:graphon-stationarity}
holds with $\mu\ge0$ and $t(H,W)=r^m$.  This is the graphon
analog of the usual first-order Karush--Kuhn--Tucker condition for
constrained finite-dimensional optimization.

To formulate the separate first-order condition for block proportion,
for $\alpha\in(0,1)$ and $s,t\in(0,1)$, set
\[
  f_{\alpha,s,t}
  :=s\1_{[0,\alpha]}+t\1_{(\alpha,1]},
  \qquad
  W_{\alpha,s,t}:=f_{\alpha,s,t}\otimes f_{\alpha,s,t}.
\]
Changing $\alpha$ moves the boundary between the two vertex classes.  Unlike
the perturbations $W_{\alpha,s,t}+\varepsilon U$ used in graphon stationarity,
replacing $\alpha$ by $\alpha+\varepsilon$, when $s\ne t$, changes the graphon
by an amount bounded away from zero on a set whose measure tends to zero.
Hence this variation is not small in $L^\infty$ and is not covered by graphon
stationarity.  We therefore impose the separate condition
\begin{equation}\label{eq:block-stationarity}
  \left.\frac{d}{d\beta}\right|_{\beta=\alpha}
  \left[
    I_p(W_{\beta,s,t})-\mu t(H,W_{\beta,s,t})
  \right]=0,
\end{equation}
obtained by varying the block proportion while keeping $s$ and $t$ fixed.

The following lemma makes a structural reduction: it shows that bipodality is
forced by the KKT conditions within the interior rank-one class.  The
construction lemma later begins by seeking a bipodal family; this structural
result justifies that choice and removes the need to assume bipodality in
the accompanying local uniqueness assertion.  The proof reduces the stationarity condition
\eqref{eq:graphon-stationarity} to a scalar equation and bounds the
number of its roots.  For this reduction, define the \emph{$d$-th moment} of
$f$ by
\[
  q(f):=\int_0^1 f(x)^d\,dx.
\]
Note that any rank-one graphon $W=f\otimes f$ satisfies
$t(H,W)=q(f)^v$ depending only on $H$ through $d$ and $v$:
\[
  t(H,f\otimes f)
  =\int_{[0,1]^v}\prod_{(i,j)\in E(H)}f(x_i)f(x_j)\,dx_1\cdots dx_v
  =\int_{[0,1]^v}\prod_{i=1}^v f(x_i)^{d}\,dx_1\cdots dx_v
  =q(f)^v.
\]

\begin{lemma}[Bipodality of stationary rank-one graphons]
\label{lem:stationary-rank-one-bipodality}
Let $p,r\in(0,1)$, let $\mu\ge0$, and let
$f:[0,1]\to[0,1]$ be measurable.  Suppose that, for some $\eta>0$,
\[
  \eta\le f\le1-\eta
  \qquad\text{a.e.},
\]
and that $f\otimes f$ satisfies the KKT conditions at $(p,r)$ with multiplier
$\mu$.  If $f$ is not a.e.\ constant, then there exist $\alpha\in(0,1)$ and
$0<s<t<1$ such that, after a measure-preserving relabeling,
\[
  f=f_{\alpha,s,t}
  \qquad\text{a.e.}
\]
\end{lemma}

\begin{proof}
Set $\gamma:=\mu m q(f)^{v-2}$ and $F_{p,\gamma}(z):=J_p'(z)-\gamma z^{d-1}$.
We first claim that
\begin{equation}
\label{eq:rank-one-kkt}
  F_{p,\gamma}(f(x)f(y))=0
  \quad\text{for a.e. }(x,y)\in[0,1]^2.
\end{equation}
Let $U$ be any bounded symmetric function on $[0,1]^2$ and set
$W_\varepsilon=f\otimes f+\varepsilon U$.  The variation of $I_p$ is
\[
  \left.\frac{d}{d\varepsilon}I_p(W_\varepsilon)\right|_{\varepsilon=0}
  =\iint J_p'(f(x)f(y))U(x,y)\,dx\,dy.
\]
Since $H$ is $d$-regular, differentiating one of the $m$ edge factors in
$t(H,W_\varepsilon)$ gives
\[
  \left.\frac{d}{d\varepsilon}t(H,W_\varepsilon)\right|_{\varepsilon=0}
  =\iint m q(f)^{v-2}f(x)^{d-1}f(y)^{d-1}
  U(x,y)\,dx\,dy.
\]
Stationarity for every $U$ therefore yields
\[
  J_p'(f(x)f(y))
  =\mu m q(f)^{v-2}f(x)^{d-1}f(y)^{d-1}
  =\gamma\bigl(f(x)f(y)\bigr)^{d-1}
\]
for almost every $(x,y)$, proving
\eqref{eq:rank-one-kkt}.

We next show that $F_{p,\gamma}(z)=0$ has at most three
solutions for $z\in(0,1)$.  Since $J_p''(z)=1/(z(1-z))$,
\[
  F_{p,\gamma}'(z)
  =\frac1{z(1-z)}-\gamma(d-1)z^{d-2}.
\]
If $\gamma\le0$, then $F_{p,\gamma}'>0$ on $(0,1)$, so
$F_{p,\gamma}$ has at most one zero.  If $\gamma>0$, the critical-point
equation is
\[
  z^{d-1}(1-z)=\frac1{(d-1)\gamma}.
\]
The function $z\mapsto z^{d-1}(1-z)$ is strictly increasing on
$(0,(d-1)/d)$ and strictly decreasing on $((d-1)/d,1)$.  Hence
$F_{p,\gamma}'$ has at most two distinct zeros,
and by Rolle's theorem, $F_{p,\gamma}$ has at most three distinct zeros.

By Fubini's theorem, one can choose $y_0\in[0,1]$ such that
\[
  F_{p,\gamma}(f(x)f(y_0))=0
  \qquad\text{for a.e. }x.
\]
The lower bound on $f$ gives $f(y_0)>0$.  Hence multiplication by
$f(y_0)$ is injective, and the three-root bound for $F_{p,\gamma}$ implies
that $f$ takes at most three values up to null sets.

Suppose that it takes three distinct values $0<a<b<c<1$ up to null sets.
Since \eqref{eq:rank-one-kkt} holds outside a null subset
of $[0,1]^2$, each of
\[
  a^2,\qquad ab,\qquad ac,\qquad bc,\qquad c^2
\]
is a zero of $F_{p,\gamma}$.  They satisfy
$a^2<ab<ac<bc<c^2$, contradicting the three-root bound.  Hence $f$ takes at
most two values.
Therefore, if $f$ is not a.e. constant, it has the form
\[
  f=s\1_A+t\1_{A^c}
\]
for some $0<s<t<1$ and measurable $A$ with $0<\lambda(A)<1$.  Setting
$\alpha:=\lambda(A)$ and relabeling $A$ as $[0,\alpha]$ gives
$f=f_{\alpha,s,t}$ almost everywhere.
\end{proof}

\Cref{lem:stationary-rank-one-bipodality} reduces the problem to a finite-dimensional
one: after a measure-preserving relabeling, the factor of every nonconstant
rank-one graphon satisfying the KKT conditions has the form
$f_{\alpha,s,t}$ for some $\alpha\in(0,1)$ and $0<s<t<1$.  It does not,
however, construct such graphons near $P_*$ or determine their parameters.
The next lemma constructs this analytic family.
Reparametrize $s$ and $t$ by
\[
  u:=\frac{s+t}{2},
  \qquad
  h:=\frac{t-s}{2},
  \qquad
  \gamma:=\mu m q(f_{\alpha,s,t})^{v-2}.
\]
Thus, $u$ is the midpoint of the two factor values, $|h|$ is half their
separation, and $\gamma$ is the coefficient appearing in the scalar
stationarity equation \eqref{eq:rank-one-kkt}.
The three stationarity equations, evaluated at the three graphon values,
determine $u,p,\gamma$; the block-proportion equation then determines
$\alpha$, and the defining relation for $\gamma$ determines the multiplier $\mu$.
Changing $h$ to $-h$ exchanges the two factor values and their proportions,
while remaining quantities are unchanged.

We denote the value of the coefficient $\gamma$ at the singular endpoint by
\[
  \gamma_*:=\left(\frac{d}{d-1}\right)^d=r_*^{-d}.
\]

\begin{lemma}[Analytic rank-one KKT family]
\label{lem:rank-one-kkt-family}
There exist $h_0>0$ and real-analytic functions
\[
  h\longmapsto(u_h,p_h,\gamma_h,\alpha_h)
  \qquad (|h|<h_0)
\]
such that
\[
  (u_0,p_0,\gamma_0,\alpha_0)
  =\left(u_*,p_*,\gamma_*,\frac12\right),
\]
and the following properties hold.  First, for all $|h|<h_0$,
\[
  u_{-h}=u_h,\qquad p_{-h}=p_h,\qquad
  \gamma_{-h}=\gamma_h,\qquad \alpha_{-h}=1-\alpha_h.
\]
Next, for $|h|<h_0$, define the two-valued factor and its rank-one graphon by
\[
  s_h:=u_h-h,
  \qquad 
  t_h:=u_h+h,
  \qquad
  f_h:=f_{\alpha_h,s_h,t_h}, 
  \qquad
  W_h:=f_h\otimes f_h,
\]
and define the $d$-th moment of $f_h$, the associated target density, and the
Lagrange multiplier by
\[
  q_h:=q(f_h)=\alpha_hs_h^d+(1-\alpha_h)t_h^d,
  \qquad
  r_h:=q_h^{2/d}, 
  \qquad
  \mu_h:=\frac{\gamma_h}{m q_h^{v-2}}.
\]
The resulting functions $q_h,r_h,\mu_h$ are real analytic, and for all
$|h|<h_0$, these quantities satisfy
\[
  0<p_h<r_h<1,
  \qquad s_h,t_h\in(0,1),
  \qquad \alpha_h\in(0,1),
  \qquad \mu_h>0,
  \qquad\text{and}\qquad
  t(H,W_h)=r_h^m.
\]
Finally, for $0<|h|<h_0$, the nonconstant rank-one graphon $W_h$ satisfies
the KKT conditions at $(p_h,r_h)$ with multiplier $\mu_h$, and
\eqref{eq:block-stationarity} holds.  Conversely, in the precise
sense stated below, let $W=f\otimes f$ be a nearby nonconstant rank-one
graphon whose factor $f$ is bounded away from $0$ and $1$.  If $W$ satisfies
the KKT conditions with multiplier $\mu$ and
\eqref{eq:block-stationarity} holds, then $W$ agrees, up to
measure-preserving relabeling, with a unique member $W_h$ indexed by
$h\in(0,h_0)$, and $\mu=\mu_h$.
\end{lemma}

The graphons $W_h$ supplied by \Cref{lem:rank-one-kkt-family},
with their measure-preserving relabelings, form the family used in
\Cref{thm:singular-endpoint}.  Before proving the lemma, we clarify the
family at the singular endpoint and the local uniqueness assertion.
At $h=0$, the definitions in the lemma give
\[
  s_0=t_0=u_*,
  \qquad
  q_0=u_*^d,
  \qquad
  r_0=u_*^2=r_*,
  \qquad
  \mu_0=\frac1{m r_*^m},
  \qquad
  W_0\equiv r_*.
\]
Thus, $(p_0,r_0)=P_*$, so the family passes through the constant graphon
$W_0\equiv r_*$ at the singular endpoint.  Also, the two factor values $s_0$ and
$t_0$ coincide, as do the three values $s_0^2,s_0t_0,t_0^2$ taken by the
graphon $W_0$.
For the local uniqueness assertion, consider a graphon $W=f\otimes f$ and a
multiplier $\mu$ appearing in the converse assertion of the lemma.  By
\Cref{lem:stationary-rank-one-bipodality}, after a measure-preserving relabeling its
factor has the form $f_{\alpha,s,t}$ with
$\alpha\in(0,1)$ and $0<s<t<1$.  If
$(s,t,p,r,\mu,\alpha)$ is sufficiently close to
\[
  \left(u_*,u_*,p_*,r_*,\frac1{m r_*^m},\frac12\right),
\]
then $h:=(t-s)/2$ belongs to $(0,h_0)$ and
\[
  (s,t,p,r,\mu,\alpha)
  =(s_h,t_h,p_h,r_h,\mu_h,\alpha_h).
\]

\begin{proof}[Proof outline]
Write
\[
  \ell(z):=\log\frac{1-z}{z},
  \qquad
  \ell_*:=\ell(p_*)=\frac{d}{d-1}-\log(d-1).
\]
By \Cref{lem:stationary-rank-one-bipodality}, every nonconstant rank-one graphon
satisfying the KKT conditions has, after a measure-preserving relabeling, a
two-valued factor.  Motivated by this reduction, we seek a family of the form
\[
  s:=u-h,
  \qquad
  t:=u+h,
  \qquad
  f:=f_{\alpha,s,t},
  \qquad
  W:=f\otimes f.
\]
Since $f\otimes f$ takes the three values $s^2,st,t^2$, the scalar
stationarity equation \eqref{eq:rank-one-kkt} requires
\begin{equation}\label{eq:three-value-kkt}
  \ell(p)-\ell(s^2)=\gamma s^{2d-2},\qquad
  \ell(p)-\ell(st)=\gamma(st)^{d-1},\qquad
  \ell(p)-\ell(t^2)=\gamma t^{2d-2}.
\end{equation}
We therefore begin by solving this necessary system for $u,p,\gamma$ as
functions of $h$; the block proportion $\alpha$ does not occur and will be
determined later.
At $h=0$, these equations have the solution $(u,p,\gamma)=(u_*,p_*,\gamma_*)$.
For $h\ne0$, subtracting consecutive equations in \eqref{eq:three-value-kkt} eliminates
$\ell(p)$ and expresses $\gamma$ as each of the two divided slopes across
$s^2,st,t^2$.  Although each slope appears as $0/0$ at $h=0$, both extend
analytically.  Let $A(u,h)$ denote the difference of the two divided slopes,
so the system requires $A(u,h)=0$.  The symmetry $h\mapsto-h$ makes $A$ odd in $h$,
and hence $A(u,h)=hB(u,h)$ for some analytic $B$.  A direct calculation shows that
$B(u,0)$ has the unique nondegenerate zero $u=u_*$.  The analytic implicit
function theorem therefore gives $h_0>0$ and a unique even analytic solution
$u=u_h$ for $|h|<h_0$ satisfying $B(u_h,h)=A(u_h,h)=0$.  The common divided slope then defines the even function
$\gamma=\gamma_h$, and the middle equation in
\eqref{eq:three-value-kkt} determines the even function $p=p_h$.

We now choose the block proportion $\alpha$ as a function of $h$.
Once $\alpha$ is fixed, set
\[
  q:=\alpha s_h^d+(1-\alpha)t_h^d,
  \qquad
  r:=q^{2/d},
  \qquad
  \mu:=\frac{\gamma_h}{m q^{v-2}},
  \qquad
  f:=f_{\alpha,s_h,t_h},
  \qquad
  W:=f\otimes f.
\]
Also define the function
\[
  \mathcal M_h(z):=J_{p_h}(z)-\frac{\gamma_h}{d}z^d
  \qquad (0 < z < 1).
\]
Evaluating the derivative in \eqref{eq:block-stationarity} and
solving the resulting linear equation gives
\begin{equation}\label{eq:block-proportion-formula}
  \alpha_h:=
  \frac{\mathcal M_h(t_h^2)-\mathcal M_h(s_ht_h)}
  {\left\{\mathcal M_h(t_h^2)-\mathcal M_h(s_ht_h)\right\}+
  \left\{\mathcal M_h(s_h^2)-\mathcal M_h(s_ht_h)\right\}}.
\end{equation}
To analyze this quotient as $h\to0$, we factor $\mathcal M_h'$.  By
\eqref{eq:three-value-kkt}, the three values
$s_h^2,s_ht_h,t_h^2$ are zeros of $\mathcal M_h'$, and
\begin{equation}\label{eq:mh-derivative-factorization}
  \mathcal M_h'(z)
  =(z-s_h^2)(z-s_ht_h)(z-t_h^2)R(h,z),
\end{equation}
where $R$ is jointly analytic near $(0,r_*)$.
At $h=0$, the three zeros coincide at $r_*$, and direct calculation gives
\[
  \mathcal M_0'(z)
  =k_3(z-r_*)^3+O_d((z-r_*)^4),
  \qquad
  k_3:=\frac{d^5}{6(d-1)^2}>0.
\]
Thus $R(0,r_*)=k_3$.  Since every $z\in[s_h^2,t_h^2]$ satisfies
$|z-r_*|=O_d(|h|)$, joint analyticity of $R$ gives
\[
  R(h,z)=k_3+O_d(|z-r_*|+|h|)=k_3+O_d(|h|).
\]
Substituting this into \eqref{eq:mh-derivative-factorization}
and integrating over $[s_h^2,s_ht_h]$ and $[s_ht_h,t_h^2]$,
and using the expansions of $s_h$ and $t_h$ in powers of $h$,
show that the two differences in \eqref{eq:block-proportion-formula} vanish to order exactly four
and have the same leading term.  Canceling their common factor $h^4$ extends
the quotient analytically to $h=0$, where it equals $1/2$.  Finally,
$h\mapsto-h$ exchanges the two differences, so
$\alpha_{-h}=1-\alpha_h$.  The resulting analytic function $\alpha_h$ then
determines the analytic functions $q_h,r_h,\mu_h$, and the rank-one
factors $f_h$ and the bipodal graphons $W_h$.

The definitions of $q_h,r_h,\mu_h$ now give
$t(H,W_h)=q_h^v=r_h^m$.  The three equations
\eqref{eq:three-value-kkt} verify the stationarity condition
\eqref{eq:graphon-stationarity}, while
\eqref{eq:block-proportion-formula} verifies
\eqref{eq:block-stationarity}.
After decreasing $h_0$, continuity gives $0<p_h<r_h<1$,
$s_h,t_h,\alpha_h\in(0,1)$, and $\gamma_h,q_h>0$, hence $\mu_h>0$.
Thus, $(p_h,r_h)$ is admissible and $W_h$ is an interior bipodal graphon.

Finally, \Cref{lem:stationary-rank-one-bipodality} reduces any nearby rank-one graphon family in the
converse assertion to the same bipodal form that we selected in the beginning of the proof.
Repeating the construction, uniqueness in the implicit function theorem and in the linear
block-proportion equation forces the graphon, up to relabeling, to be
$W_h$ for a unique $h>0$, with multiplier $\mu_h$.
Full details are given in
\Cref{app:rank-one-kkt-family}.
\end{proof}

\begin{remark}[Dependence on $d$ and $H$]\label{rmk:rank-one-family-universality}
The equations \eqref{eq:three-value-kkt} and \eqref{eq:block-proportion-formula} depend only on $d$.
Thus, for fixed $d$, the local parameter curve $(p_h,r_h)$ and graphon
family $W_h$, up to relabeling, are independent of the particular
$d$-regular graph $H$.  The multiplier $\mu_h=\gamma_h/(m q_h^{v-2})$
and the interval on which we will establish global optimality and uniqueness
in the subsequent sections may depend on $H$.
\end{remark}

%% file: sections/singular_constant_comparison.tex
\subsection{Comparison with the constant graphon}
\label{sec:constant-graphon-comparison}

The constant graphon $W\equiv r_h$ has $H$-density $r_h^m$ and is therefore
feasible for the same constraint as $W_h$, where $W_h$ is the member of the
analytic rank-one KKT family constructed in
\Cref{lem:rank-one-kkt-family}.  We will prove that
\[
  I_{p_h}(W_h)<I_{p_h}(W\equiv r_h)=J_{p_h}(r_h),
\]
and compute the asymptotic expansion of the
relative entropy difference for all sufficiently small $h>0$.  At $h=0$, the
two graphons coincide, so their $I_{p_h}$-values are equal.  The symmetries in
\Cref{lem:rank-one-kkt-family} make the relative entropy difference
an even analytic function of $h$, and the calculation below shows that its
leading term is of order $h^4$.  The sign of the difference for small $h>0$ is therefore
determined by the coefficient of $h^4$.  To compute that coefficient, we
first record the expansions of the parameters that will be used below.
Write
\[
  \ell(z):=\log\frac{1-z}{z}, \qquad
  \ell_*:=\ell(p_*)=\frac d{d-1}-\log(d-1), \qquad
  \Lambda_h:=\ell(p_h)-\ell_*.
\]
Here, $\Lambda_h$ represents the slope of the affine difference
\begin{equation}\label{eq:entropy-parameter-shift}
  J_{p_h}(z)-J_{p_*}(z)=\Lambda_h z+J_{p_h}(0)-J_{p_*}(0).
\end{equation}

\begin{lemma}[Expansions along the rank-one KKT family]
\label{lem:rank-one-parameter-expansions}
The analytic family constructed in
\Cref{lem:rank-one-kkt-family} satisfies
\begin{equation}\label{eq:rank-one-parameter-expansions}
\begin{aligned}
  u_h&=u_*+\frac{5-3d}{6u_*}h^2+O_d(h^4),
  &&&
  \Lambda_h&=\frac{2d^3}{3(d-1)}h^2+O_d(h^4),\\
  \alpha_h&=\frac12+\frac{2d^2u_*}{3(d-1)}h+O_d(h^3),
  &&&
  r_h&=r_*-\frac{2(4d-1)}3h^2+O_d(h^4).
\end{aligned}
\end{equation}
\end{lemma}

These expansions immediately locate the family relative to
$P_*=(p_*,r_*)$.  For all sufficiently small $h>0$, the expansion of $r_h$
gives $r_h<r_*$.  Moreover, since $\ell$ is strictly decreasing,
$\Lambda_h>0$ implies $p_h<p_*$.

\begin{proof}[Proof outline]
The symmetries in \Cref{lem:rank-one-kkt-family} show that
$u_h,\gamma_h$, and $\Lambda_h$ have expansions in even powers of $h$, while
$\alpha_h-1/2$ has an expansion in odd powers.  Let $U_2, G_2, L_2$
denote the quadratic coefficients of $u_h,\gamma_h$, and $\Lambda_h$,
respectively, and substitute the expansions of $s_h^2,s_ht_h,t_h^2$ into
\eqref{eq:three-value-kkt}.
The middle equation at order $h^2$ relates $G_2$ and $L_2$,
and either outer equation at order $h^3$ determines $G_2$,
and the comparison at order $h^4$ relates $U_2$ and $G_2$.
Solving these three equations gives the first two expansions in
\eqref{eq:rank-one-parameter-expansions}.

Now we determine the expansions of $\alpha_h$ and $r_h$.
Since $R(h,z)$ is even in $h$, its Taylor expansion gives
\[
  R(h,z)=k_3+k_4(z-r_*)+O_d(h^2)
  \quad (z\in[s_h^2,t_h^2]),
  \qquad
  k_4:=\frac{5d^6(d-2)}{24(d-1)^3}.
\]
Using this expansion and integrating the factorization
\eqref{eq:mh-derivative-factorization} over $[s_h^2,s_ht_h]$ and
$[s_ht_h,t_h^2]$ gives the two differences in \eqref{eq:block-proportion-formula}
through order $h^5$.  Their asymmetry at order $h^5$ determines the linear
term of $\alpha_h$.  The expansions of $q_h$ and $r_h$ then follow from
their definitions.
Full details are given in \Cref{app:rank-one-parameter-expansions}.
\end{proof}

We now turn to the comparison stated at the beginning of this subsection.

\begin{lemma}[Comparison with the constant graphon]
\label{lem:constant-graphon-comparison}
For sufficiently small $h>0$,
\[
  e(W_h)=r_h-(d-1)h^2+O_d(h^4),
  \qquad
  I_{p_h}(W_h)=J_{p_h}(r_h)-\frac{d^3}{3}h^4+O_d(h^6).
\]
In particular, the constant graphon
$W\equiv r_h$ is feasible but has strictly larger $I_{p_h}$-value,
and consequently $p_h<\pc(r_h)$ for all sufficiently small $h>0$.
\end{lemma}

The affine identity \eqref{eq:entropy-parameter-shift} separates the change in $p$
from the comparison at $p_*$:
\[
  I_{p_h}(W_h)-J_{p_h}(r_h)
  =\{I_{p_*}(W_h)-J_{p_*}(r_h)\}+\Lambda_h\{e(W_h)-r_h\}.
\]
The second term is evaluated using \Cref{lem:rank-one-parameter-expansions}.
For the first difference, we return to the supporting-line geometry of
\Cref{thm:scalar-lz-boundary}.  Its proof obtains the quadratic separation
in \eqref{eq:contact-quadratic-separation} by Taylor-expanding the difference between
$\varphi_{\pc(r),d}(x)=J_{\pc(r)}(x^{1/d})$ and its tangent line $\ell_r(x)$
near the two contact points.  We use the same approach at the exceptional
endpoint, where these contacts coincide.

By the continuous extensions in \Cref{thm:scalar-lz-boundary},
$\pc(r)\to p_*$ and $\sm(r)\to r_*$ as $r\to r_*$.  The tangent lines
$\ell_r$ therefore extend continuously to the tangent of
$\varphi_{p_*,d}$ at $x=r_*^d$:
\[
  \ell_{r_*}(x)=J_{p_*}(r_*)+\beta_d(x-r_*^d).
\]
Here $\beta_d$ denotes the slope of this tangent, which the chain rule gives as
\[
  \beta_d:=\varphi_{p_*,d}'(r_*^d)
  =\frac{J_{p_*}'(r_*)}{d r_*^{d-1}}
  =\frac{\gamma_*}{d}
  =\frac{d^{d-1}}{(d-1)^d}.
\]
Following \eqref{eq:supporting-cost-gap}, define the \emph{supporting cost gap} by
\begin{equation}\label{eq:endpoint-supporting-gap}
  \Gamma_d(z)
  :=\varphi_{p_*,d}(z^d)-\ell_{r_*}(z^d)
  =J_{p_*}(z)-J_{p_*}(r_*)-\beta_d\{z^d-r_*^d\}
  \qquad(0\le z\le1).
\end{equation}
Thus, $\Gamma_d$ is the continuous extension to $r=r_*$ of the function
$g_r$ used there.  Its Taylor expansion describes the separation from
the tangent at their common contact point; as shown below, the leading term is quartic.

By \eqref{eq:endpoint-supporting-gap}, we can express the cost difference
at $p_*$ in terms of $\Gamma_d$ as
\[
  I_{p_*}(W_h)-J_{p_*}(r_h)
  =\int\Gamma_d(W_h)-\Gamma_d(r_h)+\beta_d\left(\int W_h^d-r_h^d\right).
\]
Adding the affine correction from $p_*$ to $p_h$ gives
\begin{equation}\label{eq:constant-comparison-decomposition}
  I_{p_h}(W_h)-J_{p_h}(r_h)
  =\int\Gamma_d(W_h)-\Gamma_d(r_h)
  +\beta_d\left(\int W_h^d-r_h^d\right)
  +\Lambda_h\{e(W_h)-r_h\}.
\end{equation}
Since $\int W_h^d=(\int f_h^d)^2=q_h^2=r_h^d$, the $d$-th moment term vanishes, and
\begin{equation}\label{eq:constant-comparison-identity}
  I_{p_h}(W_h)-J_{p_h}(r_h)
  =\int\Gamma_d(W_h)-\Gamma_d(r_h) +\Lambda_h\{e(W_h)-r_h\}.
\end{equation}
The proof first evaluates $\Gamma_d$ by its quartic expansion,
then expands all the parameters using \Cref{lem:rank-one-parameter-expansions}.

\begin{proof}
Tangency gives $\Gamma_d(r_*)=\Gamma_d'(r_*)=0$.  At the exceptional
endpoint, direct differentiation further gives
$\Gamma_d''(r_*)=\Gamma_d'''(r_*)=0$ and
$\Gamma_d^{(4)}(r_*)=d^5/(d-1)^2>0$.
Taylor's theorem therefore yields, as $z\to r_*$,
\begin{equation}\label{eq:endpoint-gap-expansion}
  \Gamma_d(z)
  =\frac{d^5}{24(d-1)^2}(z-r_*)^4+O_d(|z-r_*|^5).
\end{equation}
By the parameter expansions in \eqref{eq:rank-one-parameter-expansions},
\[
  s_h^2-r_*=-2u_*h+O_d(h^2),
  \qquad
  t_h^2-r_*=2u_*h+O_d(h^2),
  \qquad
  s_ht_h-r_*=O_d(h^2),
  \qquad
  r_h-r_*=O_d(h^2).
\]
Each of the four rectangles determined by $A_h$ and $A_h^c$ in
$[0,1]^2$ has area $1/4+O_d(h)$.
Thus the quartic expansion and $u_*^2=(d-1)/d$ yield
\begin{equation}\label{eq:family-gap-expansion}
  \int\Gamma_d(W_h)=\frac{d^3}{3}h^4+O_d(h^6),
  \qquad
  \Gamma_d(r_h)=O_d(h^8).
\end{equation}
The first remainder improves from $O_d(h^5)$ to $O_d(h^6)$
because $\int\Gamma_d(W_h)$ is an even analytic function of $h$.
Finally, $e(W_h)=\{u_h+h(1-2\alpha_h)\}^2$ and the parameter expansions give
\[
  e(W_h)-r_h=-(d-1)h^2+O_d(h^4),
  \qquad
  \Lambda_h\{e(W_h)-r_h\}=-\frac{2d^3}{3}h^4+O_d(h^6).
\]
Substituting these expansions into \eqref{eq:constant-comparison-identity}
proves the lemma.
\end{proof}

%% file: sections/singular_localization.tex
\subsection{Localization and rank-one reduction}
\label{sec:localization-rank-one}

We now take any feasible graphon $W$ satisfying
$I_{p_h}(W)\le I_{p_h}(W_h)$ and prepare it for comparison with $W_h$.
This has two stages.  First, the $I_{p_h}$-comparison forces $W$ to be close
in $L^4$ to the constant graphon $W\equiv r_*$; this is the localization
step.  Second, we choose a rank-one part $f\otimes f$ and write
$W=f\otimes f+E$.  The first-stage $L^4$-localization then shows that the factor
$f$ is close to $u_*$ and that the error $E$ is small.  The lower bound for
the auxiliary Lagrangian difference obtained in
\Cref{sec:auxiliary-lagrangian} will control the distributions of
the values of $f$, while the arbitrary-graphon comparison in
\Cref{sec:graphon-comparison} will control $E$.
Thus, the decomposition is the form in which the localization is used in the
remainder of the proof.

For the first stage, we use the supporting cost gap $\Gamma_d$ from
\eqref{eq:endpoint-supporting-gap}.  The mechanism is to bound
$\int\Gamma_d(W)$ using the $I_{p_h}$-comparison and then convert that bound
into an $L^4$-bound on $W-r_*$.  Its quartic expansion gives the required
control near $r_*$; to extend this control to all of $[0,1]$, we also need
strict positivity away from $r_*$.  This follows from the supporting-line
interpretation above: by \Cref{lem:convexity-defect}, $\varphi_{p_*,d}$ is
convex on $[0,1]$ and its tangent $\ell_{r_*}$ has no other contact.
Therefore,
\begin{equation}\label{eq:endpoint-gap-positivity}
  \Gamma_d(z)\ge0,
  \qquad
  \Gamma_d(z)=0\iff z=r_*.
\end{equation}
This makes $\int\Gamma_d(W)$ a nonnegative measure of how far $W$ is from
the constant graphon $W\equiv r_*$.  For $L^4$-localization, we use the
stronger pointwise bound
\begin{equation}\label{eq:endpoint-gap-quartic-bound}
  \Gamma_d(z)\ge c_d|z-r_*|^4
  \qquad(0\le z\le1)
\end{equation}
for some $c_d>0$.  Indeed, the ratio $\Gamma_d(z)/|z-r_*|^4$ has a positive
limit at $r_*$ by \eqref{eq:endpoint-gap-expansion}, hence a
positive lower bound in a neighborhood of $r_*$.  On the compact complement
of this neighborhood, it has a positive minimum by continuity and
\eqref{eq:endpoint-gap-positivity}.  This proves
\eqref{eq:endpoint-gap-quartic-bound}, whose integral form is
\[
  \int\Gamma_d(W)\ge c_d\|W-r_*\|_4^4.
\]
Thus, $\int\Gamma_d(W)=O_d(h^4)$ implies the desired localization
$\|W-r_*\|_4=O_d(h)$.

To derive this integral bound from the $I_{p_h}$-comparison, we relate
$I_{p_h}$ to the support cost gap $\Gamma_d$.  Repeating the calculation leading to
\eqref{eq:constant-comparison-decomposition}, instead using $W$ and $W_h$, gives
\begin{equation}\label{eq:graphon-cost-decomposition}
  I_{p_h}(W)-I_{p_h}(W_h)
  =\int\Gamma_d(W)-\int\Gamma_d(W_h)
  +\beta_d\left(\int W^d-r_h^d\right)
  +\Lambda_h\{e(W)-e(W_h)\}.
\end{equation}
The proof of the next lemma uses this decomposition to obtain the required
bound on $\int\Gamma_d(W)$.
Briefly, the $I_{p_h}$-comparison makes the right-hand side nonpositive.
Using \eqref{eq:endpoint-gap-quartic-bound}, we absorb the edge-density
correction into half of $\int\Gamma_d(W)$ and an $O_d(h^4)$ remainder.
The remaining quantities $\int\Gamma_d(W)$ and $\int W^d-r_h^d$ are
nonnegative, the latter by feasibility.  Since $\int\Gamma_d(W_h)=O_d(h^4)$,
both must be $O_d(h^4)$, yielding the desired localization.

The preceding discussion gives the first stage of localization, but this
alone does not suffice for the later comparison with $W_h$, since $W$ need
not be rank-one.  The lemma below therefore extracts a
rank-one part $f\otimes f$ and writes the remainder as $E$.  For a symmetric
kernel $K$ and a function $g$, write
\[
  (T_Kg)(x):=\int_0^1K(x,y)g(y)\,dy.
\]
The factor is chosen so that $T_E(f^{d-1})=0$.  This identity eliminates the
terms linear in $E$ from the later moment and homomorphism-density expansions.
The following lemma records both the localization and the rank-one reduction.

\begin{lemma}[Localization and rank-one reduction]
\label{lem:localization-rank-one}
There exist $C_d,C_{d,m}>1$ and $h_0>0$ such that, for every
$0<h<h_0$, every graphon $W$ satisfying $t(H,W)\ge r_h^m$ and
$I_{p_h}(W)\le I_{p_h}(W_h)$ satisfies
\begin{equation}\label{eq:graphon-quartic-localization}
  \|W-r_*\|_4\le C_dh.
\end{equation}
Moreover, $W$ admits a decomposition $W=f\otimes f+E$ for some
$f\in L^d([0,1])$ and $E\in L^2([0,1]^2)$ satisfying
\begin{equation}\label{eq:rank-one-orthogonality}
  T_W(f^{d-1})=q(f)f,
  \qquad
  T_E(f^{d-1})=0,
  \qquad
  q(f)>0,
  \qquad
  0\le f(x) \le 2
  \quad\text{for a.e. }x.
\end{equation}
Also, the function $f$ is close to the constant function $u_*$, the error
$E$ is small, the $d$-th moment $q(f)$ is close to $q_h$, and the
homomorphism-density error is bounded by a cubic term in $\|E\|_2$:
\begin{equation}\label{eq:rank-one-reduction-bounds}
  \begin{gathered}
  \|f-u_*\|_4\le C_dh,
  \qquad
  \|E\|_2 \le C_dh,
  \qquad
  \left|q(f)-q_h\right|\le C_dh^2,
  \qquad
  |t(H,W)-q(f)^v| \le C_{d,m} \|E\|_2^3.
  \end{gathered}
\end{equation}
\end{lemma}

\begin{proof}
We first localize $W$ in $L^4$ around the constant graphon $W\equiv r_*$,
then construct its rank-one part $f\otimes f$ and an error $E$ satisfying
$T_E(f^{d-1})=0$, and finally bound the moment and homomorphism-density
errors.

\medskip
\noindent\emph{Localization of the graphon.}
Combining the $I_{p_h}$-comparison with \eqref{eq:graphon-cost-decomposition} gives
\[
  0\ge I_{p_h}(W)-I_{p_h}(W_h)
  =\int\Gamma_d(W)-\int\Gamma_d(W_h)
  +\beta_d\left(\int W^d-r_h^d\right)
  +\Lambda_h\{e(W)-e(W_h)\}.
\]
Generalized H\"older \eqref{eq:generalized-holder} and feasibility give $\int W^d\ge r_h^d$.
To absorb the edge-density correction, convexity, Taylor's theorem, and
\eqref{eq:endpoint-gap-quartic-bound} give
\[
  0\le z^d-r_*^d-d r_*^{d-1}(z-r_*)
  \le C_d|z-r_*|^2
  \le C_d\Gamma_d(z)^{1/2}
  \qquad(0\le z\le1).
\]
Applying the upper bound to $W$ and the lower bound to $W_h$, then using
$\int W_h^d=r_h^d$, yields
\[
  e(W)-e(W_h)
  \ge\frac{\int W^d-r_h^d}{d r_*^{d-1}}
   -C_d\int\Gamma_d(W)^{1/2}
  \ge-C_d\int\Gamma_d(W)^{1/2}.
\]
By \eqref{eq:rank-one-parameter-expansions} and \eqref{eq:family-gap-expansion},
we have $\Lambda_h=O_d(h^2)$ and $\int\Gamma_d(W_h)=O_d(h^4)$.
Substitution into the cost comparison gives
\begin{align*}
  \beta_d \left(\int W^d-r_h^d\right)+\int\Gamma_d(W)
  &\le C_dh^4+C_dh^2\int\Gamma_d(W)^{1/2}\\
  &\le C_dh^4+C_dh^2\left(\int\Gamma_d(W)\right)^{1/2}\\
  &\le \left(C_d+\frac{C_d^2}{2}\right)h^4+\frac12\int\Gamma_d(W).
\end{align*}
Absorbing the last term and using $\beta_d>0$ give
\begin{equation}\label{eq:moment-and-cost-gap-bounds}
  0\le\int W^d-r_h^d\le C_dh^4,
  \qquad
  0\le\int\Gamma_d(W)\le C_dh^4.
\end{equation}
Thus, as discussed above, using \eqref{eq:endpoint-gap-quartic-bound} and $\int\Gamma_d(W)=O_d(h^4)$
gives the desired localization \eqref{eq:graphon-quartic-localization}.

\medskip
\noindent\emph{Construction of the decomposition.}
We first construct $f$ and $E$ satisfying \eqref{eq:rank-one-orthogonality}.
Put $d':=d/(d-1)$ so that $(1/d)+(1/d')=1$.
To see the integral operator $T_W:L^{d'}\longrightarrow L^d$ is compact,
choose finite-step graphons $S_n$ satisfying
$\|W-S_n\|_d\to0$.  For each $n$, choose a finite partition
$A_1^{(n)},\ldots,A_{N_n}^{(n)}$ of $[0,1]$ on whose rectangles $S_n$ is
constant, and write
\[
  S_n(x,y)=\sum_{i,j=1}^{N_n}c_{ij}^{(n)}
  \1_{A_i^{(n)}}(x)\1_{A_j^{(n)}}(y),
  \qquad
  (T_{S_n}g)(x)=\sum_{i=1}^{N_n} \left(\sum_{j=1}^{N_n}c_{ij}^{(n)}\int_{A_j^{(n)}}g\right) \1_{A_i^{(n)}}(x).
\]
Thus, the range of $T_{S_n}$ is contained in the span of the finitely many
functions $\1_{A_1^{(n)}},\ldots,\1_{A_{N_n}^{(n)}}$, so $T_{S_n}$ is compact.
Moreover, H\"older's inequality gives
\[
  \|(T_W-T_{S_n})g\|_d
  \le\|W-S_n\|_d\|g\|_{d'}.
\]
Thus, $S_n\to W$ ensures that $T_{S_n}g\to T_Wg$ uniformly over
$\|g\|_{d'}\le1$.  In other words, $T_{S_n}\to T_W$ in operator norm.
Since the space of compact
operators is closed in the operator norm, $T_W$ is compact.

Now we consider the quadratic form
\[
  \mathcal Q_W(g):= \langle g,T_Wg\rangle =\iint W(x,y)g(x)g(y)\,dx\,dy
  \qquad
  (g \in L^{d'}([0,1]),\, \|g\|_{d'}\le1).
\]
Since $1<d'<\infty$, the closed unit ball of $L^{d'}$ is weakly compact,
and the compactness of $T_W$ makes $\mathcal Q_W$ weakly continuous on this ball.
Thus $\mathcal Q_W$ attains a maximum there at some $g_0$;
let $\theta:=\mathcal Q_W(g_0)$ be the maximum.
The localization bound on $W$ and H\"older's inequality give $\theta \in [r_*/2,1]$,
after decreasing $h_0$, since
\[
  \frac{r_*}{2} \le r_*-\|W-r_*\|_1 \le \int W=\mathcal Q_W(1)\le\theta
  = \mathcal Q_W(g_0)\le\|W\|_d\|g_0\|_{d'}^2\le 1.
\]
Since $W\ge0$, taking the absolute value cannot decrease $\mathcal Q_W$,
and positivity of $\theta$ and homogeneity force every maximizer to have norm
one.  We may therefore choose a maximizer $g_0\ge0$ with $\|g_0\|_{d'}=1$.

Since \(d'>1\), the \(L^{d'}\)-unit sphere is smooth.  Therefore, applying the
Lagrange-multiplier theorem to the maximization of $\mathcal Q_W$ on this sphere
gives \(T_Wg_0=\mu g_0^{d'-1}\) for some constant \(\mu\). This gives
\[
  \theta=\mathcal Q_W(g_0)=\int g_0\,T_Wg_0
  =\mu\int g_0^{d'}=\mu\|g_0\|_{d'}^{d'}=\mu.
\]
Hence $T_Wg_0=\theta g_0^{1/(d-1)}$.
Define $f:=\theta^{1/2}g_0^{1/(d-1)}$ and, for the remainder of the proof,
abbreviate $q:=q(f)=\int f^d=\theta^{d/2}>0$.  Then
\begin{align*}
  T_W(f^{d-1})
  &=T_W(\theta^{(d-1)/2}g_0)
  =\theta^{(d+1)/2}g_0^{1/(d-1)}
  =\theta^{d/2}f=qf,
  \\
  T_{f\otimes f}(f^{d-1})
  &=\left(\int f(y)^d\,dy\right)f
  =qf.
\end{align*}
Let $E:=W-f\otimes f$.  Subtracting these identities gives $T_E(f^{d-1})=0$.
Also, $\theta \in [r_*/2,1]$ implies $q = \theta^{d/2} \in [(r_*/2)^{d/2},1]$.
Consequently, since $f\ge0$,
\[
  qf(x)=\int W(x,y)f(y)^{d-1}\,dy
  \le\int f^{d-1}\le q^{(d-1)/d},
  \qquad
  0\le f(x)\le q^{-1/d}\le\sqrt{\frac{2}{r_*}} \le 2
  \quad\text{for a.e. }x.
\]

\medskip
\noindent\emph{Factor and error bounds.}
Substituting $W=r_*+(W-r_*)$ into the first identity of
\eqref{eq:rank-one-orthogonality} gives
\[
  qf=T_W(f^{d-1})
  =r_*\int f^{d-1}+T_{W-r_*}(f^{d-1}).
\]
The second term on the right-hand side is small, since
$\|T_{W-r_*}(f^{d-1})\|_4 \le\|W-r_*\|_4\|f^{d-1}\|_{4/3} \le C_d h$.
Thus, to show that $f$ is close to $u_*$, it remains to compare the
first term with $qu_*$.  Define $b$ by
\[
  b:=\frac{r_*\int f^{d-1}}q,
  \qquad
  f-b=q^{-1}T_{W-r_*}(f^{d-1}),
  \qquad
  \|f-b\|_4\le C_dh.
\]
The last inequality follows from the bound on the second term above and the
uniform lower bound on $q$.  The uniform bounds on $f$ and $q$ also give a
two-sided uniform bound on $b$:
\[
  \frac{q}{C_d}\le\int f^{d-1}\le q^{(d-1)/d},
  \qquad
  C_d^{-1}\le b\le C_d.
\]
Lipschitz continuity of $x\mapsto x^k$ for $k\in\{d-1,d\}$
on a bounded interval and H\"older's inequality give
\[
  \left|\int f^k-b^k\right|
  \le \|f^k-b^k\|_4
  \le C_d\|f-b\|_4
  \le C_dh
  \qquad(k=d-1,d).
\]
Taking $k=d$ and $k=d-1$, respectively, and using
$q=\int f^d$ and $qb=r_*\int f^{d-1}$ gives
\[
  |q-b^d|\le C_dh,
  \qquad
  |qb-r_*b^{d-1}|\le C_dh.
\]
Since $b$ is uniformly bounded away from zero and $r_*=u_*^2$, these bounds imply
\[
  |b^2-r_*| = \frac{\left| (qb-r_*b^{d-1})-b(q-b^d) \right|}{b^{d-1}} \le C_dh,
  \qquad
  |b-u_*| \le C_dh.
\]
Combining this with the bound on $f-b$ gives
\[
  \|f-u_*\|_4\le \|f-b\|_4+\|b-u_*\|_4\le C_dh.
\]
Finally, $E=W-(f\otimes f)$ is small since $W$ and $f \otimes f$ are both close to $r_*$ in $L^4$:
\begin{equation*}
  \|E\|_2\le\|E\|_4
  \le \|W-r_*\|_4+\|(f\otimes f)-r_*\|_4
  \le \|W-r_*\|_4+(\|f\|_4+u_*)\|f-u_*\|_4
  \le C_dh.
\end{equation*}

\medskip
\noindent\emph{Moment and homomorphism-density errors.}
The binomial expansion of $W^d=(f\otimes f+E)^d$ gives
\begin{equation*}
  \left| \int W^d-q^2 \right|
  = \left| \sum_{k=2}^d\binom dk\int(f\otimes f)^{d-k}E^k \right|
  \le C_d\sum_{k=2}^d\int|E|^k \le C_d\|E\|_2^2.
\end{equation*}
The $k=0$ term cancels with $q^2$, and the $k=1$ term vanishes by orthogonality $T_E(f^{d-1})=0$.
Together with the first bound in \eqref{eq:moment-and-cost-gap-bounds} and $q_h^2=r_h^d$,
\[
  |q-q_h| \le C_d\left|q^2-q_h^2\right|
  \le C_d \left( \left| \int W^d-q^2 \right| + \left| \int W^d - r_h^d \right| \right)
  \le C_dh^2.
\]
Also, expand $t(H,f\otimes f+E)$ according to the edges on which $E$ and $f\otimes f$ are used:
\[
  t(H,W)=\sum_{F\subseteq E(H)}\int
  \prod_{\{i,j\}\in F}E(x_i,x_j)
  \prod_{\{i,j\}\in E(H)\setminus F}f(x_i)f(x_j)\,d\mathbf x.
\]
The $F=\varnothing$ term corresponds to $t(H,f\otimes f)=q^v$, so it remains to bound
the terms with $F\ne\varnothing$.
If a vertex $i$ is incident to exactly one $E$-edge $\{i,j\}$, then integrating
first in $x_i$ gives
\[
  \int E(x_i,x_j)f(x_i)^{d-1}\,dx_i
  =T_E(f^{d-1})(x_j)=0,
\]
so that term vanishes.  Thus, every remaining nonzero term contains a cycle of
$E$-edges, of length $\ell\ge3$ because $H$ is simple.  Since the other
factors are uniformly bounded, we may bound them by a constant and integrate
out all variables outside a chosen cycle.  Thus, the absolute value of the term
is at most
\[
  C_{d,m}\int_{[0,1]^\ell}\prod_{k=1}^\ell|E(x_k,x_{k+1})|\,dx_1\cdots dx_\ell
  \le C_{d,m}\|E\|_2^\ell\le C_{d,m}\|E\|_2^3,
  \qquad x_{\ell+1}:=x_1,
\]
where repeated Cauchy--Schwarz gives the first norm bound, and the last one
uses $\ell\ge3$ and the uniform boundedness of $\|E\|_2$.  Summing the
absolute values of the finitely many terms gives
\begin{equation*}
  |t(H,W)-q^v|\le C_{d,m}\|E\|_2^3.
\end{equation*}
\end{proof}

%% file: sections/singular_auxiliary_lagrangian.tex
\subsection{A lower bound for the auxiliary Lagrangian difference}
\label{sec:auxiliary-lagrangian}

The aim of this subsection is to bound from below the difference between the
value of an auxiliary Lagrangian associated with the factor $f$ and its value
at the reference factor $f_h$.  Here $f$ comes from the
decomposition $W=f\otimes f+E$, while $f_h$ is the factor defining
$W_h=f_h\otimes f_h$.  By \Cref{lem:localization-rank-one}, every
feasible graphon satisfying $I_{p_h}(W)\le I_{p_h}(W_h)$ admits such a
decomposition, with $f$ close to $u_*$ and $E$ small.  Since $0\le f\le2$,
the kernel $f\otimes f$ may take values larger than $1$ and need not be a
graphon.  We therefore introduce an auxiliary Lagrangian that can be evaluated
at both $f$ and $f_h$, temporarily leaving $E$ aside.  In
\Cref{sec:graphon-comparison}, the resulting lower bound will be
combined with estimates for $E$ to compare the actual graphons $W$ and $W_h$.

The quantities in the auxiliary Lagrangian depend only on how frequently $f$
takes each of its possible values.  Let $\lambda$ denote the Lebesgue
probability measure on $[0,1]$, and define the distributions of the values of
$f$ and $f_h$ by pushing $\lambda$ forward under these two functions:
\begin{equation*}
  \nu:=f_\#\lambda,
  \qquad
  \nu_h:=(f_h)_\#\lambda.
\end{equation*}
Thus, for every Borel set $B\subseteq[0,2]$,
$\nu(B)=\lambda\{x:f(x)\in B\}$, and the analogous identity with $f_h$ holds
for $\nu_h$.  Because $f(x)f(y)$ may exceed $1$, we use a particular auxiliary
continuation $\widetilde J_{p_h}$ of $J_{p_h}$ from $[0,1]$ to $[0,4]$, defined
below.  The first integral term in the auxiliary Lagrangian is then
\[
 \iint \widetilde J_{p_h}(f(x)f(y))\,dx\,dy.
\]
By the definition of $\nu$, the integral and the moment
$q(f)$ can be written as
\begin{equation*}
  \iint \widetilde J_{p_h}(f(x)f(y))\,dx\,dy
  =\iint \widetilde J_{p_h}(st)\,d\nu(s)d\nu(t),
  \qquad
  q(f)=\int s^d\,d\nu(s),
\end{equation*}
and the same identities hold with $f_h$ and $\nu_h$, with
$q(f_h)=q_h$.  Although $f\otimes f$ need not be a graphon, the integral
defining $t(H,\,\cdot\,)$ is still well defined for this bounded kernel.  Since
$H$ is $d$-regular, the rank-one identity gives
\begin{equation*}
  t(H,f\otimes f)=q(f)^v,
  \qquad
  t(H,W_h)=q_h^v=r_h^m.
\end{equation*}
The auxiliary Lagrangian difference studied in this subsection is
\begin{align*}
  &\left[
    \iint \widetilde J_{p_h}(f(x)f(y))\,dx\,dy
    -\mu_h\{t(H,f\otimes f)-r_h^m\}
  \right]
  \\
  &\qquad-
  \left[
    \iint \widetilde J_{p_h}(f_h(x)f_h(y))\,dx\,dy
    -\mu_h\{t(H,f_h\otimes f_h)-r_h^m\}
  \right]
  \\
  & {} = \left[
    \iint \widetilde J_{p_h}(st)\,d\nu(s)d\nu(t)
    -\mu_h\{q(f)^v-r_h^m\}
  \right]
  \\
  &\qquad\quad-
  \left[
    \iint \widetilde J_{p_h}(st)\,d\nu_h(s)d\nu_h(t)
    -\mu_h\{q_h^v-r_h^m\}
  \right].
\end{align*}
The main conclusion, \Cref{lem:auxiliary-lagrangian-bound}, provides
constants $a>0$ and $C_{d,m}\ge1$, independent of $\rho$ and $h$, with the
following property.  For every sufficiently small $\rho>0$, one can choose
$b_\rho>0$ so that, for all sufficiently small $h>0$, the above difference is
bounded below by
\begin{equation*}
  \frac{a}{2}
  \int_{\{|s-u_*|<\rho\}}(s-s_h)^2(s-t_h)^2\,d\nu(s)
  +\frac{b_\rho}{2}\nu\{|s-u_*|\ge\rho\}
  -C_{d,m}\{q(f)-q_h\}^2.
\end{equation*}
The integral in this lower bound measures the deviation, within a
neighborhood of $u_*$, from the two values $s_h,t_h$ taken by $f_h$.  The
next term measures the mass outside that neighborhood, and the final,
possibly negative, term depends only on the difference between the $d$-th
moments of $f$ and $f_h$.  This lower bound for the auxiliary Lagrangian
difference is the conclusion sought in this subsection.  It does not yet
compare the graphon costs of $W$ and $W_h$ or identify $f$ with $f_h$.  In
\Cref{sec:graphon-comparison}, we restore the remainder $E$ and
convert this bound on the auxiliary Lagrangian difference into a lower bound
involving the actual cost and constraint value of $W$.  The resulting lower
bound for the full-graphon Lagrangian difference is then used in
\Cref{sec:singular-proof} to establish optimality and uniqueness.

We now define $\widetilde J_{p_h}$.  Recall that
$\Lambda_h=\ell(p_h)-\ell_*$.  First set
\[
  \widetilde\Gamma_d(z)=
  \begin{cases}
    \Gamma_d(z),&0\le z\le1,\\
    \Gamma_d(1)+(z-1)^2,&1<z\le4,
  \end{cases}
\]
and for $0 \le z \le 4$, define
\begin{equation}\label{eq:entropy-continuation}
  \widetilde J_{p_*}(z)
  :=J_{p_*}(r_*)+\beta_d(z^d-r_*^d)
  +\widetilde\Gamma_d(z),
  \qquad
  \widetilde J_{p_h}(z)
  :=\widetilde J_{p_*}(z)+\Lambda_h z+J_{p_h}(0)-J_{p_*}(0).
\end{equation}
On $[0,1]$, this definition agrees with $J_{p_h}$ by
\eqref{eq:entropy-parameter-shift} and \eqref{eq:endpoint-supporting-gap}.  For
$1<z\le4$, however, we do not use the relative-entropy formula; instead,
\eqref{eq:entropy-continuation} means explicitly that
\[
  \widetilde J_{p_h}(z)
  =J_{p_*}(r_*)+\beta_d(z^d-r_*^d)+\Gamma_d(1)+(z-1)^2
   +\Lambda_h z+J_{p_h}(0)-J_{p_*}(0).
\]
Thus $\widetilde J_{p_h}$ is a real-valued continuous function on $[0,4]$,
but its values above $1$ are only auxiliary and are not relative entropies.
Write
$K_h(x,y):=\widetilde J_{p_h}(xy)$ for the associated kernel on $[0,2]^2$.
The following lemma gives H\"older continuity and uniform bounds for
$\widetilde J_{p_h}$ and $K_h$.
\begin{lemma}[Uniform control of auxiliary continuation and kernel]
\label{lem:continuation-kernel-bounds}
There are $C_d\ge1$ and $h_0>0$ such that, whenever $0\le h<h_0$,
$z,z'\in[0,4]$, and $x,x',y\in[0,2]$,
\begin{equation}\label{eq:continuation-kernel-bounds}
  \begin{aligned}
    |\widetilde J_{p_h}(z)-\widetilde J_{p_h}(z')|&\le C_d|z-z'|^{1/2},
    & \|\widetilde J_{p_h}-\widetilde J_{p_*}\|_{L^\infty([0,4])}&\le C_dh^2,
    & \|\widetilde J_{p_h}\|_{L^\infty([0,4])}&\le C_d,
    \\
    |K_h(x,y)-K_h(x',y)| &\le C_d|x-x'|^{1/2},
    & \|K_h-K_0\|_{L^\infty([0,2]^2)}&\le C_dh^2,
    & \|K_h\|_{L^\infty([0,2]^2)}&\le C_d.
  \end{aligned}
\end{equation}
\end{lemma}

\begin{proof}
On $[0,1]$, the only non-Lipschitz terms in
$\widetilde J_{p_*}=J_{p_*}$ are $z\log z$ and
$(1-z)\log(1-z)$.  If $0\le z'\le z\le1$,
\[
  |z\log z-z'\log z'|
  \le\int_{z'}^z|1+\log t|\,dt
  \le C_d\int_{z'}^z t^{-1/2}\,dt
  \le C_d|z-z'|^{1/2},
\]
with the same bound for $(1-z)\log(1-z)$.
On $[1,4]$, $\widetilde J_{p_*}$ is a polynomial and hence Lipschitz.
If $z\le1\le z'$, 
\[
  |\widetilde J_{p_*}(z)-\widetilde J_{p_*}(z')|
  \le |\widetilde J_{p_*}(z)-\widetilde J_{p_*}(1)|
    +|\widetilde J_{p_*}(1)-\widetilde J_{p_*}(z')|
  \le C_d\{|1-z|^{1/2}+|z'-1|^{1/2}\}
  \le C_d|z-z'|^{1/2},
\]
by continuity at $1$.
Thus $\widetilde J_{p_*}$ is $1/2$-H\"older continuous on $[0,4]$.
The definition \eqref{eq:entropy-continuation} gives
\[
  \widetilde J_{p_h}(z)-\widetilde J_{p_h}(z')
  =\widetilde J_{p_*}(z)-\widetilde J_{p_*}(z')
  +\Lambda_h(z-z'),
  \qquad
  \widetilde J_{p_h}(z)-\widetilde J_{p_*}(z)
  =\Lambda_h z+J_{p_h}(0)-J_{p_*}(0).
\]
By \eqref{eq:rank-one-parameter-expansions}, $\Lambda_h=O_d(h^2)$.
Since $p_h=(1+e^{\ell_*+\Lambda_h})^{-1}$ and
$J_p(0)=-\log(1-p)$, we also have
$J_{p_h}(0)-J_{p_*}(0)=O_d(h^2)$.
Together with $|z-z'|\le2|z-z'|^{1/2}$ on $[0,4]$, these estimates give the
desired bounds: the first identity proves the $1/2$-H\"older continuity of
$\widetilde J_{p_h}$, while the second proves the uniform bound on
$\widetilde J_{p_h}-\widetilde J_{p_*}$.
The uniform bound on $\widetilde J_{p_h}$ follows from the uniform bounds on
$\widetilde J_{p_*}$ and $\widetilde J_{p_h}-\widetilde J_{p_*}$.
The three bounds for $K_h$ follow from the corresponding bounds for
$\widetilde J_{p_h}$ by taking $z=xy$ and, for the H\"older estimate,
$z'=x'y$.
\end{proof}

The auxiliary Lagrangian comparison uses the following moment and deviation
quantities.  Since $f_h=f_{\alpha_h,s_h,t_h}$, the distribution of
its values has the explicit form
\begin{equation*}
  \nu_h=\alpha_h\delta_{s_h}+(1-\alpha_h)\delta_{t_h},
\end{equation*}
where $\delta_x$ denotes the Dirac mass at $x$.
For any probability measure $\xi$ on $[0,2]$, write $m_d(\xi)$ for its
$d$-th moment and $\Delta_h(\xi)$ for the defect of this moment relative to
$q_h$:
\begin{equation*}
  m_d(\xi):=\int x^d\,d\xi(x),
  \qquad
  \Delta_h(\xi):=m_d(\xi)-q_h.
\end{equation*}
For the two measures fixed above, $m_d(\nu)=q(f)$ and
$m_d(\nu_h)=q_h$.
In this notation, the rank-one identities above become
\begin{equation*}
  t(H,f\otimes f)=m_d(\nu)^v,
  \qquad
  t(H,W_h)=m_d(\nu_h)^v=q_h^v.
\end{equation*}
In addition to the moment defect, the lower bound controls two further ways in
which $\nu$ may differ from $\nu_h$.  Within a neighborhood of $u_*$, $\nu$
may assign mass to values other than $s_h$ and $t_h$; outside that
neighborhood, it may have nonzero mass.  To record these two deviations, for
$\rho>0$, define the central neighborhood and the tail region in $[0,2]$ by
\[
  \mathcal N_\rho:=(u_*-\rho,u_*+\rho),
  \qquad
  \mathcal T_\rho:=[0,2]\setminus\mathcal N_\rho.
\]
We always choose $\rho_0>0$ below small enough that
$\overline{\mathcal N_{\rho_0}}\subset(0,1)$.
For $h>0$, set $Q_h(x):=(x-s_h)(x-t_h)$, and for any probability measure
$\xi$ on $[0,2]$ define
\[
  A_{h,\rho}(\xi)
  :=\int_{\mathcal N_\rho}Q_h(x)^2\,d\xi(x).
\]
Here $A_{h,\rho}(\xi)$ measures deviation from the two factor values $s_h$
and $t_h$ within the central neighborhood, whereas $\xi(\mathcal T_\rho)$ is
the mass outside that neighborhood.  For the
decomposition furnished by \Cref{lem:localization-rank-one},
Chebyshev's inequality gives, for every fixed $\rho>0$, a constant
$C_{d,\rho}\ge1$ such that
\begin{equation}\label{eq:factor-tail-mass}
  \nu(\mathcal T_\rho)
  \le \rho^{-4}\|f-u_*\|_4^4
  \le C_{d,\rho}h^4.
\end{equation}

We now turn to the proof of the lower bound for the auxiliary Lagrangian
difference.  For any probability measure $\xi$ on $[0,2]$, define the
auxiliary factor functional
$\mathcal J_h$ and, using $q_h^v=r_h^m$, the auxiliary Lagrangian
$\mathcal L_h$ by
\[
  \mathcal J_h(\xi)
  :=\iint K_h(x,y)\,d\xi(x)d\xi(y),
  \qquad
  \mathcal L_h(\xi)
  :=\mathcal J_h(\xi)-\mu_h\{m_d(\xi)^v-q_h^v\}.
\]
Put $\sigma:=\xi-\nu_h$.  Expanding $\mathcal L_h(\xi)$ around $\nu_h$
produces a term linear in $\sigma$, a term quadratic in $\sigma$, and a Taylor
remainder depending on $\Delta_h(\xi)$.  The definition of $\mu_h$ and the
identity $m=dv/2$ give
$\mu_hvq_h^{v-1}=2\gamma_hq_h/d$.  To express the linear term, define the
first-variation potential $\Psi_h$ by
\begin{equation}\label{eq:first-variation}
  \Psi_h(x):=2\int K_h(x,y)\,d\nu_h(y)
  -\frac{2\gamma_hq_h}{d}x^d-\kappa_h,
\end{equation}
where $\kappa_h$ is chosen so that $\Psi_h(s_h)=0$.
Since $\int d\sigma=0$ and $\int x^d\,d\sigma=\Delta_h(\xi)$, we have
\begin{equation*}
  \mathcal J_h(\xi)-\mathcal J_h(\nu_h)
    -\mu_h\{m_d(\xi)^v-q_h^v\}
  =\int\Psi_h\,d\sigma+\iint K_h\,d\sigma d\sigma
    -\Bigl[\mu_h\{(q_h+\Delta_h(\xi))^v-q_h^v\}
      -\frac{2\gamma_hq_h}{d}\Delta_h(\xi)\Bigr].
\end{equation*}
Since $m_d(\nu_h)=q_h$, the left-hand side is precisely
$\mathcal L_h(\xi)-\mathcal L_h(\nu_h)$.  Using
$\mu_hvq_h^{v-1}=2\gamma_hq_h/d$, the bracket can be rewritten as
\[
  \mu_h\Bigl\{(q_h+\Delta_h(\xi))^v-q_h^v
  -vq_h^{v-1}\Delta_h(\xi)\Bigr\}.
\]
This is the Taylor remainder after the constant and linear terms have been
removed, so its absolute value is at most
$C_{d,m}\Delta_h(\xi)^2$.  The other two terms in the expansion are
$\int\Psi_h\,d\sigma$, which is linear in $\sigma$ and is therefore called the
first-variation term, and $\iint K_h\,d\sigma d\sigma$, which is quadratic in
$\sigma$.  We next obtain a positive lower bound for the first-variation term
and control how negative the quadratic kernel term can be.

The reference measure $\nu_h$ is supported on $\{s_h,t_h\}$.  Our choice of
$\kappa_h$ gives $\Psi_h(s_h)=0$, while the block-proportion condition
\eqref{eq:block-stationarity}, satisfied by $W_h$ by
\Cref{lem:rank-one-kkt-family}, gives $\Psi_h(t_h)=0$.  For $x$ in a
neighborhood of $u_*$, all products in the following formula lie in $(0,1)$,
so differentiating \eqref{eq:first-variation} yields
\[
  \Psi_h'(x)
  =2\left[
    \alpha_hs_hJ_{p_h}'(xs_h)
    +(1-\alpha_h)t_hJ_{p_h}'(xt_h)
  \right]-2\gamma_hq_hx^{d-1}.
\]
At $x=s_h$ and $x=t_h$, the rank-one KKT equation and
$m_d(\nu_h)=q_h$ give
\[
  \Psi_h'(x)
  =2\gamma_hq_hx^{d-1}-2\gamma_hq_hx^{d-1}
  =0 \qquad (x\in\{s_h,t_h\}).
\]
Hence $\Psi_h$ has double zeros at $s_h$ and $t_h$.
The following lemma turns these zeros into the two bounds needed below: near
$u_*$, the first variation controls the distance from $s_h$ and $t_h$, while
outside that neighborhood it imposes a fixed positive penalty.

\begin{lemma}[First variation lower bound]
\label{lem:first-variation-bound}
There exist $\rho_0>0$ and $a>0$ with the following property.  For every
$0<\rho<\rho_0$, one can choose $b_\rho,h_\rho>0$ so that, for all
$0<h<h_\rho$,
\begin{equation*}
  \Psi_h(x)\ge a Q_h(x)^2
  \quad(x\in\mathcal N_\rho),
  \qquad
  \Psi_h(x)\ge b_\rho
  \quad(x\in\mathcal T_\rho).
\end{equation*}
\end{lemma}

\begin{proof}
The central bound comes from the four equalities
$\Psi_h(s_h)=\Psi_h'(s_h)=\Psi_h(t_h)=\Psi_h'(t_h)=0$: they force
$\Psi_h$ to contain the factor $Q_h^2$, with a positive quotient near $u_*$.
The tail bound then follows from the strict positivity of the limiting
potential away from $u_*$ and uniform convergence as $h\to0$.

At $h=0$ one has $\nu_0=\delta_{u_*}$, $\gamma_0=d\beta_d$, and
$q_0=u_*^d$.  Since $\kappa_0$ is chosen so that
$\Psi_0(u_*)=0$, substituting these values into
\eqref{eq:first-variation} and using
\eqref{eq:entropy-continuation} gives
\begin{equation}\label{eq:endpoint-first-variation}
  \Psi_0(x)=2\widetilde\Gamma_d(u_*x).
\end{equation}
Near $x=u_*$, $\widetilde\Gamma_d=\Gamma_d$, and using
\eqref{eq:endpoint-gap-expansion} and
$u_*=((d-1)/d)^{1/2}$ gives
\begin{equation*}
  \Psi_0(x)
  =2\Gamma_d(r_*+u_*(x-u_*))
  =\frac{d^3}{12}(x-u_*)^4+O_d((x-u_*)^5).
\end{equation*}
To obtain a lower bound that remains uniform as the two zeros approach one
another, apply analytic division, on a neighborhood of $u_*$, to $\Psi_h(x)$
by the monic polynomial $Q_h(x)^2$.  This gives
\[
  \Psi_h(x)=Q_h(x)^2D_h(x)+S_h(x),
  \qquad \deg_xS_h\le3,
\]
where $D_h(x)$ is jointly analytic in $(h,x)$ and the coefficients of
$S_h(x)$ are analytic in $h$.  For every $h\ne0$, the four equalities above
make this remainder vanish at $s_h,t_h$, with multiplicity two;
hence $S_h(x)$ is identically zero.  By analyticity, it also vanishes at
$h=0$.  Thus
\[
  \Psi_h(x)=Q_h(x)^2D_h(x).
\]
At $h=0$ and $x=u_*$, the Taylor expansion of $\Psi_0$ gives
\[
  D_0(u_*)=\frac{d^3}{12}.
\]
By continuity, after decreasing $\rho_0$ and $h_\rho$, we therefore have
$D_h(x)\ge d^3/24$ for $0<h<h_\rho$ and $x\in\mathcal N_\rho$.  Thus the
lower bound holds with $a:=d^3/24$.

It remains to bound the tail.
By \eqref{eq:endpoint-first-variation}, $\widetilde\Gamma_d\ge0$, with equality
only at $r_*$.  Therefore, $\Psi_0\ge0$, with equality only at $u_*$.
Since $\mathcal T_\rho$ is compact and $u_* \notin \mathcal T_\rho$,
the minimum of $\Psi_0$ on $\mathcal T_\rho$ is positive.  Let $b_\rho$ be
half of this minimum:
\[
  b_\rho:=\frac12\min_{x\in\mathcal T_\rho}\Psi_0(x)>0.
\]
Writing out $\nu_h=\alpha_h\delta_{s_h}+(1-\alpha_h)\delta_{t_h}$ in
\eqref{eq:first-variation} gives
\[
  \Psi_h(x)-\Psi_0(x)
  =2\alpha_h\{K_h(x,s_h)-K_0(x,u_*)\}
  +2(1-\alpha_h)\{K_h(x,t_h)-K_0(x,u_*)\}
  -\frac{2}{d}(\gamma_hq_h-\gamma_0q_0)x^d-(\kappa_h-\kappa_0).
\]
By \Cref{lem:continuation-kernel-bounds} and the convergence
$s_h,t_h\to u_*$, we have
\[
  \|K_h(\,\cdot\,,s_h)-K_0(\,\cdot\,,u_*)\|_\infty
  +\|K_h(\,\cdot\,,t_h)-K_0(\,\cdot\,,u_*)\|_\infty
  \longrightarrow0.
\]
Also, $\alpha_h\to1/2$, $\gamma_hq_h\to\gamma_0q_0$, and
$\kappa_h\to\kappa_0$ imply $\|\Psi_h-\Psi_0\|_\infty\to0$.
Decreasing $h_\rho$ so that $\|\Psi_h-\Psi_0\|_\infty\le b_\rho$,
we get $\Psi_h\ge b_\rho$ on $\mathcal T_\rho$.
\end{proof}

Because $\Psi_h$ vanishes at both points in the support of $\nu_h$, we have
$\int\Psi_h\,d\sigma=\int\Psi_h\,d\xi$.  The preceding lemma therefore makes
the first-variation term nonnegative and quantifies its size.  The quadratic
term $\iint K_h\,d\sigma d\sigma$, however, need not be nonnegative.  The
following lemma controls its possible negative contribution when both
variables lie in $\mathcal N_\rho$.

\begin{lemma}[Central kernel bound]
\label{lem:central-kernel-bound}
Let $a>0$ be the constant from
\Cref{lem:first-variation-bound}.
There exist $\rho_0>0$ and $C_d\ge1$ with the following property.  For every
$0<\rho<\rho_0$, one can choose $h_\rho>0$ so that the following holds.  Let
$0<h<h_\rho$, let $\xi$ be any probability measure on $[0,2]$, and put
$\sigma:=\xi-\nu_h$.  Then
\begin{equation}\label{eq:central-kernel-bound}
  \iint_{\mathcal N_\rho^2}K_h\,d\sigma d\sigma
  \ge{}-\frac a2A_{h,\rho}(\xi)
  -C_d\left\{\Delta_h(\xi)^2+\xi(\mathcal T_\rho)^2\right\}.
\end{equation}
\end{lemma}

\begin{proof}
We interpolate $K_h$ in the variables $x^d$ and $y^d$ at the two reference
values $s_h^d,t_h^d$.  This separates the kernel into a part controlled by
the total mass and the $d$-th moment of $\sigma$, two parts containing one
factor of $Q_h$, and a remainder containing $Q_h(x)Q_h(y)$.  The first part
is controlled by $\Delta_h(\xi)$ and $\xi(\mathcal T_\rho)$.  The factors of
$Q_h$ allow the other three parts to be controlled by $A_{h,\rho}(\xi)$
together with these same two quantities.

To carry out this decomposition, for a function $F$ of $(x,y)$ define the
interpolation and remainder operators in the first variable by
\begin{equation}\label{eq:moment-interpolation}
  (I_h^xF)(x,y)
  :=\frac{t_h^d-x^d}{t_h^d-s_h^d}F(s_h,y)
  +\frac{x^d-s_h^d}{t_h^d-s_h^d}F(t_h,y),
  \quad
  (R_h^xF)(x,y):=F(x,y)-(I_h^xF)(x,y),
\end{equation}
and define $I_h^y,R_h^y$ analogously.  The two interpolation operators
commute.  We handle their limiting behavior at $h=0$ below using divided
differences.
Apply the exact identity
\begin{equation*}
  K_h=I_h^xI_h^yK_h+R_h^xI_h^yK_h
   +I_h^xR_h^yK_h+R_h^xR_h^yK_h.
\end{equation*}
Denote the first term by $P_h(x,y)$.  Since it is affine in each of $x^d$
and $y^d$, it can be written as
\[
  P_h(x,y)=p_{00,h}+p_{d0,h}x^d+p_{0d,h}y^d+p_{dd,h}x^dy^d.
\]
The second term vanishes at $x=s_h,t_h$ and, after
division by $Q_h(x)$, is affine in $y^d$.
The third term is the symmetric counterpart.
The last term vanishes at both $x=s_h,t_h$ and $y=s_h,t_h$.
Thus, for suitable functions $u_{0,h},u_{d,h},v_{0,h},v_{d,h}$ and
$\Theta_h$, the identity becomes
\begin{equation*}
  K_h(x,y)
  =P_h(x,y)+Q_h(x)\{u_{0,h}(x)+u_{d,h}(x)y^d\}
  +Q_h(y)\{v_{0,h}(y)+v_{d,h}(y)x^d\}
  +Q_h(x)Q_h(y)\Theta_h(x,y).
\end{equation*}

We next show that all the terms in this decomposition remain uniformly
controlled as $s_h$ and $t_h$ approach one another.  Pass to the coordinates
$w=x^d$ and $z=y^d$.
Choose $\rho_0$ and $h_\rho$ so small that
$\mathcal N_{\rho_0}$ lies
in a fixed compact interval $\mathcal U\subset(0,1)$
and $\{s_h,t_h\} \subset \mathcal N_{\rho}$.
For a function $F$ analytic on a neighborhood of $\mathcal U^2$, define
\[
  \widehat F(w,z):=F(w^{1/d},z^{1/d}),
  \qquad
  \widehat{\mathcal U}:=\{w:w^{1/d}\in\mathcal U\}.
\]
Then, $\widehat F$ is also analytic on a neighborhood of
$\widehat{\mathcal U}^2$.  In particular, its derivatives are uniformly
bounded on $\widehat{\mathcal U}^2$.
In these coordinates, $I_h^x$ and $I_h^y$ are ordinary linear interpolants
at $s_h^d,t_h^d$.  At $h=0$, these interpolation points both equal $u_*^d$, and
the interpolants converge to the corresponding tangent lines.  Thus,
although the individual coefficients in \eqref{eq:moment-interpolation}
contain $(t_h^d-s_h^d)^{-1}$, the combined interpolants and remainders have
convergent limits.  The divided-difference formulas below make this precise.

For fixed $y\in\mathcal U$, formula \eqref{eq:moment-interpolation} gives
\begin{equation*}
  \frac{R_h^xF(x,y)}{Q_h(x)}
  =\frac{x^d-s_h^d}{x-s_h}\frac{x^d-t_h^d}{x-t_h}
    \widehat F[s_h^d,x^d,t_h^d;y^d],
\end{equation*}
where the bracket denotes the second divided difference in the first
variable, with the second variable fixed at $y^d$.
The first two factors are polynomial sums and extend continuously to $x=s_h$
and $x=t_h$; for example,
\[
  \frac{x^d-s_h^d}{x-s_h}
  =\sum_{j=0}^{d-1}x^{d-1-j}s_h^j
  \le d.
\]
The divided difference has the integral representation and uniform bound:
\[
\begin{aligned}
  \widehat F[s_h^d,x^d,t_h^d;y^d]
  &=\int_0^1\int_0^{1-\theta}
    \partial_w^2\widehat F
    \!\left((1-\theta-\tau)s_h^d+\theta x^d+\tau t_h^d,y^d\right)\,d\tau\,d\theta,\\
  \left|\widehat F[s_h^d,x^d,t_h^d;y^d]\right|
  &\le \frac12\sup_{(w,z)\in\widehat{\mathcal U}^2}
    |\partial_w^2\widehat F(w,z)|.
\end{aligned}
\]
Consequently, the quotient $R_h^xF/Q_h$ extends continuously whenever any of
$\{s_h,x,t_h\}$ coincide.

We now apply the formula in both variables, taking $F=K_h$.  On
$\mathcal U^2$, all products $xy$ lie in $(0,1)$, so $K_h(x,y)=J_{p_h}(xy)$
and $K_h$ depends analytically on $(h,x,y)$.  Using brackets for the
successive second divided differences in the two variables, the last term in
the decomposition above is given exactly by
\[
  \Theta_h(x,y)
  =\prod_{c\in\{s_h,t_h\}}
    \frac{x^d-c^d}{x-c}\frac{y^d-c^d}{y-c}
    \widehat K_h[s_h^d,x^d,t_h^d;\,s_h^d,y^d,t_h^d].
\]
The successive integral representations show that this expression extends
jointly continuously when interpolation points coincide, including at
$h=0$.  The analogous lower-order divided-difference formulas and the uniform
derivative bounds on $\widehat K_h$ give a single $C_d\ge1$ such that
\[
  \max\bigl\{|p_{ij,h}|,|u_{i,h}(x)|,|v_{i,h}(y)|\bigr\}\le C_d
  \qquad  (i,j\in\{0,d\},\,0<h<h_0,\,x,y\in\mathcal N_{\rho_0}).
\]
The bounds for the one-variable functions remain valid when $x$ or $y$ equals
$s_h$ or $t_h$, by continuous extension.

The sign of the last term in the kernel decomposition is controlled by the
limiting value of $\Theta_h$.  Evaluating the jointly continuous extension at
$h=0$ and $x=y=u_*$ gives
\[
  \Theta_0(u_*,u_*)
  =\frac{(d u_*^{d-1})^4}{4}
    \left.\partial_w^2\partial_z^2 \widehat K_0(w,z)\right|_{(w,z)=(u_*^d,u_*^d)}.
\]
By \eqref{eq:entropy-continuation},
\[
  \widehat K_0(w,z)
  =\widetilde J_{p_*}((wz)^{1/d})
  =J_{p_*}(r_*)+\beta_d(wz-r_*^d)
    +\widetilde\Gamma_d\bigl((wz)^{1/d}\bigr).
\]
Since $\widetilde\Gamma_d=\Gamma_d$ near $r_*$ and
$\Gamma_d^{(j)}(r_*)=0$ for $1\le j\le3$, the chain rule leaves only the
$\Gamma_d^{(4)}(r_*)$ term, giving
\[
  \Theta_0(u_*,u_*)
  =\frac{(d u_*^{d-1})^4}{4} \left(\frac{u_*^{2-d}}{d}\right)^4 \Gamma_d^{(4)}(r_*)
  =\frac{d^3}{4}.
\]
By the joint continuity of $\Theta_h(x,y)$ at $(h,x,y)=(0,u_*,u_*)$,
after decreasing $\rho_0$ and $h_\rho$, we also have
\[
  \sup_{0<h<h_\rho}
  \left\|\Theta_h-\frac{d^3}{4}\right\|_{L^\infty(\mathcal N_\rho^2)}
  \le\frac a4.
\]

We now integrate the kernel decomposition against $\sigma\otimes\sigma$.
Since the support of $\nu_h$ lies in $\mathcal N_\rho$, we have
$\sigma(\mathcal N_\rho)=-\xi(\mathcal T_\rho)$,
$\int |Q_h|\,d\nu_h=0$, and
\begin{gather*}
  |\sigma(\mathcal N_\rho)|^2
  +\left|\int_{\mathcal N_\rho}x^d\,d\sigma(x)\right|^2
  =\xi(\mathcal T_\rho)^2
   +\left|\Delta_h(\xi)-\int_{\mathcal T_\rho}x^d\,d\xi(x)\right|^2
  \le C_d\bigl\{\Delta_h(\xi)^2+\xi(\mathcal T_\rho)^2\bigr\}, \\
  \left(\int_{\mathcal N_\rho}|Q_h|\,d|\sigma|\right)^2
  =\left(\int_{\mathcal N_\rho}|Q_h|\,d\xi\right)^2
  \le\int_{\mathcal N_\rho}Q_h^2\,d\xi
  =A_{h,\rho}(\xi).
\end{gather*}
Here the first estimate uses $x^d\le2^d$ on $[0,2]$, while the second uses
Cauchy--Schwarz.
Using these estimates together with the uniform bounds on the coefficients
and one-variable functions gives
\begin{gather*}
  \left|\iint_{\mathcal N_\rho^2}P_h\,d\sigma d\sigma\right|
  \le C_d\bigl(\Delta_h(\xi)^2+\xi(\mathcal T_\rho)^2\bigr), \\
  \left|\iint_{\mathcal N_\rho^2}Q_h(x)\{u_{0,h}(x)+u_{d,h}(x)y^d\}\,d\sigma(x)d\sigma(y)\right|
  \le \frac a8A_{h,\rho}(\xi)+C_d\bigl\{\Delta_h(\xi)^2+\xi(\mathcal T_\rho)^2\bigr\}, \\
  \left|\iint_{\mathcal N_\rho^2}Q_h(y)\{v_{0,h}(y)+v_{d,h}(y)x^d\}\,d\sigma(x)d\sigma(y)\right|
  \le \frac a8A_{h,\rho}(\xi)+C_d\bigl\{\Delta_h(\xi)^2+\xi(\mathcal T_\rho)^2\bigr\}, \\
  \iint_{\mathcal N_\rho^2}Q_h(x)Q_h(y)\Theta_h(x,y) \,d\sigma(x)d\sigma(y)
  \ge \iint_{\mathcal N_\rho^2}Q_h(x)Q_h(y) \left(\Theta_h(x,y)-\frac{d^3}{4}\right) \,d\sigma(x)d\sigma(y)
  \ge -\frac{a}{4} A_{h,\rho}(\xi).
\end{gather*}
In the last line, the omitted constant part is the nonnegative square
\[
  \frac{d^3}{4}
  \left(\int_{\mathcal N_\rho}Q_h(x)\,d\sigma(x)\right)^2.
\]
The second and third bounds use
$uv \le (a/8)u^2+(2/a)v^2$.
Combining these bounds gives \eqref{eq:central-kernel-bound}.
\end{proof}

We now combine the preceding estimates to obtain a lower bound for the
auxiliary Lagrangian difference $\mathcal L_h(\nu)-\mathcal L_h(\nu_h)$.
This is the conclusion announced at the beginning of this subsection and is
recorded in the following lemma.

\begin{lemma}[Lower bound for the auxiliary Lagrangian difference]
\label{lem:auxiliary-lagrangian-bound}
There exist $\rho_0>0$, $a>0$, and $C_{d,m}\ge1$ such that
for every $0<\rho<\rho_0$, one can choose $b_\rho,h_\rho>0$ with the following property.
Let $0<h<h_\rho$, and let $W$ satisfy the hypotheses of
\Cref{lem:localization-rank-one}.  Write $W=f\otimes f+E$ for the
decomposition furnished there, and let $\nu$ be the distribution of the
values of $f$.  Then
\begin{equation}\label{eq:auxiliary-lagrangian-bound}
\begin{aligned}
  \mathcal L_h(\nu)-\mathcal L_h(\nu_h)
  &=\mathcal J_h(\nu)-\mathcal J_h(\nu_h)
    -\mu_h\{m_d(\nu)^v-q_h^v\}\\
  &\ge \frac{a}{2} A_{h,\rho}(\nu)
    +\frac{b_\rho}{2} \nu(\mathcal T_\rho)
    -C_{d,m}\Delta_h(\nu)^2.
\end{aligned}
\end{equation}
\end{lemma}

\begin{proof}
We combine the positive first-variation bound with the lower bound for the
quadratic kernel term.  The central part follows from
\Cref{lem:central-kernel-bound} with $\xi=\nu$; the parts meeting the
tail are controlled by the uniform bounds on $K_h$ and then absorbed using the
small tail mass from \eqref{eq:factor-tail-mass}.

Put $\sigma:=\nu-\nu_h$.  Applying the expansion following
\eqref{eq:first-variation} with $\xi=\nu$, and using
$\int \Psi_h\,d\nu_h=0$ and Taylor's theorem, gives
\begin{equation*}
  \mathcal L_h(\nu)-\mathcal L_h(\nu_h)
  =\mathcal J_h(\nu)-\mathcal J_h(\nu_h)
    -\mu_h\{m_d(\nu)^v-q_h^v\}
  \ge\int\Psi_h\,d\nu+\iint K_h\,d\sigma d\sigma
    -C_{d,m}\Delta_h(\nu)^2.
\end{equation*}
The central and tail bounds in \Cref{lem:first-variation-bound}
give $\rho_0$, $a$, $b_\rho$, and $h_\rho$ such that
\[
  \int\Psi_h\,d\nu
  \ge a A_{h,\rho}(\nu)+b_\rho\nu(\mathcal T_\rho)
  \qquad
  (0<h<h_\rho,\, 0<\rho<\rho_0).
\]
Now we bound the double integral.  Split the quadratic term according to
$\mathcal N_\rho$ and $\mathcal T_\rho$:
\[
  \iint K_h\,d\sigma d\sigma
  =\iint_{\mathcal N_\rho^2}K_h\,d\sigma d\sigma
   +2\iint_{\mathcal N_\rho\times\mathcal T_\rho}K_h\,d\sigma d\sigma
   +\iint_{\mathcal T_\rho^2}K_h\,d\sigma d\sigma .
\]
The mixed rectangles have equal integrals by the symmetry of $K_h$ and
Fubini's theorem, since the continued kernel is bounded on $[0,2]^2$.
The central-kernel bound \eqref{eq:central-kernel-bound} controls the
first integral.  It remains to control the two parts meeting
$\mathcal T_\rho$.
After decreasing $h_\rho$, we have
$\{s_h,t_h\}\subset\mathcal N_\rho$, and hence
$\sigma(\mathcal N_\rho)=-\nu(\mathcal T_\rho)$ and $\sigma=\nu$ on
$\mathcal T_\rho$.  Thus, for $y\in\mathcal T_\rho$,
\[
\begin{aligned}
  \int_{\mathcal N_\rho}K_h(x,y)\,d\sigma(x)
  &=\int_{\mathcal N_\rho}
    \{K_h(x,y)-K_h(u_h,y)\}\,d\sigma(x)
    -\nu(\mathcal T_\rho) K_h(u_h,y).
\end{aligned}
\]
Since $|\sigma|\le\nu+\nu_h$, the H\"older bound in
\Cref{lem:continuation-kernel-bounds}, together with
\eqref{eq:rank-one-reduction-bounds} and
\eqref{eq:rank-one-parameter-expansions}, gives
\begin{equation*}
\begin{aligned}
  \left|\int_{\mathcal N_\rho}
    \{K_h(x,y)-K_h(u_h,y)\}\,d\sigma(x)\right|
  \le C_d\int_{\mathcal N_\rho}|x-u_h|^{1/2}\,d(\nu+\nu_h)(x)
  \le C_d\left\{\|f-u_h\|_4^{1/2}+h^{1/2}\right\}
  \le C_dh^{1/2}.
\end{aligned}
\end{equation*}
The uniform bound on $K_h$ from
\Cref{lem:continuation-kernel-bounds} therefore gives
\[
  \left|\int_{\mathcal N_\rho}K_h(x,y)\,d\sigma(x)\right|
  \le C_dh^{1/2}+C_d\nu(\mathcal T_\rho).
\]
Integrating over $y\in\mathcal T_\rho$ and using the same uniform bound on
$K_h$ gives the two remaining estimates:
\[
  \left|2\iint_{\mathcal N_\rho\times\mathcal T_\rho} K_h\,d\sigma d\sigma\right|
  \le C_dh^{1/2}\nu(\mathcal T_\rho)+C_d\nu(\mathcal T_\rho)^2,
  \qquad
  \left|\iint_{\mathcal T_\rho^2}K_h\,d\sigma d\sigma\right|
  \le C_d\nu(\mathcal T_\rho)^2.
\]
Combining the first-variation, central, mixed, and tail bounds gives
\begin{equation*}
  \mathcal J_h(\nu)-\mathcal J_h(\nu_h)
    -\mu_h\{m_d(\nu)^v-q_h^v\}
  \ge\frac{a}{2}A_{h,\rho}(\nu)+b_\rho\nu(\mathcal T_\rho)
  -C_{d,m}\Delta_h(\nu)^2-C_d\nu(\mathcal T_\rho)^2
    -C_dh^{1/2}\nu(\mathcal T_\rho).
\end{equation*}
Using the tail bound \eqref{eq:factor-tail-mass}, decrease $h_\rho$ so that
$C_dh^{1/2}\le b_\rho/4$ and $C_d\nu(\mathcal T_\rho)\le b_\rho/4$.
Then, \eqref{eq:auxiliary-lagrangian-bound} follows.
\end{proof}

%% file: sections/singular_graphon_comparison.tex
\subsection{The arbitrary-graphon comparison}
\label{sec:graphon-comparison}

We now restore the remainder $E$ and convert the lower bound for the auxiliary
Lagrangian difference into the corresponding bound for the Lagrangian
difference between the actual graphons $W$ and $W_h$.
By \Cref{lem:localization-rank-one}, every feasible graphon satisfying
$I_{p_h}(W)\le I_{p_h}(W_h)$ can be written as $W=f\otimes f+E$, while
\Cref{lem:auxiliary-lagrangian-bound} controls the auxiliary Lagrangian
difference between $f$ and $f_h$.  The next lemma estimates the contribution
of $E$ and gives the required full-graphon bound.

\begin{lemma}[Lower bound for the full-graphon Lagrangian difference]
\label{lem:graphon-lagrangian-bound}
There exist $\rho_0>0$ and $C_{d,m}\ge1$ such that for
every $0<\rho<\rho_0$, one can choose $C_{d,\rho}\ge1$ and $h_\rho>0$ with
the following property.  Let $0<h<h_\rho$, and let $W$ satisfy the hypotheses of
\Cref{lem:localization-rank-one}.  Write $W=f\otimes f+E$ for the
decomposition provided there, and let $\nu$ be the distribution of the
values of $f$.  Then,
\begin{equation}\label{eq:graphon-lagrangian-bound}
\begin{aligned}
  &\left[I_{p_h}(W)-\mu_h\{t(H,W)-r_h^m\}\right]
   -\left[I_{p_h}(W_h)-\mu_h\{t(H,W_h)-r_h^m\}\right]\\
  &\qquad\qquad\qquad\qquad\qquad\qquad\qquad=I_{p_h}(W)-I_{p_h}(W_h)-\mu_h\{t(H,W)-r_h^m\}\\
  &\qquad\qquad\qquad\qquad\qquad\qquad\qquad
  \ge C_{d,\rho}^{-1}\!\left[A_{h,\rho}(\nu)+\nu(\mathcal T_\rho)+\|E\|_2^2\right]
  -C_{d,m}\Delta_h(\nu)^2.
\end{aligned}
\end{equation}
\end{lemma}

\begin{proof}
Put $q:=q(f)=m_d(\nu)$ and
$\Omega_\rho:=\{x:f(x)\in\mathcal N_\rho\}$.  Then
$\lambda(\Omega_\rho^c)=\nu(\mathcal T_\rho)$.
Combining the localization results in \Cref{lem:localization-rank-one}
and \eqref{eq:factor-tail-mass}, we may use throughout the proof
\begin{equation}\label{eq:graphon-comparison-estimates}
  \begin{gathered}
    \nu(\mathcal T_\rho)\le C_{d,\rho}h^4,
    \qquad |\Delta_h(\nu)|\le C_dh^2,
    \qquad \|E\|_2\le C_dh,
    \qquad \|f-u_*\|_4\le C_dh,\\
    |t(H,W)-q^v|\le C_{d,m}\|E\|_2^3,
    \qquad \lambda(\Omega_\rho)\ge\frac12,
    \qquad 0\le f\le2,
    \qquad \|E\|_\infty\le5.
  \end{gathered}
\end{equation}
The bound on $\lambda(\Omega_\rho)$ follows from
$\lambda(\Omega_\rho^c)=\nu(\mathcal T_\rho)$, while the last bound uses
$0\le W\le1$ and $0\le f\le2$.
Since
\[
  \mathcal J_h(\nu)
  =\iint\widetilde J_{p_h}(f(x)f(y))\,dx\,dy,
  \qquad
  \mathcal J_h(\nu_h)=I_{p_h}(W_h),
\]
and $t(H,W_h)=q_h^v=r_h^m$, the full-graphon Lagrangian difference has the
exact decomposition
\begin{equation*}
\begin{aligned}
  &I_{p_h}(W)-I_{p_h}(W_h)-\mu_h\{t(H,W)-r_h^m\}\\
  &\quad=\bigl[\mathcal L_h(\nu)-\mathcal L_h(\nu_h)\bigr]
  +\iint\{J_{p_h}(W(x,y))-\widetilde J_{p_h}(f(x)f(y))\}\,dx\,dy\\
  &\qquad-\mu_h\{t(H,W)-q^v\}.
\end{aligned}
\end{equation*}
The first term is controlled by \Cref{lem:auxiliary-lagrangian-bound}, and
the last by the homomorphism-density estimate in
\eqref{eq:rank-one-reduction-bounds}.  It remains to bound the middle
double integral over its central, mixed, and tail--tail parts.

We first consider the central square $\Omega_\rho^2$.  Decrease $\rho_0$ so
that $f(x)f(y)$ lies in a fixed closed subinterval of $(0,1)$ whenever
$x,y\in\Omega_\rho$.  Put
\[
  \|E\|_{2,G}
  :=\left( \iint_G E(x,y)^2\,dx\,dy \right)^{1/2}
  \qquad (G \subset[0,1]^2).
\]
On $\Omega_\rho^2$, the identity $\mathcal M_h'(xy)=J_{p_h}'(xy)-\gamma_h(xy)^{d-1}$ yields
the decomposition of the integrand:
\begin{equation*}
  J_{p_h}(W)-J_{p_h}(f\otimes f)
  = \bigl[ J_{p_h}(W)-J_{p_h}(f\otimes f) -(J_{p_h})'(f\otimes f)E \bigr]
  +\mathcal M_h'(f\otimes f)E
      +\gamma_h(f\otimes f)^{d-1}E.
\end{equation*}
In this step, we bound the integrals of these three terms on $\Omega_\rho^2$ separately.
Since $J_{p_h}''(z)=1/(z(1-z))\ge4$, strong convexity gives the first bound
\[
  \iint_{\Omega_\rho^2}
  \{J_{p_h}(W)-J_{p_h}(f\otimes f)-(J_{p_h})'(f\otimes f)E\}
  \ge2\|E\|_{2,\Omega_\rho^2}^2.
\]
For the second bound, we divide $\mathcal M_h'$ by $Q_h$ in each variable and write
\[
  \mathcal M_h'(xy)=Q_h(x)U_h(x,y)+Q_h(y)V_h(x,y)+C_h(x,y),
\]
where $C_h$ has degree at most one in each variable.  The factorization
\eqref{eq:mh-derivative-factorization} then gives
$\mathcal M_h'(xy)=0$ for every $(x,y)\in\{s_h,t_h\}^2$.  Since
$Q_h(s_h)=Q_h(t_h)=0$, the bilinear remainder $C_h$ must vanish identically.
At $h=0$, $\mathcal M_0'(xy)$ is an analytic multiple of $(xy-r_*)^3$ and
$Q_0(z)=(z-u_*)^2$, so $U_0$ and $V_0$ vanish at $(u_*,u_*)$.  By
joint continuity, after decreasing $\rho_0$ and $h_\rho$, we may ensure that
$|U_h|$ and $|V_h|$ are at most $\sqrt{a/4}$ on $\mathcal N_\rho^2$, and
\[
  |\mathcal M_h'(xy)|
  \le\sqrt{\frac a4}
    \{|Q_h(x)|+|Q_h(y)|\}
  \qquad(x,y\in\mathcal N_\rho,\ 0<h<h_\rho).
\]
Then, $uv \le (1/4)u^2+v^2$ and
$(u+v)^2\le 2(u^2+v^2)$ give
\begin{align*}
  \left|\iint_{\Omega_\rho^2} \{\mathcal M_h'(f(x)f(y))\} E(x,y)\,dx\,dy\right|
  &\le \frac{1}{4} \iint_{\Omega_\rho^2} \frac{a}{4}\{|Q_h(f(x))|+|Q_h(f(y))|\}^2 \,dx\,dy + \|E\|_{2,\Omega_\rho^2}^2 \\
  &\le \frac{a}{16} \iint_{\Omega_\rho^2} 2|Q_h(f(x))|^2+2|Q_h(f(y))|^2 \,dx\,dy + \|E\|_{2,\Omega_\rho^2}^2 \\
  &\le \frac a4 A_{h,\rho}(\nu) + \|E\|_{2,\Omega_\rho^2}^2.
\end{align*}
For the third bound, the orthogonality in \eqref{eq:rank-one-orthogonality} moves
the remaining linear term to the tail--tail square:
\[
  \left| \iint_{\Omega_\rho^2}\gamma_h(f(x)f(y))^{d-1}E(x,y)\,dx\,dy \right|
  = \left| \iint_{(\Omega_\rho^c)^2}\gamma_h(f(x)f(y))^{d-1}E(x,y)\,dx\,dy \right|
  \le C_d \nu(\mathcal T_\rho)^2.
\]
Combining the three bounds gives
\begin{equation*}
  \iint_{\Omega_\rho^2} \{J_{p_h}(W)-J_{p_h}(f\otimes f)\}
  \ge \|E\|_{2,\Omega_\rho^2}^2 -\frac a4 A_{h,\rho}(\nu)-C_d\nu(\mathcal T_\rho)^2.
\end{equation*}

We next consider the mixed rectangles.  For almost every $x$, the
orthogonality in
\eqref{eq:rank-one-orthogonality}, together with
\eqref{eq:graphon-comparison-estimates}, gives
\begin{equation*}
  u_*^{d-1}\left|\int_{\Omega_\rho}E(x,y)\,dy\right|
  \le\int_{\Omega_\rho}
    |f(y)^{d-1}-u_*^{d-1}|\,|E(x,y)|\,dy
  +\int_{\Omega_\rho^c}f(y)^{d-1}|E(x,y)|\,dy
  \le C_{d,\rho}h.
\end{equation*}
Consequently, if $\overline W_x:= (\int_{\Omega_\rho}W(x,y)\,dy)/\lambda(\Omega_\rho)$,
then
\begin{equation*}
  |\overline W_x-u_*f(x)|
  \le |f(x)|\left|
    \frac1{\lambda(\Omega_\rho)}\int_{\Omega_\rho} (f(y)-u_*)\,dy\right|
    +\frac1{\lambda(\Omega_\rho)}\left|\int_{\Omega_\rho}E(x,y)\,dy\right|
  \le C_{d,\rho}h.
\end{equation*}
Since $\overline W_x\in[0,1]$, we have
$J_{p_h}(\overline W_x)=\widetilde J_{p_h}(\overline W_x)$.
By Jensen's and H\"older's inequalities, with
\eqref{eq:continuation-kernel-bounds} and
\eqref{eq:graphon-comparison-estimates},
\begin{equation*}
\begin{aligned}
&\int_{\Omega_\rho}
  \{J_{p_h}(W(x,y))-\widetilde J_{p_h}(f(x)f(y))\}\,dy\\
&\quad\ge \lambda(\Omega_\rho) \left\{J_{p_h}(\overline W_x)-\widetilde J_{p_h}(u_*f(x))\right\}
-\int_{\Omega_\rho} \left\{\widetilde J_{p_h}(f(x)f(y))-\widetilde J_{p_h}(u_*f(x))\right\}\,dy\\
&\quad\ge-C_d |\overline W_x - u_*f(x)|^{1/2} -C_d\int_{\Omega_\rho}|f(y)-u_*|^{1/2}\,dy \ge-C_{d,\rho}h^{1/2} \ge -\frac{b_\rho}{16}.
\end{aligned}
\end{equation*}
Integrating this bound over $x\in\Omega_\rho^c$ yields
\begin{equation*}
\begin{aligned}
  2\iint_{\Omega_\rho^c\times \Omega_\rho}
  \{J_{p_h}(W)-\widetilde J_{p_h}(f\otimes f)\}
  \ge -\frac{b_\rho}{8}\nu(\mathcal T_\rho).
\end{aligned}
\end{equation*}

It remains to consider the tail--tail square.  The uniform bound in
\eqref{eq:continuation-kernel-bounds} gives
\[
  \iint_{(\Omega_\rho^c)^2}
  \{J_{p_h}(W)-\widetilde J_{p_h}(f\otimes f)\}
  \ge-C_d\nu(\mathcal T_\rho)^2.
\]

Combining all bounds with \eqref{eq:auxiliary-lagrangian-bound}
and $|t(H,W)-q^v|\le C_{d,m}\|E\|_2^3$ gives
\begin{equation*}
\begin{aligned}
  I_{p_h}(W)-I_{p_h}(W_h)-\mu_h\{t(H,W)-r_h^m\}
  &\ge\frac a4 A_{h,\rho}(\nu)
    +\frac{3b_\rho}{8}\nu(\mathcal T_\rho)
    +\|E\|_{2,\Omega_\rho^2}^2\\
  &\quad-C_{d,m}\Delta_h(\nu)^2
    -C_d\nu(\mathcal T_\rho)^2
    -C_{d,m}\|E\|_2^3.
\end{aligned}
\end{equation*}
Since $\|E\|_\infty\le5$ and
$\lambda([0,1]^2\setminus\Omega_\rho^2)\le2\nu(\mathcal T_\rho)$, choose
$C_{d,\rho}\ge1$ so that $50C_{d,\rho}^{-1}\le b_\rho/16$.  Then
\begin{equation*}
  \|E\|_{2,\Omega_\rho^2}^2
  \ge C_{d,\rho}^{-1}\|E\|_{2,\Omega_\rho^2}^2
  =C_{d,\rho}^{-1}\left(
    \|E\|_2^2-\|E\|_{2,[0,1]^2\setminus\Omega_\rho^2}^2\right)
  \ge C_{d,\rho}^{-1}\|E\|_2^2
    -50C_{d,\rho}^{-1}\nu(\mathcal T_\rho)
  \ge C_{d,\rho}^{-1}\|E\|_2^2
    -\frac{b_\rho}{16}\nu(\mathcal T_\rho).
\end{equation*}
By \eqref{eq:graphon-comparison-estimates}, we may also decrease $h_\rho$ so that
$C_d\nu(\mathcal T_\rho)\le b_\rho/16$ and $C_{d,m}\|E\|_2\le C_{d,\rho}^{-1}/2$.
Consequently,
\begin{equation*}
  I_{p_h}(W)-I_{p_h}(W_h)-\mu_h\{t(H,W)-r_h^m\}
  \ge\frac a4 A_{h,\rho}(\nu)
  +\frac{b_\rho}{4}\nu(\mathcal T_\rho)
  +\frac{1}{2}C_{d,\rho}^{-1}\|E\|_2^2
  -C_{d,m}\Delta_h(\nu)^2.
\end{equation*}
Increasing $C_{d,\rho}$ if necessary proves \eqref{eq:graphon-lagrangian-bound}.
\end{proof}

%% file: sections/singular_proof.tex
\subsection{Proof of the singular-endpoint theorem}
\label{sec:singular-proof}

We finish by applying the lower bound for the full-graphon Lagrangian
difference to the decomposition supplied by the localization and rank-one
reduction.

\begin{proof}[Proof of \Cref{thm:singular-endpoint}]
\Cref{lem:rank-one-kkt-family} provides the real-analytic map
$h\mapsto(p_h,r_h,u_h,\alpha_h)$ for $|h|<h_0$, with the desired
values at $h=0$.  For every measurable $A_h\subset[0,1]$ of measure
$\alpha_h$, the corresponding factor is $f_{\alpha_h,s_h,t_h}$, up to measure-preserving relabeling.
Since $t_h-s_h=2h\ne0$ for $h>0$, the resulting graphon $W_h$ is
nonconstant, rank-one, and bipodal, and the same lemma gives
$t(H,W_h)=r_h^m$.  Its comparison with the constant graphon and
the inequality $p_h<\pc(r_h)$ follow from \Cref{lem:constant-graphon-comparison}.

It remains to prove optimality and uniqueness.
Fix $0<\rho<\rho_0$ and $0<h<h_\rho$ as in \Cref{lem:graphon-lagrangian-bound}.
Assume $I_{p_h}(W)\le I_{p_h}(W_h)$ and $W$ is feasible.
By \Cref{lem:localization-rank-one}, write
\[
  W=f\otimes f+E,
  \qquad
  T_E(f^{d-1})=0,
  \qquad
  q:=q(f).
\]
Let $\nu$ be the distribution of the values of $f$, so that $q=m_d(\nu)$ and
$\Delta_h(\nu)=q-q_h$.  Applying \Cref{lem:graphon-lagrangian-bound} gives
\eqref{eq:graphon-lagrangian-bound}.  Its left-hand side is nonpositive:
\[
  I_{p_h}(W)-I_{p_h}(W_h)-\mu_h\{t(H,W)-r_h^m\}
  \le I_{p_h}(W)-I_{p_h}(W_h)\le0,
\]
because $W$ is feasible and $\mu_h>0$.  Comparing this upper bound with the
lower bound in \eqref{eq:graphon-lagrangian-bound}, and increasing the
constant if necessary, gives $C_{d,m,\rho}\ge1$ such that
\[
  A_{h,\rho}(\nu)+\nu(\mathcal T_\rho)+\|E\|_2^2
  \le C_{d,m,\rho}^2\Delta_h(\nu)^2,
  \qquad
  \|E\|_2\le C_{d,m,\rho}|\Delta_h(\nu)|.
\]
Combining the homomorphism-density bound in
\eqref{eq:rank-one-reduction-bounds} with Taylor's theorem gives
\begin{equation*}
\begin{aligned}
  \left|t(H,W)-r_h^m-vq_h^{v-1}\Delta_h(\nu)\right|
  \le 
  |t(H,W)-q^v|
    +\left|(q_h+\Delta_h(\nu))^v-q_h^v
      -vq_h^{v-1}\Delta_h(\nu)\right|
  \le C_{d,m,\rho}\Delta_h(\nu)^2.
\end{aligned}
\end{equation*}
Since $q_h$ stays bounded away from zero, the linear term in the preceding
estimate dominates the quadratic error when $\Delta_h(\nu)<0$.  Feasibility
would then be violated, so, after decreasing $h_\rho$, we have
$\Delta_h(\nu)\ge0$.  The localization estimate
\eqref{eq:rank-one-reduction-bounds}, together with the positive lower
bounds for $q_h$ and $\mu_h$, allows us to decrease $h_\rho$ further so that
\[
  C_{d,m,\rho}\Delta_h(\nu)\le\frac14vq_h^{v-1},
  \qquad
  C_{d,m,\rho}\Delta_h(\nu)\le\frac14\mu_hvq_h^{v-1}.
\]
This yields
\begin{equation*}
  \mu_h\{t(H,W)-r_h^m\}-C_{d,m,\rho}\Delta_h(\nu)^2
  \ge\frac12\mu_hvq_h^{v-1}\Delta_h(\nu).
\end{equation*}
Substituting this into \eqref{eq:graphon-lagrangian-bound} gives
\begin{equation*}
  I_{p_h}(W)-I_{p_h}(W_h)
  \ge
  \left(\frac12\mu_hvq_h^{v-1}\Delta_h(\nu)
    +C_{d,m,\rho}^{-1}
    \left[A_{h,\rho}(\nu)+\nu(\mathcal T_\rho)+\|E\|_2^2\right]\right)
  \ge0.
\end{equation*}
The left-hand side is nonpositive by the upper bound above, so equality holds
throughout.  Consequently,
\begin{equation*}
  \Delta_h(\nu)=0,
  \qquad A_{h,\rho}(\nu)=0,
  \qquad \nu(\mathcal T_\rho)=0,
  \qquad E=0\quad\text{a.e.}
\end{equation*}
Since $A_{h,\rho}(\nu)=0$ and $\nu(\mathcal T_\rho)=0$,
$\nu$ is supported on $\{s_h,t_h\}$, and $\Delta_h(\nu)=0$ implies
$\nu=\alpha_h\delta_{s_h}+(1-\alpha_h)\delta_{t_h}$.
Moreover, $E\equiv0$ implies that $W=f\otimes f$ and $W$ is a measure-preserving relabeling of $W_h$.
\end{proof}

%% file: sections/lean.tex
\section{Lean formalization}\label{sec:lean-formalization}

A substantial part of the paper has been formalized in Lean (v4.33.1) using
Mathlib (v4.33.1).  The code is available in the
\href{https://github.com/Lim-Sangho/upper-tail-optimizers-lean}{online
repository}.\footnote{https://github.com/Lim-Sangho/upper-tail-optimizers-lean}
The formalization machine-checks the deterministic graphon arguments
underlying our results and contains no \texttt{sorry} placeholders.

\paragraph{Nonexceptional theory.}
\Cref{thm:nonexceptional-optimizers}, our main result away from the
exceptional density, has been formalized.
The development includes the construction and analysis of the
Lubetzky--Zhao boundary in
\Cref{sec:lz-boundary}, the reduction to a one-dimensional problem in
\Cref{sec:local-reduction}, the positivity of $A_H$ in
\Cref{sec:nonexceptional-quadratic-growth}, the uniqueness and analyticity results and the
expansions of the edge density and rate function in
\Cref{sec:nonexceptional-proof}, and the parameter expansions derived in
\Cref{app:bipodal-parameter-expansions}.
All results of Kenyon, Radin, Ren, and Sadun~\cite{kenyon2014entropy}
used in this paper, as stated in \Cref{thm:krrs-bipodality,thm:krrs-cross-density},
have also been formalized in Lean in the $d$-regular setting used here.
The formalized proofs use different
arguments from those in the original paper; these alternative proofs are
not described here.  The proof of \Cref{thm:krrs-analytic-extension} in
Appendix~\ref{app:krrs-analytic-extension} has also been formalized.
These results are combined into a single formal statement of
\Cref{thm:nonexceptional-optimizers}.

\paragraph{Exceptional endpoint.}
\Cref{thm:endpoint-optimizers}, our exceptional-endpoint result, has also
been formalized.
The analytic construction of the rank-one family converging to the
constant graphon at the singular endpoint, its parameter expansions and its strict
improvement in relative entropy over the corresponding constant graphon,
the localization and rank-one reduction, the rank-one
factor comparison, and the arbitrary-graphon comparison are
machine-checked; not every auxiliary lemma is formalized in the exact
form stated in the paper.  The formalization of \Cref{thm:singular-endpoint}
proves both global optimality and uniqueness up to measure-preserving relabeling.
These results are combined into a single formal statement of
\Cref{thm:endpoint-optimizers}.

\paragraph{External results.}
The only results assumed as axioms are the standard graphon
compactness, continuity, and lower-semicontinuity properties used in the paper,
the generalized H\"older inequality \eqref{eq:generalized-holder}, and
the Lubetzky--Zhao criterion \Cref{thm:lz-criterion}.
Each is assumed under the hypotheses stated in its source; all other results
described above are derived from these axioms in Lean.
The formalization does not include the Chatterjee--Varadhan large deviation
principle and therefore does not cover the probabilistic consequences stated in
the introduction.

%% file: sections/conclusion.tex
\section{Concluding remarks}
\label{sec:concluding-remarks}

In this paper, we analyze the optimizers of
the Chatterjee--Varadhan variational problem for regular graphs.
More precisely, \Cref{thm:nonexceptional-optimizers} establishes the uniqueness
and bipodal structure of the nonconstant optimizers in a small neighborhood of
the Lubetzky--Zhao boundary except at the exceptional density $r_*=(d-1)/d$.
\Cref{thm:endpoint-optimizers} constructs an analytic curve approaching
the boundary point $(\pc(r_*),r_*)$ at this exceptional density.
Along this curve, the unique optimizers are rank-one and bipodal, and
both block sizes tend to $1/2$.  This differs from the vanishing-block
behavior as $p\uparrow\pc(r)$ at a fixed nonexceptional density~$r$.

A natural question is whether the bipodality and uniqueness established in
\Cref{thm:nonexceptional-optimizers,thm:endpoint-optimizers} persist throughout
the symmetry-breaking region.
For $H=K_3$, the experiments in \Cref{fig:neural-bipodal-patterns}, based on
the neural representation of Kim, Baek, Lee, and Yang~\cite{KimBaekLeeYangGraphons},
support bipodality throughout this region.
Restricting to bipodal graphons, the experiments in \Cref{fig:bipodal-parameter-maps}
use the Sequential Least Squares Programming algorithm~\cite{Kraft1988SLSQP} to track the two within-block
edge densities, the cross-block edge density, and the size of the denser block as $(p,r)$ varies.
These computations suggest a curve of rank-one optimizers approaching the
phase boundary at $P_*=(\pc(r_*),r_*)$,
which motivated the rank-one bipodal family constructed in \Cref{sec:rank-one-stationary-family}.
We extract this curve from sign changes of
$D:=q_{11}q_{22}-q_{12}^2$ along the $r$-grid at fixed $p$.
The experiments also suggest another curve along which uniqueness fails and across
which some of the bipodal parameters change discontinuously.
Such discontinuities would prevent an analytic continuation of the bipodal
parameters throughout the region, even if the optimizers remain bipodal.
Understanding the optimizers throughout this region and resolving the
bipodality question raised by Lubetzky and
Zhao~\cite[Section~7]{lubetzky2015large} would therefore require methods that can
accommodate possible nonuniqueness and discontinuous changes in the
parameters.

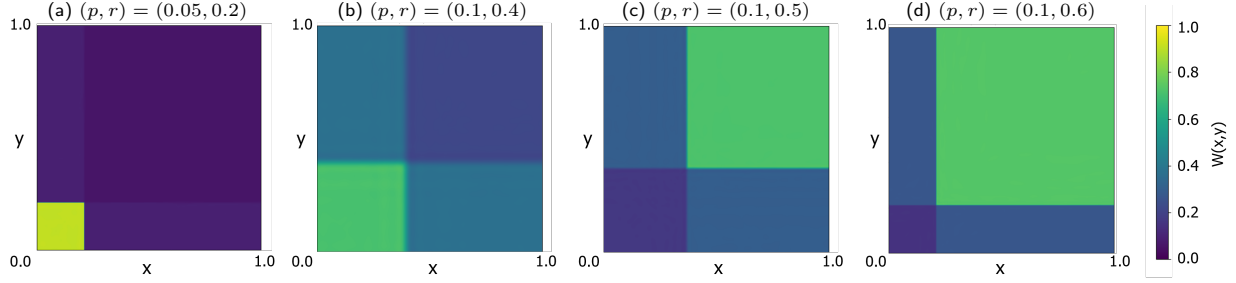
\begin{figure}[t]
\centering
\begin{tikzpicture}
\node[anchor=south west,inner sep=0] (parameterplots) at (0,0)
  {\includegraphics[width=\textwidth]{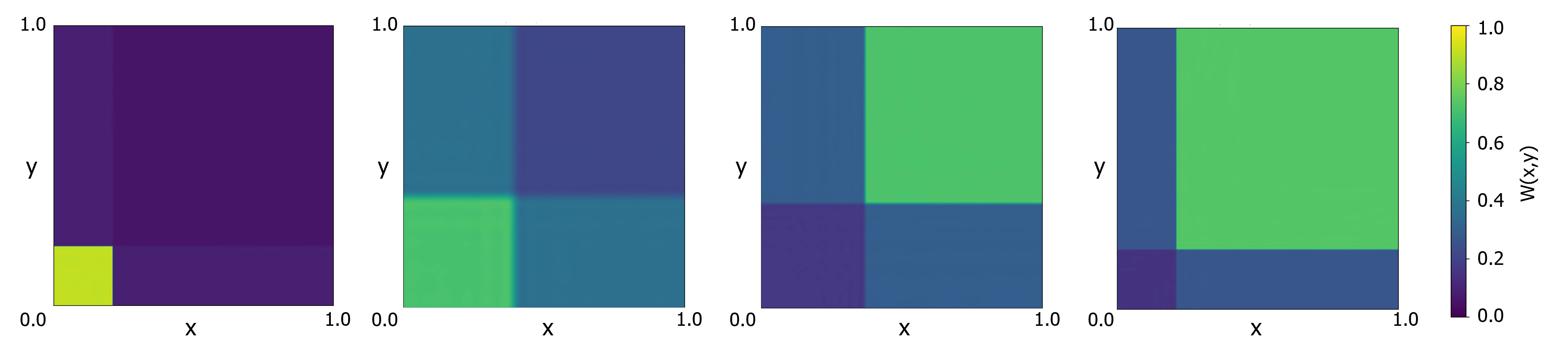}};
  \begin{scope}[x={(parameterplots.south east)},y={(parameterplots.north west)}]
    \node[font=\sffamily\fontsize{7}{7}\selectfont] at (0.122,0.98)
      {(a) $(p,r)=(0.05,0.2)$};
    \node[font=\sffamily\fontsize{7}{7}\selectfont] at (0.349,0.98)
      {(b) $(p,r)=(0.1,0.4)$};
    \node[font=\sffamily\fontsize{7}{7}\selectfont] at (0.575,0.98)
      {(c) $(p,r)=(0.1,0.5)$};
    \node[font=\sffamily\fontsize{7}{7}\selectfont] at (0.802,0.98)
      {(d) $(p,r)=(0.1,0.6)$};
  \end{scope}
\end{tikzpicture}
\caption{Graphons learned for $H=K_3$ by optimizing
\eqref{eq:graphon-variational-problem} using the neural representation~\cite{KimBaekLeeYangGraphons}.
Each heatmap shows $W(x,y)$ for the indicated parameter pair $(p,r)$ in the symmetry-breaking region;
color denotes edge probability.  The two-block patterns support an
affirmative answer to the bipodality question.}
\label{fig:neural-bipodal-patterns}
\end{figure}

\begin{figure}[t]
\centering
\begin{tikzpicture}
\node[anchor=south west,inner sep=0] (parameterplots) at (0,0)
  {\includegraphics[width=\textwidth]{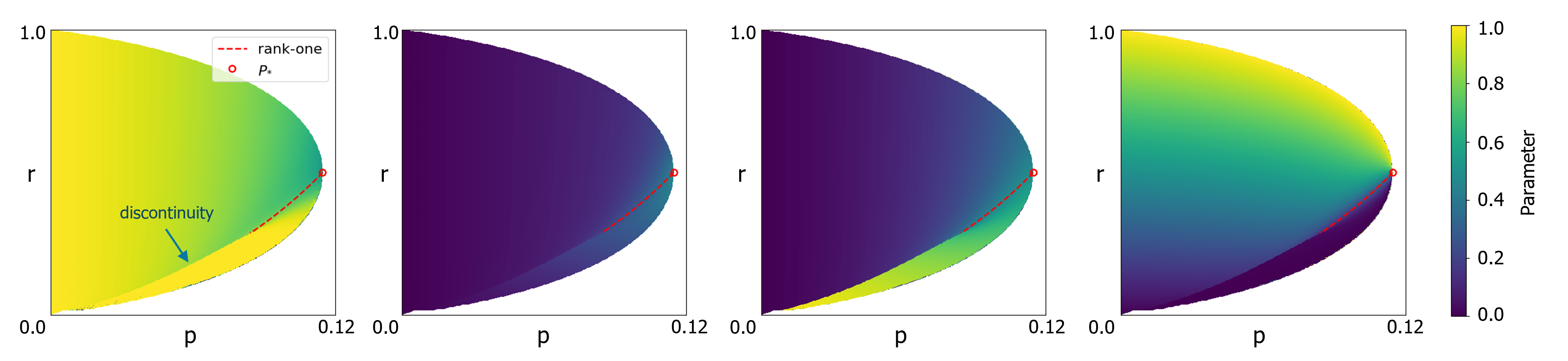}};
  \begin{scope}[x={(parameterplots.south east)},y={(parameterplots.north west)}]
    \node[font=\sffamily\fontsize{7}{7}\selectfont] at (0.122,0.98)
      {(a) Higher within-block density};
    \node[font=\sffamily\fontsize{7}{7}\selectfont] at (0.349,0.98)
      {(b) Lower within-block density};
    \node[font=\sffamily\fontsize{7}{7}\selectfont] at (0.575,0.98)
      {(c) Cross-block density};
    \node[font=\sffamily\fontsize{7}{7}\selectfont] at (0.802,0.98)
      {(d) Size of the denser block};
  \end{scope}
\end{tikzpicture}
\caption{Numerically optimized bipodal parameters for $H=K_3$ in the
symmetry-breaking region.  Panels (a)--(d) show the higher and lower
within-block edge densities, the cross-block density, and the size of the
block with higher within-block density.
The dashed red curve is the numerical near-rank-one curve approaching $P_*$ (red circle),
motivating the rank-one bipodal family constructed in
\Cref{sec:rank-one-stationary-family}.  The blue arrows indicate another
curve across which the bipodal parameters change discontinuously.}
\label{fig:bipodal-parameter-maps}
\end{figure}
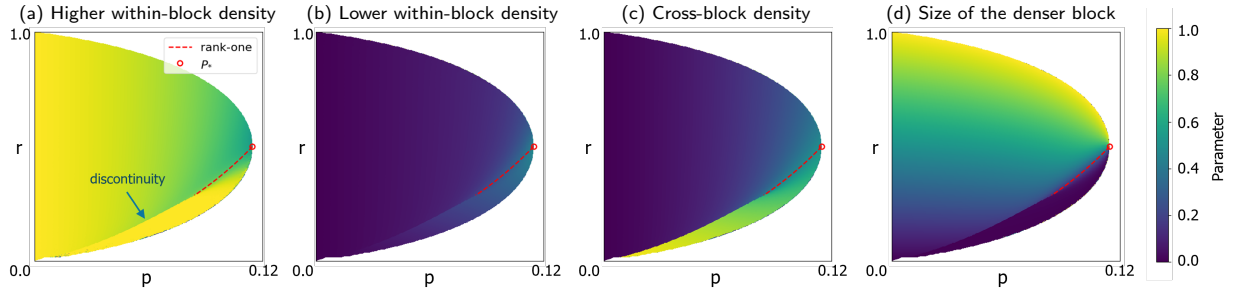

Our numerical observations for $H=K_3$ suggest the following conjecture for
general regular graphs.  Its final assertion extends the $\alpha_*=1/2$ case proved
in \Cref{thm:endpoint-optimizers} to arbitrary limiting smaller-block
sizes $\alpha_*\in(0,1/2]$.
\begin{conj}
\label{conj:symmetry-breaking-optimizers}
Let $H$ be a $d$-regular graph with $d\ge2$, and write
$\mathcal B_H:=\{(p,r)\in(0,1)^2:p<\pc(r)\}$ for the symmetry-breaking region.
The optimizers of \eqref{eq:graphon-variational-problem} have the following properties.
\begin{itemize}
  \item Every optimizer in $\mathcal B_H$ is bipodal.

  \item The set of points in $\mathcal B_H$ with nonunique optimizers,
  up to relabeling, is the curve
  \[
    \mathcal C_H=\{(\pi_H(r),r):r\in D_H\}\subseteq\mathcal B_H,
  \]
  where $D_H\subseteq(0,1)$ is a nonempty interval and
  $\pi_H:D_H\to(0,1)$ is strictly increasing and analytic.

  \item On $\mathcal B_H\setminus\mathcal C_H$, the bipodal parameters
  of the unique optimizer are locally analytic in $(p,r)$.
  For each $r\in D_H$, their left and right limits as $p\to\pi_H(r)^{\pm}$ exist and are
  distinct even after relabeling.

  \item For every $\alpha_*\in(0,1/2]$, there is a continuous path
  $\eta_{\alpha_*}:(0,1)\to\mathcal B_H\setminus\mathcal C_H$ such that, as
  $\varepsilon\downarrow0$, $\eta_{\alpha_*}(\varepsilon)\to P_*$ and the smaller
  block of its optimizer has size tending to $\alpha_*$.
\end{itemize}
\end{conj}

%% file: sections/appendix_preliminaries.tex
\section{Proof of the extended KRR--S theorem}\label{app:krrs-analytic-extension}

We first record the results from~\cite{kenyon2014entropy} needed for the
proof of \Cref{thm:krrs-analytic-extension}, restricting throughout to
$d$-regular graphs.
Their Theorem~1.1 describes the optimizer parameters only for
\(\vartheta:=\tau-\eps^m>0\).  The rest of this section constructs their
two-sided analytic continuation and verifies that it agrees with those
optimizer parameters where they are defined.

\begin{theorem}[{KRR--S bipodality for regular graphs, \cite[Theorem~1.1]{kenyon2014entropy}}]
\label{thm:krrs-bipodality}
Let \(H\) be a \(d\)-regular graph with \(d\ge2\) and \(m\ge2\) edges, and let
\(\eps\in(0,1)\setminus\{(d-1)/d\}\).  There exists
\(\tau_0(\eps)>\eps^m\) such that, whenever
\(\eps^m<\tau<\tau_0(\eps)\), the entropy maximizer subject to
\[
  e(W)=\eps,
  \qquad
  t(H,W)=\tau
\]
is unique up to relabeling and is bipodal.  Let $c$ be the size of the smaller block,
and let $q_{11},q_{12},q_{22}$ denote the edge probabilities within the smaller block (of size $c$),
between the two blocks, and within the larger block (of size $1-c$), respectively.
Then, its parameters \((c, q_{11},q_{12},q_{22})\) are jointly
real-analytic in \((\eps,\tau)\) on the region
\[
  \left\{(\eps,\tau):
    \eps\ne\frac{d-1}{d},\quad
    \eps^m<\tau<\tau_0(\eps)\right\}.
\]
For each fixed nonexceptional \(\eps\), as \(\tau\downarrow\eps^m\),
\[
  q_{11}\longrightarrow a_d(\eps),
  \qquad
  q_{12}\longrightarrow\zeta_d(\eps),
  \qquad
  q_{22}\longrightarrow\eps,
  \qquad
  c=O_{H,\eps}(\tau-\eps^m),
\]
where \(a_d(\eps)\in(0,1)\) and \(\zeta_d(\eps)\ne\eps\).
\end{theorem}

\begin{theorem}[{KRR--S cross density for regular graphs, \cite[Theorem~3.3]{kenyon2014entropy}}]
\label{thm:krrs-cross-density}
For every \(d\ge2\) and \(\eps\in(0,1)\), the function
\(z\mapsto\psi_d(\eps,z)\) has a unique critical point, denoted
\(\zeta_d(\eps)\), at which it attains its maximum on \((0,1)\).
The map \(\zeta_d:(0,1)\to(0,1)\) is strictly decreasing, and is an involution, i.e.,
$\eps=\zeta_d(\zeta_d(\eps))$.  Its unique fixed point is \((d-1)/d\).
\end{theorem}

\begin{remark}\label{rmk:krrs-omitted-identity}
We omit the identity \(S_0'(q_{11})=2S_0'(q_{12})-S_0'(q_{22})\)
asserted in \cite[Theorem~1.1]{kenyon2014entropy} and the claim in
\cite[Theorem~3.3]{kenyon2014entropy} that $\zeta_d'$ is nowhere zero.
Neither assertion is used in our subsequent arguments.
In place of the identity for $q_{11}$, we use only that its limiting value
$a_d(\eps)$ lies in $(0,1)$.
\end{remark}

Before proving \Cref{thm:krrs-analytic-extension}, we clarify what must be added to
\Cref{thm:krrs-bipodality}.  Although the latter states the threshold
\(\tau_0(\eps)\) pointwise, its joint-analyticity assertion implies that, for
every nonexceptional \(\eps_0\), there are a neighborhood \(U\) of \(\eps_0\)
and \(\Delta>0\) such that the theorem statement holds uniformly for
\(U\times(0,\Delta)\).  It does not, however, provide an analytic extension
through \(\vartheta=0\).  The main purpose of the following proof is to construct this
two-sided jointly analytic extension, which is needed for the boundary
analysis in
\Cref{thm:positive-second-variation,thm:nonexceptional-endpoint}.

\begin{proof}[Proof of \Cref{thm:krrs-analytic-extension}]
Let $\eps_*:=r_*=(d-1)/d$
and fix
\[
  \eps_0\in(0,1)\setminus\{\eps_*\},
  \qquad
  a_0:=a_d(\eps_0),
  \qquad
  z_0:=\zeta_d(\eps_0).
\]
By \Cref{thm:krrs-bipodality}, \(z_0\ne\eps_0\) and the parameters converge to
\begin{equation}\label{eq:krrs-parameter-limits}
  (q_{11},q_{12},q_{22},c)\longrightarrow(a_0,z_0,\eps_0,0).
\end{equation}
The local Taylor remainders below are uniform on sufficiently small fixed
neighborhoods of this base point, chosen in terms of $H$ and $\eps_0$.

\paragraph{1. Nondegeneracy of the scalar maximizer.}
For fixed \(\eps\), put
\begin{align}
  \mathcal N_\eps(z)
  &:=2\bigl[S_0(z)-S_0(\eps)-S_0'(\eps)(z-\eps)\bigr],\label{eq:krrs-entropy-remainder}\\
  \mathcal D_\eps(z)
  &:=z^d-\eps^d-d\eps^{d-1}(z-\eps).\label{eq:krrs-moment-remainder}
\end{align}
Thus, \(\psi_d(\eps,z)=\mathcal N_\eps(z)/\mathcal D_\eps(z)\) when
\(z\ne\eps\).  We claim that
\begin{equation}\label{eq:krrs-scalar-nondegeneracy}
  \partial_z^2\psi_d(\eps_0,z_0)<0.
\end{equation}
To see this, \(z_0\) is an interior maximizer, so its second derivative is
nonpositive.  Suppose for contradiction that it vanishes.  The third
derivative must then vanish as well, since an interior local maximum cannot
have a first nonzero derivative of odd order.  At \((\eps_0,z_0)\) we have
\(z_0\ne\eps_0\), and hence the identity
\(\mathcal N_{\eps_0}=\psi_d(\eps_0,\cdot)\mathcal D_{\eps_0}\) holds.  Since the first three derivatives of \(\psi_d\)
at \(z_0\) would vanish, by differentiating this identity we would obtain
\begin{equation}\label{eq:krrs-ratio-derivatives}
  \mathcal N_{\eps_0}''(z_0)
    =\psi_d(\eps_0,z_0)\mathcal D_{\eps_0}''(z_0),
  \qquad
  \mathcal N_{\eps_0}'''(z_0)
    =\psi_d(\eps_0,z_0)\mathcal D_{\eps_0}'''(z_0).
\end{equation}
The relevant derivatives are
\begin{align*}
  \mathcal N_\eps''(z)&=-\frac1{z(1-z)},
  &
  \mathcal N_\eps'''(z)&=\frac{1-2z}{z^2(1-z)^2},\\
  \mathcal D_\eps''(z)&=d(d-1)z^{d-2},
  &
  \mathcal D_\eps'''(z)&=d(d-1)(d-2)z^{d-3}.
\end{align*}
If \(d=2\), then \(\mathcal D_{\eps_0}'''\equiv0\), so the second identity
in \eqref{eq:krrs-ratio-derivatives} gives
\(\mathcal N_{\eps_0}'''(z_0)=0\), hence \(z_0=1/2=\eps_*\).  If
\(d\ge3\), divide the second identity in
\eqref{eq:krrs-ratio-derivatives} by the first; the division is legitimate
because both sides of the first equal
\(\mathcal N_{\eps_0}''(z_0)=-1/(z_0(1-z_0))\ne0\), and it yields
\[
  \frac{2z_0-1}{z_0(1-z_0)}=\frac{d-2}{z_0},
\]
hence again \(z_0=(d-1)/d=\eps_*\).  In either case, this implies $\eps_0 = (d-1)/d$,
contrary to the choice of \(\eps_0\).  This proves
\eqref{eq:krrs-scalar-nondegeneracy}.

\paragraph{2. An analytic parametrization of the two constraints.}
Write a bipodal graphon, with the first pode of size \(c\), as
\[
  W(a,b,q,c),
  \qquad
  (a,b,q)=(q_{11},q_{12},q_{22}).
\]
Its edge density is
\[
  \mathcal E(a,b,q,c)
  =c^2a+2c(1-c)b+(1-c)^2q.
\]
For prescribed edge density \(\eps\), solve this equation exactly for \(q\):
\begin{equation}\label{eq:krrs-edge-constraint}
  q=Q(\eps,a,b,c)
  :=\frac{\eps-c^2a-2c(1-c)b}{(1-c)^2}.
\end{equation}
The function \(Q\) is real-analytic near
\((\eps_0,a_0,z_0,0)\), and
\begin{equation}\label{eq:krrs-large-block-expansion}
  Q(\eps,a,b,c)
  =\eps+2(\eps-b)c+(3\eps-2b-a)c^2+O(c^3).
\end{equation}
Here the remainder has an absolute constant, uniformly for
$(\eps,a,b)\in[0,1]^3$ and $|c|\le1/2$.
Define the edge-constrained \(H\)-density and entropy by
\begin{align*}
  \widehat{\mathcal T}(\eps,a,b,c)
    &:=t\bigl(H,W(a,b,Q(\eps,a,b,c),c)\bigr),\\
  \widehat{\mathcal S}(\eps,a,b,c)
    &:=c^2S_0(a)+2c(1-c)S_0(b)
       +(1-c)^2S_0\bigl(Q(\eps,a,b,c)\bigr).
\end{align*}

Let \(n_1\) and \(n_d\) be the numbers of vertices of \(H\) of degrees
\(1\) and \(d\), respectively, and put \(v=n_1+n_d\).  Then
\(n_1+dn_d=2m\), and \(n_d>0\).  Expanding the homomorphism-density sum
according to the set of vertices assigned to the first pode gives
\begin{align}
  \left.\partial_c\widehat{\mathcal T}(\eps,a,b,c)\right|_{c=0}
  &=\sum_{x\in V(H)}b^{\deg(x)}\eps^{m-\deg(x)}
    -v\eps^m+2m\eps^{m-1}(\eps-b)\notag\\
  &=n_d\eps^{m-d}
      \bigl[b^d-\eps^d-d\eps^{d-1}(b-\eps)\bigr].
      \label{eq:krrs-density-derivative}
\end{align}
To verify the first equality, consider first an assignment in which \(x\) is
the unique vertex in the first pode.  Its contribution to
\(\widehat{\mathcal T}\) is
\[
  c(1-c)^{v-1}b^{\deg(x)}
  Q(\eps,a,b,c)^{m-\deg(x)},
\]
whose derivative at \(c=0\) is
\(b^{\deg(x)}\eps^{m-\deg(x)}\).  Summing over \(x\in V(H)\) gives the first
term.  If at least two vertices are assigned to the first pode, the
corresponding summand contains a factor \(c^k\) for some \(k\ge2\), while
all its other factors are analytic near \(c=0\); hence its derivative at
\(c=0\) vanishes.  Finally, the all-second-pode assignment contributes
\((1-c)^vQ(\eps,a,b,c)^m\).  Since \(Q(\eps,a,b,0)=\eps\) and
\(\partial_cQ|_{c=0}=2(\eps-b)\), its derivative at \(c=0\) is
\[
  -v\eps^m+m\eps^{m-1}\,2(\eps-b).
\]
This verifies the first equality in \eqref{eq:krrs-density-derivative}.
The second equality follows from the two identities for \(n_1,n_d\) above.

Set
\begin{equation}\label{eq:krrs-density-coefficient}
  \mathcal A(\eps,b)
  :=n_d\eps^{m-d}\mathcal D_\eps(b),
\end{equation}
which is equal to $\left.\partial_c\widehat{\mathcal T}(\eps,a,b,c)\right|_{c=0}$ by \eqref{eq:krrs-density-derivative}.
Strict convexity of \(z\mapsto z^d\) gives
\(\mathcal D_\eps(b)>0\) whenever \(b\ne\eps\).  In particular,
\(\mathcal A(\eps_0,z_0)>0\).  The analytic implicit function theorem,
applied to
\[
  \widehat{\mathcal T}(\eps,a,b,c)=\eps^m+\vartheta,
\]
gives neighborhoods of \((\eps_0,a_0,z_0,0)\) and a unique
real-analytic function
\begin{equation}\label{eq:krrs-density-constraint}
  c=C(\eps,a,b,\vartheta),
  \qquad
  C(\eps,a,b,0)=0,
\end{equation}
that solves the \(H\)-density constraint.  After shrinking the
neighborhood, \(Q\) and \(C\) take values in \((0,1)\) whenever
\(\vartheta>0\) and \((\eps,a,b)\) is sufficiently close to the base point:
for \(Q\), this follows by continuity and
\(Q(\eps_0,a_0,z_0,0)=\eps_0\in(0,1)\), while for \(C\) it follows from
\(C(\eps,a,b,0)=0\) together with
\(\partial_\vartheta C(\eps,a,b,0)=1/\mathcal A(\eps,b)>0\), obtained by
differentiating the constraint in \(\vartheta\).

Substitute \eqref{eq:krrs-density-constraint} into the entropy and write
\begin{equation}\label{eq:krrs-constrained-entropy}
  \Sigma(\eps,a,b,\vartheta)
  :=\widehat{\mathcal S}
       \bigl(\eps,a,b,C(\eps,a,b,\vartheta)\bigr).
\end{equation}
Since \(\Sigma(\eps,a,b,0)=S_0(\eps)\), there exists a real-analytic function \(L\) such that
\begin{equation}\label{eq:krrs-entropy-quotient}
  \Sigma(\eps,a,b,\vartheta)
  =S_0(\eps)+\vartheta L(\eps,a,b,\vartheta),
  \qquad
  L(\eps,a,b,\vartheta)
  :=\int_0^1
     \partial_\vartheta\Sigma(\eps,a,b,t\vartheta)\,dt.
\end{equation}
For \(\vartheta\ne0\), \(L\) agrees with
\((\Sigma-S_0(\eps))/\vartheta\), while the integral formula defines its
real-analytic extension to \(\vartheta=0\).

Using \(Q(\eps,a,b,0)=\eps\) and
\(\partial_cQ(\eps,a,b,0)=2(\eps-b)\), since the term \(c^2S_0(a)\) makes no contribution to the first derivative, direct differentiation gives
\begin{equation}\label{eq:krrs-entropy-derivative}
  \begin{aligned}
  \left.\partial_c\widehat{\mathcal S}(\eps,a,b,c)\right|_{c=0}
  &=2S_0(b)-2S_0(\eps)+2S_0'(\eps)(\eps-b)\\
  &=2\bigl[S_0(b)-S_0(\eps)-S_0'(\eps)(b-\eps)\bigr]\\
  &=\mathcal N_\eps(b).
  \end{aligned}
\end{equation}

Differentiating the \(H\)-density constraint with respect to \(\vartheta\) at
\(\vartheta=0\), and using
\eqref{eq:krrs-density-derivative} and \eqref{eq:krrs-density-coefficient}, gives
\[
  \mathcal A(\eps,b)\,\partial_\vartheta C(\eps,a,b,0)=1,
  \qquad
  \partial_\vartheta C(\eps,a,b,0)=\frac{1}{\mathcal A(\eps,b)}.
\]
Moreover, \eqref{eq:krrs-entropy-quotient} gives
\(L(\eps,a,b,0)=\partial_\vartheta\Sigma(\eps,a,b,0)\).  Hence, the chain rule
and \eqref{eq:krrs-entropy-derivative} yield the boundary identity
\begin{equation}\label{eq:krrs-quotient-boundary}
  \begin{aligned}
  L(\eps,a,b,0)
  &=\left.\partial_c\widehat{\mathcal S}(\eps,a,b,c)\right|_{c=0}
    \partial_\vartheta C(\eps,a,b,0)\\
  &=\frac{\mathcal N_\eps(b)}{\mathcal A(\eps,b)}
   =\frac{\psi_d(\eps,b)}{n_d\eps^{m-d}}.
  \end{aligned}
\end{equation}
Thus, the coefficient \(L(\eps,a,b,0)\) of \(\vartheta\) in the Taylor
expansion of the entropy \(\Sigma(\eps,a,b,\vartheta)\) about \(\vartheta=0\) is
independent of \(a=q_{11}\).

\paragraph{3. Construction of a two-sided analytic solution of the
stationarity equations.}
We next retain exactly the second-order information needed to recover the
otherwise invisible parameter \(a=q_{11}\).

We use the following elementary analytic factorization fact.  Let \(f\) be
real-analytic near a point of the hyperplane \(\{x_j=0\}\) and suppose
\(\partial_{x_j}^{\,i}f=0\) on that hyperplane for \(0\le i<k\).  Then
\(f=x_j^{k}g\) with \(g\) real-analytic.  For \(k=1\), one may take
\[
  g(x)=\int_0^1
    (\partial_jf)(x_1,\dots,tx_j,\dots,x_n)\,dt.
\]
The general case follows by iteration.  We apply this factorization in the
\(c\)-variable with \(k=3\).  For each of \(\widehat{\mathcal T}\) and
\(\widehat{\mathcal S}\), the difference from its second-order Taylor
polynomial at \(c=0\) has vanishing \(c\)-derivatives of orders \(0,1,2\) on
\(\{c=0\}\).  Hence, these differences can be written as
\(c^3R_{\mathcal T}\) and \(c^3R_{\mathcal S}\), respectively, with
\(R_{\mathcal T}\) and \(R_{\mathcal S}\) real-analytic.  Writing the
second-order Taylor coefficients in the form justified below gives
real-analytic functions \(B_0,B_1,R_0,R_1\), depending on \((\eps,b)\), such
that
\begin{align}
  \widehat{\mathcal T}(\eps,a,b,c)
  &=\eps^m+\mathcal A(\eps,b)c
    +\bigl[B_0(\eps,b)+B_1(\eps,b)a\bigr]c^2
    +c^3R_{\mathcal T}(\eps,a,b,c),
    \label{eq:krrs-density-expansion}\\
  \widehat{\mathcal S}(\eps,a,b,c)
  &=S_0(\eps)+\mathcal N_\eps(b)c
    +\bigl[S_0(a)+R_0(\eps,b)+R_1(\eps,b)a\bigr]c^2
    +c^3R_{\mathcal S}(\eps,a,b,c).
    \label{eq:krrs-entropy-expansion}
\end{align}
We now verify the stated form of the second-order coefficients in
\eqref{eq:krrs-density-expansion} and \eqref{eq:krrs-entropy-expansion}.  By
\eqref{eq:krrs-large-block-expansion}, write
\[
  Q(\eps,a,b,c)=\eps+q_1c+q_2c^2+O(c^3),
  \qquad q_1=2(\eps-b),\quad q_2=3\eps-2b-a.
\]
Thus, \(q_1\) is independent of \(a\), whereas \(q_2\) is affine in \(a\).
In the homomorphism-density expansion, the contribution at order \(c^2\)
from assignments with no vertex in the first pode depends linearly on
\(q_2\), while all its remaining terms depend only on \(q_1\).  An assignment
with one vertex in the first pode already carries a factor \(c\), so its
coefficient at order \(c^2\) depends only on \(q_1\).  An assignment with
exactly two vertices in the first pode contains at most one factor \(a\),
because \(H\) is simple and those two vertices support at most one edge.
Finally, assignments with at least three vertices in the first pode are
divisible by \(c^3\).  Hence, the coefficient of \(c^2\) in
\(\widehat{\mathcal T}\) is affine in \(a\), which proves
\eqref{eq:krrs-density-expansion}.

For \(\widehat{\mathcal S}\), the term \(c^2S_0(a)\) is the only possible
nonlinear dependence on \(a\).  To see this, \(2c(1-c)S_0(b)\) is independent of
\(a\), and the coefficient of \(c^2\) in the remaining term is
\[
  [c^2]\,(1-c)^2S_0(Q)
  =S_0(\eps)-2S_0'(\eps)q_1+S_0'(\eps)q_2
    +\frac12S_0''(\eps)q_1^2.
\]
Since \(q_1\) is independent of \(a\) and \(q_2\) is affine in \(a\), this
coefficient is affine in \(a\).  This proves
\eqref{eq:krrs-entropy-expansion}.

Under the constraint
\(\widehat{\mathcal T}(\eps,a,b,c)=\eps^m+\vartheta\), solving
\eqref{eq:krrs-density-expansion} for
\(c=C(\eps,a,b,\vartheta)\) as a power series in \(\vartheta\) gives
\begin{equation}\label{eq:krrs-block-size-expansion}
  \begin{aligned}
    C(\eps,a,b,\vartheta)
    &=\frac{\vartheta}{\mathcal A(\eps,b)}
      -\frac{B_0(\eps,b)+B_1(\eps,b)a}
             {\mathcal A(\eps,b)^3}\,\vartheta^2
      +O_{H,\eps_0}(\vartheta^3).
  \end{aligned}
\end{equation}
Substituting \eqref{eq:krrs-block-size-expansion} into
\eqref{eq:krrs-entropy-expansion} and recalling
\eqref{eq:krrs-entropy-quotient} gives
\begin{equation}\label{eq:krrs-quotient-expansion}
  L(\eps,a,b,\vartheta)
  =\frac{\mathcal N_\eps(b)}{\mathcal A(\eps,b)}
    +\vartheta M(\eps,a,b)+O_{H,\eps_0}(\vartheta^2),
\end{equation}
where
\begin{equation}\label{eq:krrs-second-coefficient}
  M(\eps,a,b)
  =\frac{S_0(a)+R_0(\eps,b)+R_1(\eps,b)a}
         {\mathcal A(\eps,b)^2}
   -\frac{\mathcal N_\eps(b)
          [B_0(\eps,b)+B_1(\eps,b)a]}
         {\mathcal A(\eps,b)^3}.
\end{equation}
Consequently,
\begin{equation}\label{eq:krrs-coefficient-concavity}
  \partial_a^2M(\eps,a,b)
  =\frac{S_0''(a)}{\mathcal A(\eps,b)^2}<0.
\end{equation}

By \eqref{eq:krrs-quotient-boundary}, \(\partial_aL(\eps,a,b,0)=0\); hence, by the
analytic factorization fact above, there is a unique analytic function
\(\mathcal F_1\) satisfying
\begin{equation}\label{eq:krrs-small-block-stationarity}
  \partial_aL(\eps,a,b,\vartheta)
  =\vartheta\mathcal F_1(\eps,a,b,\vartheta).
\end{equation}
Also put
\begin{equation}\label{eq:krrs-cross-block-stationarity}
  \mathcal F_2(\eps,a,b,\vartheta)
  :=\partial_bL(\eps,a,b,\vartheta).
\end{equation}
For fixed \((\eps,\vartheta)\) with \(\vartheta>0\), every interior local
maximizer \((a,b)\) of
\[
  (a,b)\longmapsto\Sigma(\eps,a,b,\vartheta)
\]
satisfies \(\mathcal F_1=\mathcal F_2=0\).  To verify this,
\eqref{eq:krrs-entropy-quotient} gives
\(\partial_a\Sigma=\vartheta^2\mathcal F_1\) and
\(\partial_b\Sigma=\vartheta\mathcal F_2\).

We next show that \((\mathcal F_1,\mathcal F_2)\) vanishes at
\((\eps_0,a_0,z_0,0)\) and that its Jacobian with respect to \((a,b)\) is
invertible there.  These are the hypotheses needed to solve the system
\(\mathcal F_1=\mathcal F_2=0\) locally for \((a,b)\) as analytic functions
of \((\eps,\vartheta)\).  First,
\(\mathcal F_2(\eps_0,a_0,z_0,0)=0\) by
\eqref{eq:krrs-quotient-boundary} and the critical-point property of \(z_0\).  For
fixed \(\eps=\eps_0\) and sufficiently small \(\vartheta>0\), let
\(W_\vartheta\) be the entropy maximizer supplied by
\Cref{thm:krrs-bipodality} for \(\tau=\eps_0^m+\vartheta\).  Denote its bipodal
parameters by \((a,b,q,c)(\vartheta)\), after relabeling so that the first
pode is the one whose size \(c(\vartheta)\) tends to zero as
\(\vartheta\downarrow0\).  By \eqref{eq:krrs-parameter-limits}, they converge to
\((a_0,z_0,\eps_0,0)\), so they eventually lie in the neighborhood
parametrized above.  Since \(W_\vartheta\) has edge density \(\eps_0\),
solving the edge constraint for \(q\) gives
\(q=Q(\eps_0,a,b,c)\).  Its \(H\)-density is
\(\eps_0^m+\vartheta\); hence the local uniqueness assertion in
\eqref{eq:krrs-density-constraint} gives
\(c=C(\eps_0,a,b,\vartheta)\).  Therefore
\(s(W_\vartheta)=\Sigma(\eps_0,a,b,\vartheta)\).  For sufficiently small
\(\vartheta>0\), the three edge-probability parameters lie in \((0,1)\) by
\eqref{eq:krrs-parameter-limits}, while \(c\in(0,1)\) because a null pode would give
zero excess.  Hence, by continuity of \(C\) and \(Q\), every sufficiently
nearby pair \((\tilde a,\tilde b)\), with \(c\) and \(q\) supplied by \(C\)
and \(Q\), yields an admissible bipodal graphon with the same two constraints
and entropy \(\Sigma(\eps_0,\tilde a,\tilde b,\vartheta)\).  Since
\(W_\vartheta\) maximizes entropy among all graphons with these constraints,
\((a,b)(\vartheta)\) is an interior local maximizer of
\(\Sigma(\eps_0,\cdot,\cdot,\vartheta)\) and therefore satisfies
\(\mathcal F_1=\mathcal F_2=0\).  Passage to the limit
\(\vartheta\downarrow0\) gives
\begin{equation}\label{eq:krrs-boundary-stationarity}
  \mathcal F_1(\eps_0,a_0,z_0,0)
  =\mathcal F_2(\eps_0,a_0,z_0,0)=0.
\end{equation}
Equations \eqref{eq:krrs-small-block-stationarity} and \eqref{eq:krrs-quotient-expansion} give
\[
  \mathcal F_1(\eps,a,b,0)=\partial_aM(\eps,a,b),
\]
while \eqref{eq:krrs-cross-block-stationarity} and \eqref{eq:krrs-quotient-boundary} give
\[
  \mathcal F_2(\eps,a,b,0)
  =\frac{\partial_z\psi_d(\eps,b)}{n_d\eps^{m-d}}.
\]
Differentiating these identities at \((\eps_0,a_0,z_0)\), and using
\eqref{eq:krrs-coefficient-concavity} and \eqref{eq:krrs-scalar-nondegeneracy}, gives
\begin{align}
  \partial_a\mathcal F_1(\eps_0,a_0,z_0,0)
  &=\frac{S_0''(a_0)}{\mathcal A(\eps_0,z_0)^2}\ne0,
    \label{eq:krrs-jacobian-aa}\\
  \partial_a\mathcal F_2(\eps_0,a_0,z_0,0)
  &=0,
    \label{eq:krrs-jacobian-ba}\\
  \partial_b\mathcal F_2(\eps_0,a_0,z_0,0)
  &=\frac{\partial_z^2\psi_d(\eps_0,z_0)}
          {n_d\eps_0^{m-d}}\ne0.
    \label{eq:krrs-jacobian-bb}
\end{align}
The Jacobian of \((\mathcal F_1,\mathcal F_2)\) with respect to \((a,b)\)
is therefore triangular with nonzero diagonal.  The analytic implicit
function theorem now gives an open interval \(U_{\mathrm{IFT}}\ni\eps_0\),
a number \(\Delta_{\mathrm{IFT}}>0\), an open neighborhood
\(V_{\mathrm{IFT}}\) of \((a_0,z_0)\), and real-analytic functions
\begin{equation}\label{eq:krrs-stationary-densities}
  a_*(\eps,\vartheta),\qquad b_*(\eps,\vartheta)
  \quad
  ((\eps,\vartheta)\in
    U_{\mathrm{IFT}}\times(-\Delta_{\mathrm{IFT}},
                             \Delta_{\mathrm{IFT}}))
\end{equation}
taking values in \(V_{\mathrm{IFT}}\), with
\[
  (a_*(\eps_0,0),b_*(\eps_0,0))=(a_0,z_0).
\]
For each \((\eps,\vartheta)\) in this rectangle,
\[
  \mathcal F_i\bigl(\eps,a_*(\eps,\vartheta),
    b_*(\eps,\vartheta),\vartheta\bigr)=0
  \qquad(i=1,2),
\]
and \((a_*(\eps,\vartheta),b_*(\eps,\vartheta))\) is the unique solution in
\(V_{\mathrm{IFT}}\) of this system.  Define
\begin{align}
  c_*(\eps,\vartheta)
  &:=C\bigl(\eps,a_*(\eps,\vartheta),
                 b_*(\eps,\vartheta),\vartheta\bigr),\label{eq:krrs-block-size-map}\\
  q_*(\eps,\vartheta)
  &:=Q\bigl(\eps,a_*(\eps,\vartheta),
                 b_*(\eps,\vartheta),c_*(\eps,\vartheta)\bigr).
                 \label{eq:krrs-large-block-map}
\end{align}
After shrinking \(U_{\mathrm{IFT}}\) and \(\Delta_{\mathrm{IFT}}\) so that
these compositions are defined, all four functions are real-analytic on the
two-sided rectangle
\(U_{\mathrm{IFT}}\times(-\Delta_{\mathrm{IFT}},\Delta_{\mathrm{IFT}})\).
At \(\vartheta=0\), equations \eqref{eq:krrs-density-constraint} and
\eqref{eq:krrs-edge-constraint} give
\[
  c_*(\eps,0)=0,
  \qquad
  q_*(\eps,0)=\eps.
\]
Moreover, \(\mathcal F_2(\eps,a_*(\eps,0),b_*(\eps,0),0)=0\) and
\eqref{eq:krrs-quotient-boundary} show that
\(b_*(\eps,0)\) is a critical point of
\(z\mapsto\psi_d(\eps,z)\).  The uniqueness assertion in
\Cref{thm:krrs-cross-density} shows that this critical point is
unique, so, after shrinking \(U_{\mathrm{IFT}}\),
\begin{equation}\label{eq:krrs-cross-block-limit}
  b_*(\eps,0)=\zeta_d(\eps)
  \qquad(\eps\in U_{\mathrm{IFT}}).
\end{equation}
In particular, this also proves directly that \(\zeta_d\) is real-analytic
locally away from \(\eps_*\).

\paragraph{4. Identification with the entropy-maximizer parameters.}
Let \(\Omega^+\) denote the domain on which
\Cref{thm:krrs-bipodality} supplies jointly real-analytic bipodal optimizer
parameters, expressed in the coordinates \((\eps,\vartheta)\) with
\(\tau=\eps^m+\vartheta\).  For each nonexceptional \(\eps\), this domain is
specified by
\(\eps^m<\tau<\tau_0(\eps)\) for a threshold \(\tau_0(\eps)>\eps^m\)
depending on \(\eps\), so the fiber of
\(\Omega^+\) over each nonexceptional \(\eps\) is an interval
\begin{equation}\label{eq:krrs-domain-fibres}
  \{\vartheta:(\eps,\vartheta)\in\Omega^+\}
  =(0,h(\eps)),
  \qquad
  h(\eps):=\tau_0(\eps)-\eps^m>0.
\end{equation}
Since the theorem states that the parameter map is
real-analytic on \(\Omega^+\), this domain is open (real-analytic maps are,
by definition, defined on open sets).

Choose \(\vartheta^\sharp\in(0,h(\eps_0))\).  Openness at
\((\eps_0,\vartheta^\sharp)\) gives an open interval
\(U_{\mathrm{KRR}}\ni\eps_0\) such that
\((\eps,\vartheta^\sharp)\in\Omega^+\) for every
\(\eps\in U_{\mathrm{KRR}}\).  The interval property
\eqref{eq:krrs-domain-fibres} then implies
\begin{equation}\label{eq:krrs-uniform-domain}
  U_{\mathrm{KRR}}\times(0,\vartheta^\sharp)
  \subseteq\Omega^+.
\end{equation}
Thus, the same interval \(0<\vartheta<\vartheta^\sharp\) lies in the domain of
\Cref{thm:krrs-bipodality} for every \(\eps\in U_{\mathrm{KRR}}\).  This
conclusion uses only the theorem statement and not any localization argument
from its proof.

Shrink to an open interval \(U\ni\eps_0\) whose closure is contained in
\[
  U_{\mathrm{IFT}}\cap U_{\mathrm{KRR}}
  \cap\bigl((0,1)\setminus\{\eps_*\}\bigr),
\]
and choose
\[
  0<\Delta<\min\{\Delta_{\mathrm{IFT}},\vartheta^\sharp\}.
\]
By \Cref{thm:krrs-bipodality}, throughout
\(U\times(0,\Delta)\) the entropy maximizer is unique up to relabeling and
bipodal.  For each \((\eps,\vartheta)\) in this set, let
\(W^+_{\eps,\vartheta}\) be a representative of this maximizer, with its
podes labeled so that the first pode has size \(c^+\) and
\(c^+\to0\) as \(\vartheta\downarrow0\).  The corresponding parameter map
\begin{equation}\label{eq:krrs-optimizer-map}
  (a^+,b^+,q^+,c^+)(\eps,\vartheta)
\end{equation}
is real-analytic.

It remains to identify the optimizer parameter map
\eqref{eq:krrs-optimizer-map} with the two-sided analytic solution of the
stationarity equations constructed in
\eqref{eq:krrs-stationary-densities}--\eqref{eq:krrs-large-block-map}.  At \(\eps=\eps_0\), the
one-sided limits in \Cref{thm:krrs-bipodality} and
\eqref{eq:krrs-parameter-limits} show that the parameters in
\eqref{eq:krrs-optimizer-map} lie in the local solution neighborhoods
for all sufficiently small \(\vartheta>0\).  Choose one such
\(\vartheta_1\in(0,\Delta)\).  The three edge-probability parameters have
limits \(a_0,z_0,\eps_0\in(0,1)\), and \(c^+(\eps_0,\vartheta_1)\in(0,1)\)
because a bipodal graphon with a null pode is constant and has excess
\(0\); after taking \(\vartheta_1\) sufficiently small, all four parameters
therefore lie in \((0,1)\).  By continuity, the same is true on a nonempty
open neighborhood \(\mathcal O\subset U\times(0,\Delta)\) of
\((\eps_0,\vartheta_1)\), and all parameters there remain in those
neighborhoods.

For every \((\eps,\vartheta)\in\mathcal O\), we repeat the constraint and
local-maximality argument used to prove \eqref{eq:krrs-boundary-stationarity}.  The edge
constraint gives \(q^+=Q(\eps,a^+,b^+,c^+)\), and the local uniqueness in
\eqref{eq:krrs-density-constraint} gives
\(c^+=C(\eps,a^+,b^+,\vartheta)\).  Since
\(W^+_{\eps,\vartheta}\) maximizes entropy, it is locally maximal among the
nearby bipodal competitors supplied by \(C\) and \(Q\).  Hence
\(\mathcal F_1=\mathcal F_2=0\) at
\((\eps,a^+,b^+,\vartheta)\), and the local uniqueness obtained above gives
\begin{equation}\label{eq:krrs-map-agreement}
  (a^+,b^+,q^+,c^+)
  =(a_*,b_*,q_*,c_*)
  \qquad\text{on }\mathcal O.
\end{equation}
Both sides of \eqref{eq:krrs-map-agreement} are real-analytic on
\(U\times(0,\Delta)\), which is a product of intervals and hence connected.
The one-sided limits in \Cref{thm:krrs-bipodality} hold as
\(\vartheta\downarrow0\) for each fixed \(\eps\), but are not asserted to be
uniform in \(\eps\).  Thus, the argument based on these limits initially
identifies the two parameter maps only on \(\mathcal O\).  The identity
theorem for real-analytic functions extends this equality to all of the
connected set \(U\times(0,\Delta)\).  Consequently,
\((a_*,b_*,q_*,c_*)\) is the desired two-sided real-analytic extension of the
optimizer parameter map \eqref{eq:krrs-optimizer-map}.

Since \(\overline U\) avoids \(\eps_*\), the fixed-point assertion in
\Cref{thm:krrs-cross-density} gives
\(\zeta_d(\eps)\ne\eps\), and hence
\(\mathcal D_\eps(\zeta_d(\eps))>0\), for every \(\eps\in\overline U\).
Finally, since \(C(\eps,a,b,0)\equiv0\), the partial derivatives
\(\partial_aC\) and \(\partial_bC\) vanish at \(\vartheta=0\).
Differentiating \eqref{eq:krrs-block-size-map} with respect to \(\vartheta\) at
\(\vartheta=0\) therefore yields
\[
  \partial_\vartheta c_*(\eps,0)
  =\partial_\vartheta C\bigl(
     \eps,a_*(\eps,0),b_*(\eps,0),0\bigr).
\]
Equations \eqref{eq:krrs-block-size-expansion} and \eqref{eq:krrs-cross-block-limit} now give
\begin{equation}\label{eq:krrs-block-size-derivative}
  \partial_\vartheta c_*(\eps,0)
  =\frac{1}{n_d\eps^{m-d}
       \mathcal D_\eps(\zeta_d(\eps))}>0.
\end{equation}
Moreover, \(c_*\) is analytic on an open rectangle containing the compact
set \(\overline U\times[-\Delta,\Delta]\) and vanishes on
\(\{\vartheta=0\}\), so \(\partial_\vartheta c_*\) is bounded there, and
the mean value theorem gives, uniformly for \(\eps\in U\),
\[
  c_*(\eps,\vartheta)=O_{H,U}(\vartheta)
  \qquad(\vartheta\to0).
\]
Finally, \(c_*(\eps,\vartheta)\) takes values in \((0,1)\) when
\(\eps\in U\) and \(0<\vartheta<\Delta\): this follows, after shrinking
\(\Delta\) once more, from \eqref{eq:krrs-block-size-derivative} and
\(c_*(\eps,0)=0\).

Together with \eqref{eq:krrs-cross-block-limit} and the boundary identities for
\(q_*\) and \(c_*\), this proves every assertion of
\Cref{thm:krrs-analytic-extension}.
\end{proof}

%% file: sections/appendix_lz_boundary.tex
\section{Scalar geometry of the Lubetzky--Zhao boundary}
\label{app:lz-scalar-geometry}

This appendix gives the proofs of the three supporting lemmas used in
\Cref{sec:lz-boundary}, together with the detailed verifications of (M1), (M3),
and (M4) in \Cref{thm:scalar-lz-boundary}.  They are collected here to
separate the scalar convex-analysis details from the construction of the
boundary functions $\pc$ and $\sm$ in the main text.

We shall use the following standard two-point representation of a convex
minorant.

\begin{lemma}[Two-point representation of convex minorants]
\label{lem:two-point-convex-minorant}
Let $I\subset\mathbb R$ be a compact interval, let $f:I\to\mathbb R$ be
continuous, and let $\widehat f$ be the convex minorant of $f$.  Then, for
every $x\in I$,
\[
  \widehat f(x)
  =
  \min\left\{
    \Big(\lambda f(x_1)+(1-\lambda)f(x_2)\Big):
      x_1,x_2\in I,\  \lambda\in[0,1],\
      \lambda x_1+(1-\lambda)x_2=x
  \right\}.
\]
\end{lemma}
\begin{proof}
The general representation theorem expresses the convex minorant as an
infimum over convex combinations of values of $f$.  In one dimension, two
points suffice, and compactness of $I$ and continuity of $f$ ensure that the
infimum is attained.  See, for example,
\cite[Corollary~17.1.5]{RockafellarConvexAnalysis}.
\end{proof}

\begin{proof}[Proof of \Cref{lem:convexity-defect}]
Differentiating the definition of $h_{p,d}$ and using
\[
  J_p'(z)=\log\frac{z(1-p)}{(1-z)p},
  \qquad
  J_p''(z)=\frac1{z(1-z)},
  \qquad
  J_p'''(z)=\frac{2z-1}{z^2(1-z)^2}
\]
gives
\[
  h_{p,d}'(z)=(2-d)J_p''(z)+zJ_p'''(z)
  =\frac{d(z-r_*)}{z(1-z)^2}.
\]
The denominator is positive.  Hence, $h_{p,d}'$ is negative before $r_*$ and
positive after $r_*$, which proves that $h_{p,d}$ is strictly decreasing on
$(0,r_*)$, strictly increasing on $(r_*,1)$, and has its unique minimum at
$r_*$.

We next evaluate this minimum.  Since $r_*=(d-1)/d$,
\[
  r_*J_p''(r_*)=d,
  \qquad
  J_p'(r_*)=\log\frac{(d-1)(1-p)}{p},
\]
and therefore
\[
  h_{p,d}(r_*)
  =d-(d-1)\log\frac{(d-1)(1-p)}{p}.
\]
As a function of $p$, this value is strictly increasing because its derivative
is
\[
  (d-1)\left(\frac1{1-p}+\frac1p\right)>0.
\]
Solving $h_{p,d}(r_*)=0$ gives
\[
  p=\frac{d-1}{(d-1)+\exp(d/(d-1))}=p_*.
\]
It follows that $h_{p,d}(r_*)\ge0$ exactly when $p\ge p_*$.  Since $r_*$ is
the minimum point, $h_{p,d}\ge0$ throughout $(0,1)$ in this case.

Now suppose that $p<p_*$.  Then, $h_{p,d}(r_*)<0$.  The alternative expression
\[
  h_{p,d}(z)=\frac1{1-z}-(d-1)J_p'(z)
\]
shows that $h_{p,d}(z)\to+\infty$ as $z\downarrow0$ and as $z\uparrow1$.
At the left endpoint $J_p'(z)\to-\infty$, while at the right endpoint
$1/(1-z)$ dominates the logarithmic divergence of $J_p'(z)$.  The intermediate
value theorem and the strict monotonicity on either side of $r_*$ therefore
give unique zeros $u_-(p)<r_*<u_+(p)$.  The same monotonicity gives the stated
sign of $h_{p,d}$ on the three intervening intervals.

It remains to translate this sign information into curvature of
$\varphi_{p,d}$.  For $z=x^{1/d}$, the chain rule gives
\[
  \varphi_{p,d}'(x)=\frac{J_p'(z)}{d z^{d-1}},
  \qquad
  \varphi_{p,d}''(x)
  =\frac{zJ_p''(z)-(d-1)J_p'(z)}{d^2z^{2d-1}}
  =\frac{h_{p,d}(z)}{d^2z^{2d-1}}.
\]
The denominator is positive, so $\varphi_{p,d}''(x)$ has the same sign as
$h_{p,d}(z)$.  When $p\ge p_*$, this proves convexity on $(0,1)$, and
continuity extends it to $[0,1]$.  When $p<p_*$, the change of variables
$x=z^d$ converts the three sign intervals for $h_{p,d}$ into precisely the
convexity and concavity intervals stated in the lemma.
\end{proof}

\begin{proof}[Proof of \Cref{lem:contact-points}]
Write $\varphi=\varphi_{p,d}$ and $u_\pm=u_\pm(p)$, put
\[
  \alpha=u_-^d,
  \qquad
  \beta=u_+^d,
\]
and let $\widehat\varphi$ be the convex minorant of $\varphi$.  By
Lemma~\ref{lem:convexity-defect}, $\varphi$ is strictly convex on
$(0,\alpha)$ and $(\beta,1)$ and strictly concave on $(\alpha,\beta)$.
The proof has three parts.  First we use the convex minorant to produce the
two contact points and identify the shape of the minorant.  Then, we prove that
no other pair of contact points can satisfy the same common-tangent equations.
Finally we record the analytic dependence on $p$ and the monotonicity of the
two contacts.

We begin with one elementary tangency observation: if an affine
function $L$ lies below a differentiable function $\varphi$ in a neighborhood
of an interior point $t$, and if $L(t)=\varphi(t)$, then the slope of $L$ is
$\varphi'(t)$.

We first extract the two contacts from the convex minorant.  The convex
minorant cannot touch $\varphi$ at any point $c\in(\alpha,\beta)$.  To see
this, suppose $\widehat\varphi(c)=\varphi(c)$ and take a supporting line to the
convex function $\widehat\varphi$ at $c$.  This line lies below
$\widehat\varphi$, hence below
$\varphi$, and has equality with $\varphi$ at $c$.  By the tangency
observation, this line has slope $\varphi'(c)$.  But $\varphi$ is strictly
concave near $c$, and the tangent line to a strictly concave function lies
strictly above the graph on both sides of the tangency point.  This is a
contradiction.  Thus,
\[
  (\alpha,\beta)\subset\{x : \widehat\varphi(x)<\varphi(x)\}.
\]

Let $(a,b)$ be the connected component of
$\{x : \widehat\varphi(x)<\varphi(x)\}$ containing $(\alpha,\beta)$.  Fix
$x\in(a,b)$.  By Lemma~\ref{lem:two-point-convex-minorant}, there are
$s,t\in[0,1]$ and $\lambda\in[0,1]$ such that
\[
  x=\lambda s+(1-\lambda)t,
  \qquad
  \widehat\varphi(x)=\lambda\varphi(s)+(1-\lambda)\varphi(t).
\]
After interchanging $s$ and $t$ and replacing $\lambda$ by $1-\lambda$ if
necessary, we must have
$s<x<t$ and $0<\lambda<1$; otherwise the value in this representation would
be $\varphi(x)$.  Convexity of $\widehat\varphi$ and the inequality
$\widehat\varphi\le\varphi$ give
\[
  \widehat\varphi(x)
  \le \lambda\widehat\varphi(s)+(1-\lambda)\widehat\varphi(t)
  \le \lambda\varphi(s)+(1-\lambda)\varphi(t)
  =\widehat\varphi(x).
\]
Thus, both inequalities are equalities.  Since both weights are positive,
$\widehat\varphi(s)=\varphi(s)$ and
$\widehat\varphi(t)=\varphi(t)$.  For a convex function, equality at a
nontrivial convex combination forces equality with the chord on the whole
interval; hence $\widehat\varphi$ is affine on $[s,t]$.
Since $(a,b)$ contains no contact points, we must have $s\le a$ and $t\ge b$.
Hence
$\widehat\varphi=L$ on $(a,b)$ for some affine function $L$.

The endpoints of this component are contact points.  At an interior
endpoint this follows from the continuity of
$\varphi-\widehat\varphi$.  At the endpoints of the domain,
Lemma~\ref{lem:two-point-convex-minorant} gives
\[
  \widehat\varphi(0)=\varphi(0),
  \qquad
  \widehat\varphi(1)=\varphi(1),
\]
because the only convex combination of points of $[0,1]$ with barycenter
$0$ (respectively, $1$) is supported at $0$ (respectively, $1$).
Consequently, the affine function $L$ extends continuously to $[a,b]$, and
\[
  L(a)=\varphi(a),\qquad L(b)=\varphi(b),
\]
where the value of $L$ at a domain endpoint is its corresponding one-sided
limit from $(a,b)$.  We next rule out these endpoint cases.

First suppose $a=0$.  Since $L$ is affine, it has some finite slope $m$, and
the contact condition at the endpoint gives $L(0)=\varphi(0)$.  The derivative
of $\varphi$ satisfies $\varphi'(x)\to-\infty$ as $x\downarrow0$.  By the mean
value theorem, for all sufficiently small $x>0$,
\[
  \frac{\varphi(x)-\varphi(0)}{x}<m.
\]
Using $L(x)=L(0)+mx=\varphi(0)+mx$, this inequality becomes
\[
  \varphi(x)<\varphi(0)+mx=L(x),
\]
contradicting $L=\widehat\varphi\le\varphi$ on $(a,b)$.

Since $\varphi'(x)\to+\infty$ as $x\uparrow1$, the case $b=1$ is ruled out by
the same argument.  Therefore $0<a\le\alpha<\beta\le b<1$.
Since the contact points $a$ and $b$ are now interior points, the tangency
observation applies to $L$ at both endpoints.  Thus, the slope of $L$ is
$\varphi'(a)$ and also $\varphi'(b)$.

The inequalities $a\le\alpha$ and $\beta\le b$ are in fact strict.  Suppose
$a=\alpha$.  Then, the slope of $L$ is $\varphi'(\alpha)$.  Since $\varphi'$ is
strictly decreasing immediately to the right of $\alpha$, for small
$x>\alpha$ we have
\[
  \varphi(x)-L(x)
  =
  \int_\alpha^x(\varphi'(t)-\varphi'(\alpha))\,dt
  <0,
\]
again contradicting $L\le\varphi$.  The same argument at $\beta$ shows
$b>\beta$: if $b=\beta$, the tangent line at $\beta$ lies above $\varphi$ just
to the left of $\beta$.
Thus,
\[
  0<a<\alpha<\beta<b<1.
\]

The affine function $L$ is therefore the line through
$(a,\varphi(a))$ and $(b,\varphi(b))$, and the tangency observation gives
\[
  \varphi'(a)=\varphi'(b),
  \qquad
  \varphi(b)-\varphi(a)=\varphi'(a)(b-a),
\]
which are \eqref{eq:contact-equal-slopes}--\eqref{eq:contact-chord-identity}.

There are no other connected components of
\[
  \{x:\widehat\varphi(x)<\varphi(x)\}.
\]
To see this, any such component would be contained entirely in $(0,a)$ or in
$(b,1)$, hence in one of the two strictly convex outer regions.  Repeating
the two-point representation argument above on such a component would show
that $\widehat\varphi$ is affine there and joins its two endpoint contact
points.  But the chord between two points of a strictly convex graph lies
strictly above the graph between them, contradicting
$\widehat\varphi\le\varphi$.  Hence
$\widehat\varphi$ agrees with $\varphi$ outside $[a,b]$ and is the line segment
$L$ on $[a,b]$.  This proves the existence of a pair satisfying
\eqref{eq:contact-equal-slopes}--\eqref{eq:contact-chord-identity} and the asserted shape of the
convex minorant.

We now prove uniqueness of the contact pair.  More precisely, we must
show that there is at most one pair
\[
  a\in(0,\alpha),\qquad b\in(\beta,1)
\]
satisfying the two equations
\eqref{eq:contact-equal-slopes}--\eqref{eq:contact-chord-identity}.  Since the convex-minorant
construction above has already produced one such pair, this will prove
uniqueness of the contact pair in the lemma.

We first observe that on the right convex region $(\beta,1)$, the derivative
$\varphi'$ is strictly increasing.  Therefore, for a fixed
$a\in(0,\alpha)$, there is at most one $b\in(\beta,1)$ satisfying
\[
  \varphi'(b)=\varphi'(a).
\]
Let $G$ be the restriction of $\varphi'$ to the interval $(\beta,1)$.  Then
$G$ is a $C^1$ strictly increasing function with
\[
  G'(b)=\varphi''(b)>0\qquad(\beta<b<1).
\]
Hence, by the inverse function theorem, $G$ has a $C^1$ inverse on its range.
On the set of $a\in(0,\alpha)$ for which
$\varphi'(a)$ belongs to this range, define
\[
  b(a):=G^{-1}(\varphi'(a)).
\]
This is exactly the unique right contact with the same slope as the left
contact $a$ and is $C^1$ as it is a composition of $C^1$ functions.
The set on which $b(a)$ is defined is an interval, because
$\varphi'$ is strictly increasing on the left convex region as well.

By this uniqueness, every possible common-tangent pair must lie on the curve
$a\mapsto(a,b(a))$.
Along this curve, the remaining condition is that the
tangent line at $a$ actually pass through the point $(b(a),\varphi(b(a)))$.
Equivalently, define
\[
  D(a):=\varphi(b(a))-\varphi(a)-\varphi'(a)(b(a)-a).
\]
A zero of $D$ is exactly a pair satisfying the chord equation
\eqref{eq:contact-chord-identity}.  We now show that $D$ is strictly decreasing wherever
it is defined.  Differentiating both sides of the identity
$\varphi'(b(a))=\varphi'(a)$ with respect to $a$ gives
\[
  \varphi''(b(a))\,b'(a)=\varphi''(a).
\]
Differentiating $D$ therefore gives
\begin{align*}
  D'(a)
  &=
  \varphi'(b(a))b'(a)-\varphi'(a)
  -\varphi''(a)(b(a)-a)
  -\varphi'(a)(b'(a)-1)\\
  &=
  \bigl(\varphi'(b(a))-\varphi'(a)\bigr)b'(a)
  -\varphi''(a)(b(a)-a)\\
  &=
  -\varphi''(a)(b(a)-a)<0.
\end{align*}
In the last line we used the equal-slope relation
$\varphi'(b(a))=\varphi'(a)$, the inequality $b(a)>a$, and the strict
convexity on the left region, which gives $\varphi''(a)>0$.  Therefore $D$ can
vanish at most once.  The convex-minorant construction above gives one zero,
so the pair satisfying \eqref{eq:contact-equal-slopes}--\eqref{eq:contact-chord-identity} is
unique.

Finally we show that the contact points depend analytically on $p$, with
$dx_a/dp>0$ and $dx_b/dp<0$.
Define
\[
  F_1(a,b,p)=\varphi_{p,d}'(b)-\varphi_{p,d}'(a),
\]
\[
  F_2(a,b,p)=
  \varphi_{p,d}(b)-\varphi_{p,d}(a)-\varphi_{p,d}'(a)(b-a).
\]
Set $\mathbf F=(F_1,F_2)$.  This map is jointly analytic in $(a,b,p)$ on
$(0,1)^3$.  At a triple $(a,b,p)$ satisfying $F_1=F_2=0$, the derivatives
with respect to $(a,b)$ are
\[
  D_{(a,b)}\mathbf F=
  \begin{pmatrix}
    -\varphi''(a)&\varphi''(b)\\
    -\varphi''(a)(b-a)&0
  \end{pmatrix}.
\]
To obtain this matrix, note that the $b$-derivative of $F_2$ is
$\varphi'(b)-\varphi'(a)$, which
vanishes because $F_1=0$.  The determinant is positive:
\[
  \varphi''(a)\varphi''(b)(b-a)>0.
\]
By the analytic implicit function theorem, the contact points are locally
analytic functions of $p$.  Since the pair satisfying
\eqref{eq:contact-equal-slopes}--\eqref{eq:contact-chord-identity} is unique for each
$p\in(0,p_*)$, these local analytic solutions agree on overlaps.  Hence, they
define analytic functions $x_a(p)$ and $x_b(p)$ on $(0,p_*)$ such that
$F_1(x_a(p),x_b(p),p)=0$ and
$F_2(x_a(p),x_b(p),p)=0$.
By the chain rule, we have
\begin{align}
  D_{(a,b)}\mathbf F
  \binom{da/dp}{db/dp}
  +
  \binom{\partial_pF_1}{\partial_pF_2}
  =0. \label{eq:contact-derivative-system}
\end{align}

If $z=x^{1/d}$, then
\[
  \partial_p\varphi_{p,d}(x)=\frac{p-z}{p(1-p)},
  \qquad
  \partial_p\varphi_{p,d}'(x)=
  -\frac1{p(1-p)d\,z^{d-1}}.
\]
Writing $u_a=a^{1/d}$ and $u_b=b^{1/d}$, this gives
\[
  \partial_pF_1(a,b,p)
  =
  \frac1{p(1-p)d}\left(u_a^{1-d}-u_b^{1-d}\right)>0,
\]
and
\[
  \partial_pF_2(a,b,p)
  =
  \frac1{p(1-p)}
  \left(\frac{u_b^d-u_a^d}{d u_a^{d-1}}-(u_b-u_a)\right)>0.
\]
The last inequality comes from the strict convexity of $z\mapsto z^d$.
As the second row of \eqref{eq:contact-derivative-system} gives
\[
  -\varphi''(a)(b-a)\frac{da}{dp}+\partial_pF_2=0,
\]
by combining $\varphi''(a)>0$ and $b>a$, we obtain $da/dp>0$.

The same linear system also gives the sign of $db/dp$.  The first row gives
\[
  \varphi''(b)\frac{db}{dp}
  =
  \varphi''(a)\frac{da}{dp}-\partial_pF_1.
\]
Substituting the formula for $da/dp$ from the second row, this becomes
\[
  \varphi''(b)\frac{db}{dp}
  =
  \frac{\partial_pF_2}{b-a}-\partial_pF_1.
\]
Set $\Delta=b-a=u_b^d-u_a^d$.  Then
\[
  \frac{\partial_pF_2}{b-a}
  =
  \frac{\partial_pF_2}{\Delta}
  =
  \frac1{p(1-p)}
  \left(
    \frac1{d u_a^{d-1}}
    -\frac{u_b-u_a}{u_b^d-u_a^d}
  \right),
\]
whereas
\[
  \partial_pF_1
  =
  \frac1{p(1-p)}
  \left(
    \frac1{d u_a^{d-1}}
    -\frac1{d u_b^{d-1}}
  \right).
\]
Subtracting the first equation above from the second, the
$1/(d u_a^{d-1})$ terms cancel and we get
\[
  \partial_pF_1-\frac{\partial_pF_2}{b-a}
  =
  \frac1{p(1-p)}
  \left(
    \frac{u_b-u_a}{u_b^d-u_a^d}
    -\frac1{d u_b^{d-1}}
  \right).
\]
The secant slope of the strictly convex function $z\mapsto z^d$ from $u_a$ to
$u_b$ is strictly smaller than its derivative at the right endpoint $u_b$:
\[
  \frac{u_b^d-u_a^d}{u_b-u_a}<d u_b^{d-1}.
\]
Hence, $\partial_pF_2/(b-a)-\partial_pF_1<0$, and since $\varphi''(b)>0$, we get
$db/dp<0$.
\end{proof}

\begin{proof}[Proof of \Cref{lem:contact-point-limits}]
We first verify that the zeros $u_-(p)$ and $u_+(p)$ of $h_{p,d}$ collapse to
$r_*$ as $p\uparrow p_*$.  At $p=p_*$, Lemma~\ref{lem:convexity-defect} gives
$h_{p_*,d}(r_*)=0$, while $h_{p_*,d}$ is strictly decreasing on $(0,r_*)$ and
strictly increasing on $(r_*,1)$.  Hence
\[
  h_{p_*,d}(z)>0\qquad (z\ne r_*).
\]
Fix $\varepsilon>0$ small enough that $r_*-\varepsilon$ and
$r_*+\varepsilon$ lie in $(0,1)$.  By
continuity in $p$,
\[
  h_{p,d}(r_*-\varepsilon)>0,
  \qquad
  h_{p,d}(r_*+\varepsilon)>0
\]
for all $p<p_*$ sufficiently close to $p_*$.  Since $h_{p,d}(r_*)<0$ for such
$p$, the left zero lies in $(r_*-\varepsilon,r_*)$ and the right zero lies in
$(r_*,r_*+\varepsilon)$.  As $\varepsilon>0$ is arbitrary, we conclude that
$u_-(p)\to r_*$ and $u_+(p)\to r_*$ as $p\uparrow p_*$.

For the contact limit as $p\uparrow p_*$, by \Cref{lem:contact-points},
monotonicity gives $x_a(p)\to A$ and $x_b(p)\to B$ for some $A,B\in[0,1]$.
Since $x_a(p)<u_-(p)^d$ and $x_b(p)>u_+(p)^d$, while
$u_\pm(p)\to r_*$, we have $A\le r_*^d\le B$.
After fixing any $p_0\in(0,p_*)$, monotonicity also gives
\[
  0<x_a(p_0)\le x_a(p),
  \qquad
  x_b(p)\le x_b(p_0)<1
  \qquad (p_0<p<p_*),
\]
so $A>0$ and $B<1$.  The map $(p,x)\mapsto\varphi_{p,d}'(x)$ is continuous
at $(p_*,A)$ and $(p_*,B)$.  We may therefore pass to the limit in the
equal-slope equation~\eqref{eq:contact-equal-slopes} to obtain
\[
  \varphi_{p_*,d}'(A)=\varphi_{p_*,d}'(B).
\]
At $p=p_*$, $\varphi_{p_*,d}''\ge0$ and vanishes only at $r_*^d$, so
$\varphi_{p_*,d}'$ is strictly increasing.  Hence
$A=B=r_*^d$.
Taking positive $d$th roots gives
$u_a(p)\to r_*$ and $u_b(p)\to r_*$ as $p\uparrow p_*$.

It remains to consider $p\downarrow0$.  For every fixed $z\in(0,1)$,
\[
  h_{p,d}(z)
  =
  \frac1{1-z}-(d-1)\log\frac{z(1-p)}{(1-z)p}
  \longrightarrow -\infty.
\]
If $0<\eps<r_*$, then $h_{p,d}(\eps)<0$ for all sufficiently small $p$.  Since
$h_{p,d}$ is strictly decreasing on $(0,r_*)$, the left zero satisfies
$u_-(p)<\eps$.  Hence, $u_-(p)\to0$.  Similarly, if
$0<\eps<1-r_*$, then $h_{p,d}(1-\eps)<0$ for all sufficiently small $p$; since
$h_{p,d}$ is strictly increasing on $(r_*,1)$, the right zero satisfies
$u_+(p)>1-\eps$.  Hence, $u_+(p)\to1$.

The contact ordering
\[
  0<u_a(p)<u_-(p)<u_+(p)<u_b(p)<1
\]
then gives $u_a(p)\to0$ and $u_b(p)\to1$ as $p\downarrow0$.  Finally, the
monotonicity from Lemma~\ref{lem:contact-points} turns these limits into
$u_a(p)\downarrow0$ and $u_b(p)\uparrow1$.
\end{proof}

\subsection{Remaining details for the scalar Lubetzky--Zhao boundary theorem}
\label{app:lz-boundary-details}

\begin{proof}[Proof of \textup{(M1)}]
We first check that $\pc(r)<r$.  We shall use that
\[
  p_*=\frac{d-1}{(d-1)+\exp(d/(d-1))}
  <
  \frac{d-1}{d}
  =r_*.
\]
For $r\in(r_*,1)$, the inequality $\pc(r)<r$ is immediate, since
$\pc(r)<p_*<r_*<r$.  Now fix $r\in(0,r_*)$ and set
$p=\pc(r)$.  Then, $r=u_a(p)$ and $p<p_*<r_*<u_b(p)$.  If $u_a(p)\le p$, then
$u_a(p)^d\le p^d<u_b(p)^d$, so $p^d$ lies on the affine part of the convex
minorant joining the two contact points.  But $J_p(p)=0$ is the strict global
minimum of $J_p$.  Thus, unless one of the contact points has first coordinate
$p^d$, the affine part has positive value at $p^d$, contradicting the fact that
the convex minorant lies below $\varphi_{p,d}(p^d)=J_p(p)=0$.  If one contact
point had first coordinate $p^d$, then the common tangent slope would be
$\varphi_{p,d}'(p^d)=0$; the equal-slope condition would force the other contact
point also to have first coordinate $p^d$, since $\varphi_{p,d}'(x)=0$ only at
$x=p^d$.  This contradicts distinctness of the contacts.  Hence
$p<u_a(p)=r$.

At $p=\pc(r)$, the two contact points have first coordinates $r^d$ and
$\sm(r)^d$.  They are distinct by Lemma~\ref{lem:contact-points}, so
$\sm(r)\ne r$.  This proves (M1).
\end{proof}

\begin{proof}[Proof of \textup{(M3)} and \textup{(M4)}]
Fix $r\in(0,1)\setminus\{r_*\}$.  The affine part of the
convex minorant of $\varphi_{\pc(r),d}$ joins
\[
  \bigl(r^d,J_{\pc(r)}(r)\bigr)
  \quad\text{and}\quad
  \bigl(\sm(r)^d,J_{\pc(r)}(\sm(r))\bigr),
\]
and is tangent at both endpoints.  Hence, this affine line is precisely
$\ell_r$, the tangent line at $r^d$.

On the interval between the two contact points, $\ell_r$ is the affine part of
the convex minorant, so it lies below $\varphi_{\pc(r),d}$.  Outside that
interval, it is the tangent line at the adjacent contact point in a strictly
convex outer region, so it again lies below the graph, with equality only at
that contact.  On the open interval between the contact points, the graph is
strictly above the affine part.  Therefore $\ell_r\le\varphi_{\pc(r),d}$ on
$[0,1]$, and the zeros of
\[
  x\mapsto \varphi_{\pc(r),d}(x)-\ell_r(x)
\]
are exactly $r^d$ and $\sm(r)^d$.  This proves (M3) and (M4).
\end{proof}

%% file: sections/appendix_nonexceptional_endpoint.tex
\section{Technical proofs for nonexceptional endpoints}
\label{app:nonexceptional-proofs}

\subsection{Proofs supporting the one-dimensional reduction}
\label{app:local-reduction}

\begin{proof}[Proof of \Cref{lem:boundary-convergence}]
Put
\[
  q_n=\pc(r_n),
  \qquad
  p_0=\pc(r_0).
\]
By continuity of $\pc$ at $r_0$, we have $q_n\to p_0$, and
$|p_n-q_n|\to0$ gives $p_n\to p_0$ as well.
Choose a compact interval $P\subset(0,1)$ with $p_0$ in its interior.
After discarding finitely many terms, both $p_n$ and $q_n$ lie in $P$.
The role of $P$ is only to keep the logarithms in $I_p$ uniformly away from the
singular endpoints $p=0,1$.

Since $(\widetilde{\W}_0,\delta_\cut)$ is compact, every subsequence of
$(W_n)$ has a further cut-convergent subsequence; so it suffices to show that
every cut-convergent subsequence has limit $r_0$.
Passing to such a subsequence and relabeling it, we assume that $W_n\to W_*$ in cut distance.
We will prove that $W_*\equiv r_0$.
First note that the homomorphism densities
are continuous in the cut metric, so
\begin{equation}
  \label{eq:boundary-limit-feasibility}
  t(H,W_*)\ge r_0^m.
\end{equation}

We next compare the relative-entropy functionals with nearby values of the
parameter $p$.  For $p,q\in P$, the identity
\[
  I_p(W)
  =
  -2s(W)-\log(1-p)+e(W)\log\frac{1-p}{p}
\]
gives
\[
  I_p(W)-I_q(W)
  =
  \log\frac{1-q}{1-p}
  +e(W)\log\frac{(1-p)q}{p(1-q)}.
\]
Let
\[
  A(p')=-\log(1-p'),
  \qquad
  B(p')=\log\frac{1-p'}{p'}.
\]
Then, the preceding identity can be written as
\[
  I_p(W)-I_q(W)=A(p)-A(q)+e(W)\bigl(B(p)-B(q)\bigr).
\]
Note that as $P$ is compact, both $A$ and $B$
are uniformly continuous on $P$; the singularities at $0$ and $1$ are now a
positive distance away.
Since every graphon satisfies
$0\le e(W)\le1$, we have, for all $p,q\in P$,
\[
  |I_p(W)-I_q(W)| = |A(p)-A(q)+e(W)(B(p)-B(q))|
  \le |A(p)-A(q)| + |B(p)-B(q)|.
\]
Thus, by the uniform continuity of $A$ and $B$ on $P$,
\begin{equation}\label{eq:entropy-parameter-continuity}
\sup_{W\in\W_0}|I_p(W)-I_q(W)|\to0
\qquad (|p-q|\to0,\ p,q\in P).
\end{equation}

For fixed $p_0$, lower semicontinuity of $I_{p_0}$ in the cut metric gives
$I_{p_0}(W_*)\le\liminf_n I_{p_0}(W_n)$.  Since
\eqref{eq:entropy-parameter-continuity} and $q_n\to p_0$ give
$|I_{p_0}(W_n)-I_{q_n}(W_n)|\to0$, we obtain
\begin{equation}\label{eq:entropy-moving-parameter-lsc}
  I_{p_0}(W_*)
  \le
  \liminf_{n\to\infty} I_{q_n}(W_n).
\end{equation}

On the other hand, the constant graphon $W\equiv r_n$ is feasible for
$(p_n,r_n)$.  Since $W_n$ is an optimizer,
\[
  I_{p_n}(W_n)\le J_{p_n}(r_n).
\]
Using \eqref{eq:entropy-parameter-continuity} again, now with $p=p_n$ and $q=q_n$,
and noting that $A'$ and $B'$ are bounded on $P$, we obtain
\[
  I_{q_n}(W_n)
  \le
  J_{p_n}(r_n)+O_P(|p_n-q_n|).
\]
Applying the same bound to the constant graphon $r_n$, and using
$p_n-q_n\to0$, $q_n\to p_0$, $r_n\to r_0$, and continuity of
$(p,z)\mapsto J_p(z)$, gives
\[
  J_{p_n}(r_n)=J_{q_n}(r_n)+O_P(|p_n-q_n|)\to J_{p_0}(r_0).
\]
Therefore
\begin{equation}\label{eq:boundary-cost-limsup}
  \limsup_{n\to\infty} I_{q_n}(W_n)
  \le
  J_{p_0}(r_0).
\end{equation}

Equations \eqref{eq:boundary-limit-feasibility}, \eqref{eq:entropy-moving-parameter-lsc}, and \eqref{eq:boundary-cost-limsup} show
that $W_*$ is feasible at $(p_0,r_0)$ and has $I_{p_0}$-value no larger than the
constant graphon:
\[
  I_{p_0}(W_*)\le J_{p_0}(r_0).
\]
By condition (M2) of \Cref{thm:scalar-lz-boundary},
$(r_0^d,J_{p_0}(r_0))$ lies on the convex minorant of
$x\mapsto J_{p_0}(x^{1/d})$.  Hence
Theorem~\ref{thm:lz-criterion} gives that the constant graphon $r_0$
is the unique minimizer at $(p_0,r_0)$.  Thus,  $W_*\equiv r_0$.  Since every subsequence of $(W_n)$ has a further
cut-convergent subsequence, and every such limit equals $r_0$, the whole
sequence converges to $r_0$ in the quotient cut metric.  Finally, edge
density is continuous in the cut metric, so $e(W_n)\to r_0$.
\end{proof}

\begin{proof}[Proof of \Cref{cor:scalar-reduction}]
Let $I$ and $\rho_0$ be as in \Cref{lem:fixed-density-bipodality}.
Fix $0<\rho<\rho_0$, choose $\eta>0$ as in that lemma, and take
$r\in I$, $p\in(\pc(r)-\eta,\pc(r))$, and an optimizer $W^*$ attaining
$\Phi_H(p,r)$.  By \Cref{lem:active-constraint}, $t(H,W^*)=r^m$, and by
\Cref{lem:fixed-density-bipodality},
$\eps^*=e(W^*)\in(r-\rho,r)$.  Thus, $W^*$ is admissible in the variational
problem defining $S_H(\eps^*,r^m)$, and hence
\[
  s(W^*)\le S_H(\eps^*,r^m).
\]
Hence, by \eqref{eq:relative-entropy-identity},
\[
  I_{p,r}(\eps^*)\le I_p(W^*)=\Phi_H(p,r).
\]
On the other hand, $B_{\eps^*,r^m}$ has edge density $\eps^*$ and
$H$-density $r^m$, so it is feasible for the variational problem defining
$\Phi_H(p,r)$ in \eqref{eq:graphon-variational-problem}.  Using
\eqref{eq:reduced-objective-attainment}, we get
\[
  \Phi_H(p,r)\le I_p(B_{\eps^*,r^m})=I_{p,r}(\eps^*).
\]
Thus, the preceding inequalities are all equalities, and in particular,
\[
  I_{p,r}(\eps^*)=\Phi_H(p,r),
\]
and
\[
  s(W^*)=S_H(\eps^*,r^m).
\]
To see this, suppose that the entropy inequality
$s(W^*)\le S_H(\eps^*,r^m)$ were strict.  Then the inequality
$I_{p,r}(\eps^*)\le I_p(W^*)$ would also be strict.  Therefore,
$W^*$ is an entropy maximizer under the fixed constraints
$e(W)=\eps^*$ and $t(H,W)=r^m$.  By the uniqueness part of
\Cref{thm:krrs-analytic-extension}, $W^*$ agrees with $B_{\eps^*,r^m}$ up to relabeling.

For any $\eps\in(r-\rho,r)$, the graphon $B_{\eps,r^m}$ is feasible for the
variational problem defining $\Phi_H(p,r)$ in \eqref{eq:graphon-variational-problem}, so
\[
  I_{p,r}(\eps)=I_p(B_{\eps,r^m})\ge\Phi_H(p,r)
  =I_{p,r}(\eps^*).
\]
Therefore, $\eps^*$ minimizes $I_{p,r}$ on $(r-\rho,r)$, and
\eqref{eq:edge-density-minimization} follows.  The same inequality gives the converse
statement: if $\eps$ minimizes $I_{p,r}$ on $(r-\rho,r)$, then
$I_p(B_{\eps,r^m})=I_{p,r}(\eps)=I_{p,r}(\eps^*)
=\Phi_H(p,r)$, and $B_{\eps,r^m}$ is feasible, so it is an optimizer
attaining $\Phi_H(p,r)$.
Finally, distinct $\eps$ give graphons of distinct edge density, hence
inequivalent under relabeling, which gives the asserted bijection.
\end{proof}

\subsection{Parameter asymptotics near nonexceptional endpoints}
\label{app:bipodal-parameter-expansions}

\begin{remark}[First-order asymptotics of the bipodal parameters]
\label{rmk:bipodal-parameter-expansions}
\Cref{thm:nonexceptional-endpoint} gives the qualitative limit
\(c(p,r)\to0\).  The purpose of this remark is to quantify that limit and the
corresponding change in the block densities.  Work in the neighborhood $U$
of that theorem, shrinking it if necessary, and let \(p\uparrow\pc(r)\)
from the symmetry-breaking side.  Relabel the podes as in \Cref{thm:nonexceptional-endpoint}, so
that the first pode has size \(c=c(p,r)\), and set
\[
  \delta_*:=r-e(W_{p,r}),
  \qquad
  z:=\zeta_d(r).
\]
Denote by \(q_{11}^0(r)\) the value at
\((\eps,\vartheta)=(r,0)\) of the analytically extended KRR--S parameter
\(q_{11}\) from \Cref{thm:krrs-analytic-extension}.
All limits below refer to this regime, and the $O_{H,U}$-bounds are uniform
for $(p,r)\in U$ with $p<\pc(r)$.

For \(z\ne r\), define
\begin{equation}\label{eq:edge-deficit-coefficient}
  D_d(r,z):=\frac{z^d-r^d}{d r^{d-1}}-(z-r)>0,
\end{equation}
where positivity follows from the strict convexity of \(z\mapsto z^d\).
Since $\zeta_d(r_0)\ne r_0$, continuity lets us shrink $U$ so that
$D_d(r,\zeta_d(r))$ is uniformly bounded away from zero.
Then
\[
  \delta_*=2D_d(r,z)c+O_{H,U}(\delta_*^2).
\]
Thus, \(2D_d(r,z)\) is the first-order edge-density deficit per unit of
smaller-pode mass, after \(q_{22}\) is adjusted to preserve the
\(H\)-density.  Combining this relation with
\(\delta_*=\lambda(p,r)/A_H(r)+O_{H,U}(\lambda(p,r)^2)\) from
\Cref{thm:nonexceptional-endpoint} gives
\begin{equation}\label{eq:optimizer-block-size}
  c(p,r)=
  \frac{\lambda(p,r)}
  {2D_d(r,\zeta_d(r))A_H(r)}
  +O_{H,U}(\lambda(p,r)^2),
\end{equation}
\begin{equation}\label{eq:optimizer-block-density}
  q_{22}(p,r)
  =
  r-
  \frac{\zeta_d(r)^d-r^d}
       {d r^{d-1}D_d(r,\zeta_d(r))A_H(r)}
  \lambda(p,r)
  +O_{H,U}(\lambda(p,r)^2),
\end{equation}
and
\[
  q_{12}(p,r)=\zeta_d(r)+O_{H,U}(\lambda(p,r)),
  \qquad
  q_{11}(p,r)=q_{11}^0(r)+O_{H,U}(\lambda(p,r)).
\]
In particular, both the smaller-pode size and the displacement of the
large-pode density from \(r\) are of order \(\lambda(p,r)\).
The other two block densities satisfy
\(q_{12}\to\zeta_d(r)\ne r\) and
\(q_{11}\to q_{11}^0(r)\).  These relations quantify the vanishing-pode
mechanism described in the introduction.

Although they arise from different formulations, the KRR--S limiting cross
density and the second contact density of the Lubetzky--Zhao supporting line
coincide: \(\zeta_d(r)=\sm(r)\) for every \(r\ne r_*\).
Indeed, put \(\beta:=J_{\pc(r)}'(r)/(d r^{d-1})\).
Since \(J_{\pc(r)}+2S_0\) is affine,
\eqref{eq:supporting-cost-gap} and \eqref{eq:edge-deficit-coefficient} give
\[
  g_r(z)=-d r^{d-1}D_d(r,z)\bigl\{\psi_d(r,z)+\beta\bigr\}
  \qquad(z\in(0,1)\setminus\{r\}).
\]
The positivity of \(D_d(r,z)\) and \(g_r\ge0\) imply
\(\psi_d(r,z)\le-\beta\), also at \(z=r\) by continuity.
Equality holds at \(z=\sm(r)\) by \Cref{thm:scalar-lz-boundary}, so
the uniqueness of the maximizer in \Cref{thm:krrs-cross-density} proves the
identity.  Thus, the comparison graphons in \Cref{sec:nonexceptional-quadratic-growth} use
exactly the limiting cross density of the optimizer.

For completeness, we derive the two explicit expansions above.  Suppress the
arguments \((p,r)\) and evaluate the analytic KRR--S parameter map at
\[
  (\eps,\vartheta)
  =\bigl(r-\delta_*,\,r^m-(r-\delta_*)^m\bigr).
\]
After shrinking $U$, the segments from $(r,0)$ to these parameter pairs
lie in a compact subset of the KRR--S analytic domain, as in
\Cref{sec:nonexceptional-proof}.  Uniform Taylor estimates, the boundary
values in \Cref{thm:krrs-analytic-extension}, and
\(\vartheta=O_{H,U}(\delta_*)\) therefore give
\[
  c=O_{H,U}(\delta_*),\qquad
  q_{22}-r=O_{H,U}(\delta_*),\qquad
  q_{12}-z=O_{H,U}(\delta_*),\qquad
  q_{11}-q_{11}^0(r)=O_{H,U}(\delta_*).
\]
The exact edge-density identity
\[
  e(W_{p,r})
  =q_{22}+2c(q_{12}-q_{22})
   +c^2(q_{11}-2q_{12}+q_{22})
\]
therefore yields
\begin{equation}\label{eq:bipodal-edge-linearization}
  -\delta_*=(q_{22}-r)+2(z-r)c+O_{H,U}(\delta_*^2).
\end{equation}

For the \(H\)-density, separate the homomorphisms according to the number of
vertices mapped to the smaller pode.  Since \(H\) is \(d\)-regular and has
\(2m/d\) vertices, the contributions with zero and one such vertex give
\begin{align*}
  t(H,W_{p,r})
  &=(1-c)^{2m/d}q_{22}^m
    +\frac{2m}{d}c(1-c)^{2m/d-1}
      q_{12}^d q_{22}^{m-d}+O_{H,U}(c^2)\\
  &=r^m+m r^{m-1}(q_{22}-r)
    +\frac{2m}{d}r^{m-d}(z^d-r^d)c
    +O_{H,U}(\delta_*^2).
\end{align*}
Using \(t(H,W_{p,r})=r^m\), we obtain
\begin{equation}\label{eq:bipodal-density-linearization}
  0=m r^{m-1}(q_{22}-r)
    +\frac{2m}{d}r^{m-d}(z^d-r^d)c
    +O_{H,U}(\delta_*^2).
\end{equation}
Eliminating \(q_{22}-r\) between
\eqref{eq:bipodal-edge-linearization} and
\eqref{eq:bipodal-density-linearization} gives
\[
  \delta_*=2D_d(r,z)c+O_{H,U}(\delta_*^2),
  \qquad
  c=\frac{\delta_*}{2D_d(r,z)}+O_{H,U}(\delta_*^2),
\]
which proves \eqref{eq:optimizer-block-size}.  Finally,
\eqref{eq:bipodal-edge-linearization} gives
\[
  q_{22}
  =r-\frac{z^d-r^d}{d r^{d-1}D_d(r,z)}\delta_*
   +O_{H,U}(\delta_*^2),
\]
and \eqref{eq:optimizer-block-density} follows.

When $d=2$, a direct computation using $S_0(1-r)=S_0(r)$ and
$S_0'(1-r)=-S_0'(r)$ gives $\zeta_2(r)=1-r$, which yields the
triangle specialization of \eqref{eq:optimizer-block-size} used in the
introduction.
\end{remark}

%% file: sections/appendix_singular_endpoint.tex
\section{Technical proofs for the singular endpoint}
\label{app:endpoint-calculations}

This appendix contains the analytic and asymptotic calculations underlying
the rank-one KKT family constructed in
\Cref{sec:rank-one-stationary-family}.  We use the notation fixed there but
reproduce below all structural equations needed in the proofs.

\subsection{Proof of the analytic rank-one KKT family}
\label{app:rank-one-kkt-family}

\begin{proof}[Proof of \Cref{lem:rank-one-kkt-family}]
We first solve the graphon stationarity equation while leaving the block
proportion $\alpha_h$ unspecified.  For a rank-one graphon $W=f\otimes f$,
put $q:=\int_0^1 f^d$.  Since $H$ is $d$-regular, the first variations in a
bounded symmetric direction $U$ are
\[
  \left.\frac{d}{d\varepsilon}\right|_{\varepsilon=0}
  I_p(f\otimes f+\varepsilon U)
  =\iint J_p'(f(x)f(y))U(x,y)\,dx\,dy
\]
and
\[
  \left.\frac{d}{d\varepsilon}\right|_{\varepsilon=0}
  t(H,f\otimes f+\varepsilon U)
  =\iint m q^{v-2}f(x)^{d-1}f(y)^{d-1}U(x,y)\,dx\,dy.
\]
Thus, on writing $\gamma:=\mu m q^{v-2}$, stationarity for
$I_p-\mu t(H,\,\cdot\,)$ is equivalent to
\begin{equation}\label{eq:rank-one-kkt-proof}
  J_p'(f(x)f(y))
  =\mu m q^{v-2}\bigl(f(x)f(y)\bigr)^{d-1}
  =\gamma\bigl(f(x)f(y)\bigr)^{d-1}
  \quad\text{for a.e. }(x,y)\in[0,1]^2.
\end{equation}
For the family under construction, we first regard $\gamma_h$ as an
independent scalar coefficient, so the graphon stationarity equations do not
involve $q_h$ or $\alpha_h$.  Evaluating
\eqref{eq:rank-one-kkt-proof} on the two diagonal blocks and
the cross block, and using $J'_{p_h}(z)=\ell(p_h)-\ell(z)$, gives
\begin{equation}\label{eq:three-value-kkt-proof}
  \ell(p_h)-\ell(s_h^2)=\gamma_hs_h^{2d-2},\qquad
  \ell(p_h)-\ell(s_ht_h)=\gamma_h(s_ht_h)^{d-1},\qquad
  \ell(p_h)-\ell(t_h^2)=\gamma_ht_h^{2d-2}.
\end{equation}
We first solve \eqref{eq:three-value-kkt-proof} for
$u_h,\gamma_h,p_h$, and then impose stationarity under variation of the block
proportion to determine $\alpha_h$ and the derived quantities
$q_h,r_h,\mu_h$.
Recall that $s_h=u_h-h$ and $t_h=u_h+h$.  Subtracting consecutive equations
eliminates $\ell(p_h)$ and gives
\begin{equation}\label{eq:kkt-divided-slopes}
  \frac{\ell(s_ht_h)-\ell(s_h^2)}
       {s_h^{2d-2}-(s_ht_h)^{d-1}}
  =\frac{\ell(t_h^2)-\ell(s_ht_h)}
       {(s_ht_h)^{d-1}-t_h^{2d-2}}
  =\gamma_h
\end{equation}
for $h\ne0$.
At $h=0$, both quotients have the
indeterminate form $0/0$; the function $D$ introduced below shows that each
has a removable singularity and extends analytically to $h=0$.
Define the difference of the two quotients by
\begin{equation*}
\begin{aligned}
  A(u,h)
  &:=\frac{\ell(st)-\ell(s^2)}
            {s^{2d-2}-(st)^{d-1}}
    -\frac{\ell(t^2)-\ell(st)}
            {(st)^{d-1}-t^{2d-2}}
\end{aligned}
\end{equation*}
where $s=u-h$ and $t=u+h$.
Thus, the equation to solve is $A(u,h)=0$ with respect to $u$.
To analyze
the removable value at $h=0$, introduce
\[
  D(x,y):=\frac{\ell(y)-\ell(x)}{x^{d-1}-y^{d-1}}
\]
and extend it analytically across the diagonal.  The function $D(x,y)$ is
symmetric, and its diagonal value is
\[
  D(z,z)
  =\frac{-\ell'(z)}{(d-1)z^{d-2}}
  =\frac1{(d-1)z^{d-1}(1-z)}
  =:g_d(z).
\]
Differentiating $D(z,z)=g_d(z)$ and using symmetry gives
\[
  \partial_xD(z,z)=\partial_yD(z,z)=\frac12g_d'(z).
\]
Since
\[
  (s^2,st)\big|_{h=0}=(u^2,u^2),
  \qquad
  \left.\frac{d}{dh}(s^2,st)\right|_{h=0}=(-2u,0),
\]
we have
\[
  \left.D(s^2,st)\right|_{h=0}=g_d(u^2),
\]
and the chain rule gives
\[
  \left.\frac{d}{dh}D(s^2,st)\right|_{h=0}
  =\partial_xD(u^2,u^2)(-2u)+\partial_yD(u^2,u^2)\cdot0
  =-u g_d'(u^2).
\]
Thus, locally uniformly for $u$ near $u_*$, the Taylor expansion for the first
divided slope is
\[
  \frac{\ell(st)-\ell(s^2)}{s^{2d-2}-(st)^{d-1}}
  =g_d(u^2)-uh\,g_d'(u^2)+O_d(h^2).
\]
Replacing $h$ by $-h$ exchanges the two divided slopes.  Hence the second
Taylor expansion is
\[
  \frac{\ell(t^2)-\ell(st)}{(st)^{d-1}-t^{2d-2}}
  =g_d(u^2)+uh\,g_d'(u^2)+O_d(h^2).
\]
Subtracting the two expansions gives
\[
  A(u,h)=-2uh\,g_d'(u^2)+O_d(h^2).
\]
Interchanging $h$ and $-h$ exchanges $s$ and $t$, and hence changes the sign
of $A(u,h)$.  Thus, $A(u,h)$ is odd in $h$, so its expansion contains no
even powers of $h$.  Consequently, the last remainder improves to
\begin{equation*}
  A(u,h)=-2uh\,g_d'(u^2)+O_d(h^3).
\end{equation*}
We now observe that the quotient
\[
  B(u,h):=\frac{A(u,h)}{h}
\]
has a removable singularity at $h=0$.  Its analytic extension is even in
$h$ and, by the preceding expansion, satisfies
\[
  B(u,0)=-2u g_d'(u^2).
\]
For $h\ne0$, the equations $A(u,h)=0$ and $B(u,h)=0$ are equivalent.
Now fix a small interval $U\subset(0,1)$ containing $u_*$.  Since $u>0$ on
$U$, the explicit formula for $g_d$ gives
\[
  B(u,0)=0
  \quad\Longleftrightarrow\quad
  g_d'(u^2)=0
  \quad\Longleftrightarrow\quad
  \frac1{1-u^2}=\frac{d-1}{u^2}
  \quad\Longleftrightarrow\quad
  u=\sqrt{\frac{d-1}{d}}=u_*.
\]
Thus, $u_*$ is the unique zero of $B(u,0)$ in $U$.  This zero is
nondegenerate: since $g_d'(u_*^2)=0$ and $g_d''(u_*^2)\ne0$,
\[
  \partial_uB(u_*,0)
  =-2g_d'(u_*^2)-4u_*^2g_d''(u_*^2)
  =-4u_*^2g_d''(u_*^2)\ne0.
\]
The analytic implicit function theorem therefore gives a unique analytic
solution $u_h$ near $u_*$ satisfying $B(u_h,h)=0$ and $u_0=u_*$.  Since $B$
is even in $h$, uniqueness implies $u_{-h}=u_h$.  Hence $u_h$ is analytic
in $h^2$ and
\[
  u_h=u_*+O_d(h^2).
\]
Since $B(u_h,h)=0$, the first two expressions in
\eqref{eq:kkt-divided-slopes} agree.  We define $\gamma_h$ to be
their common value, using their analytic extensions at $h=0$.  The middle KKT
equation in \eqref{eq:three-value-kkt-proof} then gives
\[
  \ell(p_h)=\gamma_h(s_ht_h)^{d-1}+\ell(s_ht_h).
\]
Since $\ell$ is one-to-one, this identity determines a unique $p_h$.
Combining this definition of $\gamma_h$ with the middle equation recovers the
first and third equations in \eqref{eq:three-value-kkt-proof}, so all three
KKT equations are satisfied.
The analyticity of $D$ and the formulas above show that
$\gamma_h$ and $p_h$ are even and analytic in $h^2$.  At $h=0$,
the analytic extension of either quotient in
\eqref{eq:kkt-divided-slopes} gives
\[
  \gamma_0=g_d(u_*^2)=\gamma_*.
\]
The middle KKT equation in \eqref{eq:three-value-kkt-proof} then gives
\[
  \ell(p_0)=\gamma_*u_*^{2(d-1)}+\ell(u_*^2)=\ell_*,
\]
which implies $p_0=p_*$. 
We have therefore constructed the unique analytic family
$(u_h,p_h,\gamma_h)$ near $(u_*,p_*,\gamma_*)$.
Define
\[
  \mathcal M_h(z):=J_{p_h}(z)-\frac{\gamma_h}{d}z^d.
\]
At the singular endpoint, direct differentiation gives
\begin{equation}\label{eq:endpoint-entropy-derivatives}
  J_{p_*}'(r_*)=\gamma_*r_*^{d-1},\qquad
  J_{p_*}''(r_*)=\gamma_*(d-1)r_*^{d-2},\qquad
  J_{p_*}^{(3)}(r_*)=\gamma_*(d-1)(d-2)r_*^{d-3}.
\end{equation}
Equation \eqref{eq:three-value-kkt-proof} says precisely that
$s_h^2,s_ht_h,t_h^2$ are critical points of $\mathcal M_h$.
By \eqref{eq:endpoint-entropy-derivatives},
the first three derivatives of $\mathcal M_0$ vanish at $r_*$.
Direct differentiation gives its Taylor expansion
\begin{equation}\label{eq:mh-cubic-derivative}
  \mathcal M_0'(z)
  =k_3(z-r_*)^3+O_d((z-r_*)^4),
  \qquad
  k_3:=\frac{d^5}{6(d-1)^2}.
\end{equation}

We next determine the block proportion.  For fixed $p,s,t$, and $\mu$, its
stationarity condition at $\alpha$ is
\begin{equation}\label{eq:block-stationarity-proof}
  \left.\frac{d}{d\beta}\right|_{\beta=\alpha}
  \left\{I_p(W_{\beta,s,t})-\mu t(H,W_{\beta,s,t})\right\}=0.
\end{equation}
For a provisional $\alpha$ near $1/2$, apply this condition with
$p=p_h$, $s=s_h$, $t=t_h$, and set
\[
  q:=\alpha s_h^d+(1-\alpha)t_h^d,
  \qquad
  \mu:=\frac{\gamma_h}{m q^{v-2}}.
\]
Holding $\mu$ fixed, the expression differentiated in
\eqref{eq:block-stationarity-proof} is
\[
  \beta^2J_{p_h}(s_h^2)
  +2\beta(1-\beta)J_{p_h}(s_ht_h)
  +(1-\beta)^2J_{p_h}(t_h^2)
  -\mu\{\beta s_h^d+(1-\beta)t_h^d\}^v,
\]
where we used
\[
  t(H,W_{\beta,s_h,t_h})
  =\{\beta s_h^d+(1-\beta)t_h^d\}^v.
\]
Its derivative at $\beta=\alpha$ vanishes precisely when
\[
  0=2\alpha\{J_{p_h}(s_h^2)-J_{p_h}(s_ht_h)\}
  -2(1-\alpha)\{J_{p_h}(t_h^2)-J_{p_h}(s_ht_h)\}
  -\mu v q^{v-1}(s_h^d-t_h^d).
\]
Since $v=2m/d$, the definitions give
\[
  \frac{\mu v q^{v-1}}2
  =\frac{\gamma_h q}{d},
\]
while direct expansion gives
\begin{equation}\label{eq:block-moment-identity}
  q(s_h^d-t_h^d)
  =\alpha\{s_h^{2d}-(s_ht_h)^d\}
   -(1-\alpha)\{t_h^{2d}-(s_ht_h)^d\}.
\end{equation}
Using the preceding two identities, the block-proportion condition becomes
the linear equation
\begin{equation}\label{eq:block-proportion-balance}
  \alpha\{\mathcal M_h(s_h^2)-\mathcal M_h(s_ht_h)\}
  =(1-\alpha)\{\mathcal M_h(t_h^2)-\mathcal M_h(s_ht_h)\}.
\end{equation}
Whenever its denominator is nonzero, its solution is
\begin{equation}\label{eq:block-proportion-formula-proof}
  \alpha_h
  :=\frac{\left\{\mathcal M_h(t_h^2)-\mathcal M_h(s_ht_h)\right\}}
  {\left\{\mathcal M_h(s_h^2)-\mathcal M_h(s_ht_h)\right\}
   +\left\{\mathcal M_h(t_h^2)-\mathcal M_h(s_ht_h)\right\}}.
\end{equation}
Because $p_h,\gamma_h$, and $u_h$ are even in $h$, one has
\[
  \mathcal M_{-h}=\mathcal M_h,
  \qquad
  s_{-h}=t_h,
  \qquad
  t_{-h}=s_h.
\]
Thus, changing $h$ to $-h$ exchanges the two differences in
\eqref{eq:block-proportion-formula-proof}, and consequently
\[
  \alpha_{-h}=1-\alpha_h.
\]

It remains to verify the singular-endpoint value $\alpha_0$
by extending \eqref{eq:block-proportion-formula-proof} analytically through $h=0$.
To control the two differences in \eqref{eq:block-proportion-formula-proof} as
$h\to0$, we factor $\mathcal M_h'$ at its three critical points.  Let
\[
  D_h(z):=(z-s_h^2)(z-s_ht_h)(z-t_h^2)
\]
be the monic cubic polynomial with these roots.  Analytic Weierstrass division
with divisor $D_h$ gives
\[
  \mathcal M_h'(z)=D_h(z)R(h,z)+P_h(z),
  \qquad \deg_z P_h\le2,
\]
where $R$ is jointly analytic in $(h,z)$ near $(0,r_*)$, and the coefficients
of $P_h$ are analytic in $h$.  For $h\ne0$, the polynomial $P_h$ vanishes at
the three distinct
roots of $D_h$ and is therefore identically zero.  Its coefficients also
vanish at $h=0$ by analyticity.  Hence
\begin{equation}\label{eq:mh-factorization-proof}
  \mathcal M_h'(z)
  =(z-s_h^2)(z-s_ht_h)(z-t_h^2)R(h,z).
\end{equation}
By \eqref{eq:mh-cubic-derivative}, $R(0,r_*)=k_3$.  Writing
$z=s_ht_h+x$, set
\[
  a:=s_ht_h-s_h^2=2hs_h,
  \qquad
  b:=t_h^2-s_ht_h=2ht_h.
\]
The three critical points then correspond to $x=-a,0,b$.  Therefore,
uniformly for $|x|\le2|h|$,
\[
  \mathcal M_h'(s_ht_h+x)
  =x(x+a)(x-b)\{k_3+O_d(h)\}.
\]
Since $a=2u_*h+O_d(h^2)$ and $b=2u_*h+O_d(h^2)$, integration from $x=0$ to
$x=-a$ and to $x=b$ gives
\begin{align*}
  \mathcal M_h(s_h^2)-\mathcal M_h(s_ht_h)
  &=-\frac{k_3}{12}a^3(a+2b)+O_d(h^5)
    =-\frac{2d^3}{3}h^4+O_d(h^5),\\
  \mathcal M_h(t_h^2)-\mathcal M_h(s_ht_h)
  &=-\frac{k_3}{12}b^3(b+2a)+O_d(h^5)
    =-\frac{2d^3}{3}h^4+O_d(h^5).
\end{align*}
Both differences are analytic and vanish to order exactly four.  Dividing
the numerator and denominator in \eqref{eq:block-proportion-formula-proof} by $h^4$
therefore shows that $\alpha_h$ extends analytically through $h=0$, and
the limiting value is $\alpha_0=1/2$. 

Now define
\[
  q_h:=\alpha_hs_h^d+(1-\alpha_h)t_h^d,
  \qquad
  r_h:=q_h^{2/d},
  \qquad
  \mu_h:=\frac{\gamma_h}{m q_h^{v-2}}.
\]
The definitions and the analyticity proved above show that $q_h$ is analytic.
The symmetry identities give $q_{-h}=q_h$.  Since $q_0>0$, the derived
functions $r_h=q_h^{2/d}$ and $\mu_h=\gamma_h/(m q_h^{v-2})$ are also analytic
and even.  Evaluating at $h=0$ gives
\[
  q_0=u_*^d,
  \qquad
  r_0=r_*,
  \qquad
  \mu_0=\frac{\gamma_*}{m(u_*^d)^{v-2}}
  =\frac1{m r_*^m}.
\]
After decreasing $h_0$, continuity also ensures
\[
  0<p_h<r_h<1,\qquad 0<s_h,t_h<1,\qquad
  0<\alpha_h<1,\qquad \mu_h>0
  \quad (|h|<h_0),
\]
so all the candidates above are valid graphons at admissible parameter
values.
The rank-one identity also gives
\[
  t(H,W_{\alpha_h,s_h,t_h})=q_h^v=r_h^m.
\]
To verify stationarity, let $U$ be any bounded measurable symmetric function
on $[0,1]^2$.  The first-variation formulas above give
\[
\begin{aligned}
  &\left.\frac{d}{d\varepsilon}\right|_{\varepsilon=0}
  \left[I_{p_h}(W_h+\varepsilon U)
    -\mu_h t(H,W_h+\varepsilon U)\right]\\
  &\qquad=\iint\left[
    J_{p_h}'(f_h(x)f_h(y))
    -\mu_h m q_h^{v-2}f_h(x)^{d-1}f_h(y)^{d-1}
  \right]U(x,y)\,dx\,dy.
\end{aligned}
\]
The product $f_h(x)f_h(y)$ takes only the three values
$s_h^2,s_ht_h,t_h^2$.  For each such value $z$, the corresponding equation in
\eqref{eq:three-value-kkt-proof} gives
\[
  J_{p_h}'(z)=\gamma_hz^{d-1}
  =\mu_h m q_h^{v-2}z^{d-1}.
\]
Hence the integrand in the first-variation formula vanishes almost everywhere,
so $W_h$ is stationary under every such $U$.  Thus
$W_h$ satisfies the KKT conditions at $(p_h,r_h)$ with multiplier $\mu_h$, and
\eqref{eq:block-proportion-balance} is precisely
\eqref{eq:block-stationarity-proof}.

It remains to prove local uniqueness.
Suppose $W$ is a nonconstant rank-one graphon satisfying the KKT conditions at
$(p,r)$ with multiplier $\mu$.
By \Cref{lem:stationary-rank-one-bipodality}, write
\[
  W = W_{\alpha,s,t} = f_{\alpha,s,t}\otimes f_{\alpha,s,t},
  \qquad
  h:=\frac{t-s}{2},
  \qquad
  u:=\frac{s+t}{2},
  \qquad
  q:=\alpha s^d+(1-\alpha)t^d,
  \qquad
  \gamma:=\mu m q^{v-2}.
\]
Applying \eqref{eq:rank-one-kkt-proof} to the three values
$s^2,st,t^2$ gives \eqref{eq:three-value-kkt-proof} with
$(p_h,\gamma_h,s_h,t_h)$ replaced by $(p,\gamma,s,t)$.  Equality of the two
corresponding divided slopes then gives $B(u,h)=0$.  The local uniqueness
obtained above gives a neighborhood
of $(u_*,0)$ in which this equation has the unique solution $u=u_h$ for each
fixed $h$.  Hence $s=s_h$ and $t=t_h$, and either quotient in
\eqref{eq:kkt-divided-slopes} forces $\gamma=\gamma_h$.  The
middle KKT equation then forces $p=p_h$.  Finally,
condition \eqref{eq:block-stationarity-proof} is the linear equation
\eqref{eq:block-proportion-balance}; its denominator is
$-4d^3h^4/3+O_d(h^5)$ and is therefore nonzero for $h\ne0$ sufficiently small.
Thus, it has the unique solution \eqref{eq:block-proportion-formula-proof}, namely
$\alpha=\alpha_h$.  It follows that $q=q_h$ and hence that $r=r_h$ and
$\mu=\gamma_h/(m q_h^{v-2})=\mu_h$.  This proves the stated local uniqueness.
\end{proof}

\subsection{Proof of the parameter expansions}
\label{app:rank-one-parameter-expansions}

\begin{proof}[Proof of \Cref{lem:rank-one-parameter-expansions}]
Since $u_h$, $\gamma_h$, and $\ell(p_h)$ are analytic in $h^2$, write
\[
  u_h=u_*+U_2h^2+O_d(h^4),
  \qquad
  \gamma_h=\gamma_*+G_2h^2+O_d(h^4),
  \qquad
  \ell(p_h)=\ell_*+L_2h^2+O_d(h^4).
\]
Define
\[
  \mathcal L_h(z):=J_{p_h}'(z)-\gamma_hz^{d-1}.
\]
The triple-contact identities
\eqref{eq:endpoint-entropy-derivatives} show that $\mathcal L_0$ and
its first two derivatives vanish at $r_*$.  Direct differentiation gives
\[
  \mathcal L_0^{(3)}(r_*)=\frac{d^5}{(d-1)^2},
  \qquad
  \mathcal L_0^{(4)}(r_*)=\frac{5d^6(d-2)}{(d-1)^3},
\]
and hence
\begin{align}\label{eq:kkt-quartic-expansion}
  \mathcal L_0(z)
  &=\frac{d^5}{6(d-1)^2}(z-r_*)^3
    +\frac{5d^6(d-2)}{24(d-1)^3}(z-r_*)^4
    +O_d((z-r_*)^5).
\end{align}
Moreover,
\begin{align}\label{eq:kkt-parameter-shift}
  \mathcal L_h(z)
  &=\mathcal L_0(z)+\{\ell(p_h)-\ell_*\}
    -(\gamma_h-\gamma_*)z^{d-1}.
\end{align}

Since $r_*=u_*^2$, the definitions $s_h=u_h-h$ and $t_h=u_h+h$ give
\begin{equation*}
\begin{aligned}
  s_h^2&=r_*-2u_*h+(1+2u_*U_2)h^2+O_d(h^3),\\
  s_ht_h&=r_*+(2u_*U_2-1)h^2+O_d(h^4),\\
  t_h^2&=r_*+2u_*h+(1+2u_*U_2)h^2+O_d(h^3).
\end{aligned}
\end{equation*}
We substitute these three roots into \eqref{eq:kkt-parameter-shift} and
use the equations $\mathcal L_h(s_h^2)=\mathcal L_h(s_ht_h)
=\mathcal L_h(t_h^2)=0$.  For the middle root,
$s_ht_h-r_*=O_d(h^2)$, so comparison at order $h^2$ gives
\[
  L_2=G_2r_*^{d-1}.
\]
For either outer root, the terms of order $h^3$ then give
\[
  \frac{4}{3}u_*^3\frac{d^5}{(d-1)^2}
  -2u_*(d-1)r_*^{d-2}G_2=0.
\]
Therefore
\[
  G_2=\frac{2d^{d+2}}{3(d-1)^d},
  \qquad
  L_2=\frac{2d^3}{3(d-1)}.
\]
At order $h^4$, subtracting the middle-root equation from either outer-root
equation cancels the unknown fourth-order coefficients of $\ell(p_h)$ and
$\gamma_h$ and gives
\[
  \frac{d^5}{(d-1)^2}u_*^2(1+2u_*U_2)
  +\frac53\frac{d^6(d-2)}{(d-1)^3}u_*^4
  =G_2(d-1)r_*^{d-3}\{(d-2)u_*^2+r_*\}.
\]
Using $u_*^2=r_*=(d-1)/d$ and the value of $G_2$, we obtain
\[
  U_2=\frac{5-3d}{6u_*}.
\]
This proves the stated expansions of $u_h$ and $\ell(p_h)$; the calculation
also gives the intermediate expansion of $\gamma_h$ stated above.

We next compute the first nonzero term of $\alpha_h$.  Set
\[
  a:=s_ht_h-s_h^2=2hs_h,
  \qquad
  b:=t_h^2-s_ht_h=2ht_h.
\]
The factorization \eqref{eq:mh-factorization-proof} and the symmetries
$\mathcal M_{-h}=\mathcal M_h$ and $D_{-h}=D_h$ show that $R$ is even in
$h$.  At $h=0$, \eqref{eq:kkt-quartic-expansion} gives
\[
  R(0,z)=k_3+k_4(z-r_*)+O_d((z-r_*)^2),
  \qquad
  k_4:=\frac{5d^6(d-2)}{24(d-1)^3}.
\]
Consequently, uniformly for $|x|\le2|h|$,
\[
  \mathcal M_h'(s_ht_h+x)
  =x(x+a)(x-b)\{k_3+k_4x+O_d(h^2)\}.
\]
Integrating from $x=0$ to $x=-a$ and to $x=b$ gives
\begin{align*}
  \mathcal M_h(s_h^2)-\mathcal M_h(s_ht_h)
  &=-\frac{k_3}{12}a^3(a+2b)
    +\frac{k_4}{60}a^4(3a+5b)+O_d(h^6),\\
  \mathcal M_h(t_h^2)-\mathcal M_h(s_ht_h)
  &=-\frac{k_3}{12}b^3(b+2a)
    -\frac{k_4}{60}b^4(3b+5a)+O_d(h^6).
\end{align*}
By the expansion of $u_h$,
\[
  a=2u_*h-2h^2+O_d(h^3),
  \qquad
  b=2u_*h+2h^2+O_d(h^3).
\]
Substitution yields
\begin{align*}
  \mathcal M_h(s_h^2)-\mathcal M_h(s_ht_h)
  &=-\frac{2d^3}{3}h^4
    +\frac{8u_*d^5}{9(d-1)}h^5+O_d(h^6),\\
  \mathcal M_h(t_h^2)-\mathcal M_h(s_ht_h)
  &=-\frac{2d^3}{3}h^4
    -\frac{8u_*d^5}{9(d-1)}h^5+O_d(h^6).
\end{align*}
Substituting these expressions into \eqref{eq:block-proportion-formula-proof} and using
$\alpha_{-h}=1-\alpha_h$ gives
\[
  \alpha_h=\frac12+\frac{2d^2u_*}{3(d-1)}h+O_d(h^3).
\]

Finally, substituting the expansions of $u_h$ and $\alpha_h$ into
\[
  q_h=\alpha_h(u_h-h)^d+(1-\alpha_h)(u_h+h)^d
\]
and expanding by the binomial theorem gives
\[
  q_h=u_*^d-\frac{d(4d-1)}3u_*^{d-2}h^2+O_d(h^4).
\]
The stated expansion of $r_h=q_h^{2/d}$ now follows by Taylor expansion.
\end{proof}